\documentclass[10pt, twoside, a4paper]{amsart}

\usepackage{amsthm,amsmath,amsfonts,amssymb,mathtools}
\usepackage[authoryear]{natbib}
\RequirePackage[colorlinks,citecolor=blue,urlcolor=blue]{hyperref}

\usepackage[foot]{amsaddr}
\usepackage{enumitem}
\usepackage{latexsym}
\usepackage[capitalise, noabbrev]{cleveref}
\usepackage[linesnumbered,ruled]{algorithm2e}

\DeclareMathAlphabet{\mathpzc}{OT1}{pzc}{m}{it}

\usepackage{bm}
\newcommand{\blambda}{\bm{\lambda}}
\newcommand{\bLambda}{\bm{\Lambda}}
\newcommand{\bM}{\mathbf{M}}
\newcommand{\cF}{\mathcal{F}}
\newcommand{\cG}{\mathcal{G}}
\newcommand{\pa}[1]{\left(#1\right)}
\newcommand{\convUD}{\xrightarrow{\scalebox{0.6}{$\mathcal{D}/\Theta$}}}
\newcommand{\convUDnull}{\xrightarrow{\scalebox{0.6}{$\mathcal{D}/\Theta_0$}}}
\newcommand{\convUP}{\xrightarrow{\scalebox{0.6}{$P/\Theta$}}}
\newcommand{\convUPnull}{\xrightarrow{\scalebox{0.6}{$P/\Theta_0$}}}

\newcommand{\ex}{\mathbb{E}}
\newcommand{\var}{\mathrm{Var}}

\usepackage{mathbbol}
\DeclareSymbolFontAlphabet{\mathbb}{AMSb}
\DeclareSymbolFontAlphabet{\mathbbl}{bbold}
\newcommand{\one}{\mathbbl{1}}
\newcommand{\VVV}{W\hspace{-0.7em}W}

\DeclareMathOperator*{\cov}{cov}
\renewcommand{\epsilon}{\varepsilon}
\newcommand{\cl}{c\`adl\`ag}
\newcommand{\lc}{c\`agl\`ad}

\newtheorem{thm}{Theorem}[section]
\newtheorem{prop}[thm]{Proposition}
\newtheorem{lem}[thm]{Lemma}
\newtheorem{cor}[thm]{Corollary}

\theoremstyle{definition}
\newtheorem{dfn}[thm]{Definition}

\newtheorem{assump}{Assumption}[section]
\crefname{assump}{Assumption}{Assumptions}

\usepackage{tikz}
\usepgflibrary{arrows}
\usetikzlibrary{calc,fadings,decorations.pathreplacing,intersections,positioning}
\usetikzlibrary{decorations.markings}

\newcommand{\vertiii}[1]{{\left\vert\kern-0.25ex\left\vert\kern-0.25ex\left\vert #1 
    \right\vert\kern-0.25ex\right\vert\kern-0.25ex\right\vert}}
    
\title[Nonparametric CLI testing]{Nonparametric conditional local \\independence testing}

\author[A. M. Christgau]{Alexander Mangulad Christgau}
\email{amc@math.ku.dk}
\author[L. Petersen]{Lasse Petersen}
\email{lassepetersen@protonmail.com}
\author[N. R. Hansen]{Niels Richard Hansen}
\email{Niels.R.Hansen@math.ku.dk}

\address{Department of Mathematical Sciences, University
of Copenhagen, Universitetsparken 5, 2100 Copenhagen \O, Denmark
}

\begin{document}

    \begin{abstract}
        Conditional local independence is an asymmetric independence relation among continuous time stochastic processes. It describes whether the evolution of one process is directly influenced by another process given the histories of additional processes, and it is important for the description and learning of causal
        relations among processes. We formulate a model-free framework for 
        testing the hypothesis that a counting process is conditionally locally independent 
        of another process. To this end, we introduce a new functional parameter called the Local Covariance Measure (LCM), which quantifies deviations from the hypothesis. 
        Following the principles of double machine learning, we propose an estimator 
        of the LCM and a test of the hypothesis using nonparametric estimators and sample splitting or cross-fitting.
        We call this test the (cross-fitted) Local Covariance Test ((X)-LCT), and we show that its level and power can be controlled uniformly, provided that the nonparametric estimators are consistent with modest rates. We illustrate the 
        theory by an example based on a marginalized Cox model with
        time-dependent covariates, and we show in simulations that when double 
        machine learning is used in combination with cross-fitting, then 
        the test works well without restrictive parametric assumptions.    
    \end{abstract}

\maketitle

\section{Introduction}
Notions of how one variable influences a target variable are central to 
both predictive and causal modeling. Depending on the objective, the relevant 
notion of influence can be variable importance in a predictive model of the target, 
but it can also be the causal effect of the variable on the target. 
In either case, we can investigate influence conditionally on a third 
variable -- to quantify the added predictive value, the direct 
causal effect or the causal effect adjusted for a confounder. 
Our interests are in an asymmetric notion of direct influence among stochastic 
processes, which is not adequately captured by classical (symmetric)
notions of conditional dependence. The objective of this paper is 
therefore to quantify this notion of asymmetric influence 
and specifically to develop a new nonparametric test of the hypothesis
that one stochastic process does not directly influence another.  

The hypothesis we consider is formalized as the hypothesis of \emph{conditional
local independent} -- a concept introduced by \cite{schweder1970} as a continuous time
formalization of the phenomenon that the past of one stochastic process does not 
directly influence the evolution of another stochastic process.  
As such, conditional local independence is a continuous time version of 
the discrete time concept of Granger non-causality \citep{Granger:1969}. 

To illustrate the concept of conditional local independence we will in 
this introduction consider an example involving three processes: $X$, $Z$ and $N$ -- 
see Figure \ref{fig:lig0}. 
The process $N$ is the indicator of death, $N_t = \one(T \leq t)$, for an 
individual with survival time $T$, and $X_t$ denotes the total pension 
savings of the individual at time $t$. The process $Z$ is a covariate 
process, e.g., health variables or employment status, that may directly affect 
both the pension savings and the survival time. This is indicated in Figure \ref{fig:lig0}
by edges pointing from $Z$ to $X$ and $N$. Edges pointing from 
$N$ to $X$ and $Z$ indicate that a death event directly affects both $X$ and $Z$ (which 
take the values $X_T$ and $Z_T$, respectively, after time $T$, see Section \ref{sec:introex}).

To define conditional local independence let 
$\mathcal{F}^{N,Z}_t = \sigma(N_s, Z_s; s \leq t)$ denote the filtration generated by 
the $N$- and $Z$-processes. The $\sigma$-algebra $\mathcal{F}_t^{N,Z}$ represents 
the information contained in the $N$- and the $Z$- processes before
time $t$. Informally, the process $N_t$ is conditionally locally independent of
the process $X_t$ given $\mathcal{F}_t^{N,Z}$ if $(X_s)_{s \leq t}$ does not add 
predictable information to $\mathcal{F}_{t-}^{N,Z}$ about the infinitesimal evolution of $N_t$.
For this particular example this means that the conditional hazard function 
of $T$ does not depend on $(X_s)_{s \leq t}$ given $\mathcal{F}_t^{N,Z}$.
In Figure \ref{fig:lig0} the hypothesis of interest, that $N_t$ is conditionally 
locally independent of $X_t$ given $\mathcal{F}^{N,Z}_t$, is represented by 
the lack of an edge from $X$ to $N$.

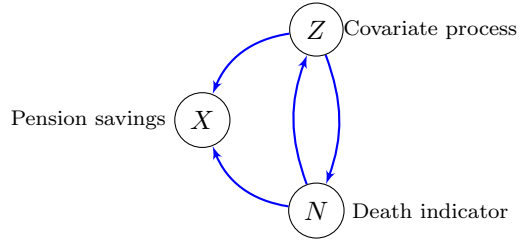
\begin{figure}
\begin{tikzpicture}[scale=0.6]
\tikzset{vertex/.style = {shape=circle,draw,minimum size=2em}}
\tikzset{vertexHid/.style = {shape=rectangle,draw,minimum size=2em}}
\tikzset{edge/.style = {->,> = latex', thick}}
\node[vertex] (Z) at  (1.5,2) {$Z$};
\node at (4,2) {\footnotesize{Covariate process}};
\node[vertex] (X) at  (-1,0) {$X$};
\node at (-3.5,0) {\footnotesize{Pension savings}};
\node[vertex] (T) at  (1.5,-2) {$N$};
\node at (4,-2) {\footnotesize{Death indicator}};

\draw[edge, bend left=20, color = blue] (Z) to (T);
\draw[edge, bend left, color = blue] (T) to (X);
\draw[edge, bend left=20, color = blue] (T) to (Z);
\draw[edge, bend right, color = blue] (Z) to (X);
\end{tikzpicture} 

\caption{Local independence graph illustrating a dependence 
structure among the three processes $X$, $Z$ and $N$. Here 
$N$ is the indicator of death for an 
individual, $X$ is their cumulative pension savings and $Z$ is a covariate process.
All nodes in this graph have implicit self-loops.
There is no edge from $X$ to $N$, which indicates that death is not directly influenced by 
pension savings. This can be formalized as $N$ being conditionally locally independent 
of $X$, which is the hypothesis we aim to test.
\label{fig:lig0}}
\end{figure}

A systematic investigation of algebraic properties of 
conditional local independence was initiated by 
\cite{didelez2006, Didelez2008, didelez2015}. She also introduced local independence 
graphs, such as the directed graph in Figure \ref{fig:lig0}, 
to graphically represent all conditional local independencies 
among several processes, and she studied the semantics of these graphs. 
This work was extended further by \cite{Mogensen:2020} to 
graphical representations of partially observed systems.
While we will not formally discuss local independence graphs, the problem 
of learning such graphs from data was an important motivation for us to 
develop a nonparametric test of conditional local
independence. A
constraint based learning algorithm of local independence graphs was given by
\cite{Mogensen:2018} in terms of a conditional local independence oracle,
but a practical algorithm requires that the oracle is replaced by 
conditional local independence tests.  

Another important motivation for considering conditional local independence 
arises from causal models.  With a
structural assumption about the stochastic process specification, a conditional
local independence has a causal interpretation \citep{aalen1987, aalen2012,
commenges2009}, and if the causal stochastic system is completely observed, a
test of conditional local independence is a test of no \emph{direct} causal
effect. If the causal stochastic system is only partially observed, a
conditional local dependency need not correspond to a direct causal effect due
to unobserved confounding, but the projected local independence graph, as
introduced by \cite{Mogensen:2020}, retains a causal interpretation, and its
Markov equivalence class can be learned by conditional local independence
testing. In addition, within the framework of structural nested models, testing
the hypothesis of no \emph{total} causal effect can also be cast as a test of
conditional local independence \citep{lok2008}.

To appreciate what conditional local independence means -- and, in particular, 
what it does not mean -- it is useful to compare with classical conditional independence. 
In our example, $N_t$ is conditionally locally independent of $X_t$ given 
$\mathcal{F}^{N,Z}_t$, but this implies neither that
\mbox{$N \!\perp\!\!\!\!\perp X \mid Z$} (as processes), nor that 
$N_t \!\perp\!\!\!\!\perp X_t \mid \mathcal{F}_t^Z$. In fact, 
these conditional independencies cannot hold in this example where 
$X_t = X_T$ for $t \geq T$ -- except in special cases such as $T$ being 
a deterministic function of $Z$. Theorem 2 in \cite{Didelez2008}
gives a sufficient condition for $N_t  \!\perp\!\!\!\!\perp X_t \mid \mathcal{F}_t^Z$
to hold in terms of the local independence graph, but this condition 
is also not fulfilled by the graph in Figure \ref{fig:lig0} due to the edge from 
$N$ to $X$. 
\cite{Didelez2008} argues that $N_t$ being conditionally locally independent of $X_t$ given 
$\mathcal{F}^{N,Z}_t$ heuristically means that $N_t \!\perp\!\!\!\!\perp 
\mathcal{F}^X_{t-} \mid \mathcal{F}_{t-}^{N,Z}$, but this is technically problematic 
in continuous time. If $T$ has a continuous distribution, then for any fixed $t$, 
$N_t = N_{t-}$ almost surely, whence $N_t$ is 
almost surely $\mathcal{F}_{t-}^{N,Z}$-measurable and conditionally independent of 
anything given $\mathcal{F}_{t-}^{N,Z}$. It is thus not possible to use 
this heuristic to formally define conditional local independence in continuous time. 
See instead the formal Definition 2 by \cite{Didelez2008} or our Definition 
\ref{dfn:cli}.

Several examples from health sciences given by \cite{Didelez2008} demonstrate  
the usefulness of conditional local independence for multivariate event systems, and more recent 
attention to event systems in the machine 
learning community \citep{Zhou:2013, Xu:2016, Achab:2017, Bacry:2018, Cai:2022} 
testifies to the relevance of conditional local independence. This line of research 
relies primarily on the linear Hawkes process model, which is 
effectively used to infer local independence graphs  
-- sometimes even interpreted causally. 
The Hawkes model is attractive because conditional local independencies can be 
inferred from corresponding kernel functions being zero -- and statistical tests can readily be based on parametric or nonparametric
estimation of kernels. A less attractive property of the Hawkes model is 
that any such test cannot be expected to maintain level if the model is misspecified.
This is compounded by the Hawkes model not being closed under marginalization 
(also known as non-collapsibility), which means that even within a 
subsystem of a linear Hawkes process, conditional local independence cannot be tested correctly using a Hawkes model. 

The challenge of model misspecification and non-collapsibility is investigated 
further in Sections \ref{sec:introex} and \ref{sec:simulations}
based on an extension of our introductory 
example and Cox's survival model. Both the Hawkes model and the Cox model 
illustrate that conditional local independence might be expressed 
and tested within a (semi-)parametric model, but non-collapsibility -- and 
model misspecification, in general -- makes us question the validity of
a model based test. Thus there is a need for a nonparametric test of 
the hypothesis of conditional local independence. Moreover, since we cannot 
translate the hypothesis into an equivalent hypothesis about 
classical conditional independence, we cannot directly use existing 
nonparametric tests, such as the GHCM \citep{lundborg2021conditional}, 
of conditional independence. 

We propose a new nonparametric test when the target process $N$ is a counting process 
and $X$ is a real valued process, and where the hypothesis is that $N$ is 
conditionally locally independent of $X$ given a filtration $\mathcal{F}_t$. 
In the context of the introductory example, $\mathcal{F}_t = \mathcal{F}_t^{N,Z}$. We 
consider a counting process target primarily because the theory of conditional 
local independence is most complete in this case, but generalizations are possible -- we 
refer to the discussion in Section \ref{sec:dis}. Within our framework we base our
test on an infinite dimensional parameter, which we call the 
Local Covariance Measure (LCM). It is a function of time, 
which is constantly equal to zero under the hypothesis. 
Our main result is that the LCM can be estimated by using the ideas of double
machine learning \citep{Chernozhukov:2018} in such a way that the estimator 
converges uniformly at a $\sqrt{n}$-rate to a mean zero Gaussian martingale under the hypothesis of conditional local independence. We use the LCM to develop the (cross-fitted)
Local Covariance Test ((X)-LCT), for which we derive uniform level and power 
results.

\subsection{Organization of the paper}
In Section \ref{sec:setup} we introduce the general framework for formulating the hypothesis of
conditional local independence. This includes the introduction in 
Section \ref{sec:setup-hyp} of an abstract residual
process, which is used to define the LCM as a functional target parameter indexed by time. The LCM equals the zero-function under the hypothesis of conditional 
local independence, and to test this hypothesis we 
introduce an estimator of the LCM in Section \ref{sec:statistic}. The estimator is a stochastic process,
and we describe how sample splitting is to be used for its computation via the
estimation of two unknown components. 

In Section \ref{sec:interpretations} we give interpretations of the LCM and its estimator. 
We show that the LCM estimator is a Neyman
orthogonalized score statistic in Section \ref{sec:neyman}, and in Section \ref{sec:copula} we relate LCM to the partial copula when $X$ is time-independent.

In Section \ref{sec:asymptotics} we state the main results of the paper. We
establish in Section \ref{sec:LCMasymptotics} that the LCM estimator generally approximates the LCM with an error of order $n^{-1/2}$. Under the hypothesis of conditional local independence, we show that the (scaled) LCM estimator converges weakly to a mean zero Gaussian martingale. The estimator requires a model of the
target process $N$ as well as the process $X$ conditionally on
$\mathcal{F}_t$ to achieve the orthogonalization at the core of double machine
learning. The model of $X$ is in this paper expressed indirectly in terms of the 
residual process, and we show that if we can learn the residual process at 
rate $g(n)$ and the model of $N$ at rate $h(n)$ such that $g(n), h(n) \to 0$ and $\sqrt{n} g(n) h(n) \to 0$ for $n \to \infty$ then we achieve a $\sqrt{n}$-rate convergence 
of the LCM estimator. We also show that the variance function of the Gaussian 
martingale can be estimated consistently, and we give a general result
on the asymptotic distribution of univariate test statistics based on the 
LCM estimator. All asymptotic results are presented in the framework of 
uniform stochastic convergence.

Section \ref{sec:lct} gives explicit examples of univariate test statistics,
including the Local Covariance Test based on the normalized supremum of 
the LCM estimator. Its asymptotic distribution is derived and we present 
results on uniform asymptotic level and power. In 
Section \ref{sec:cf} we present the generalization from the sample split estimator to the cross-fit estimator. Though this estimator and the corresponding cross-fit Local Covariance Test (X-LCT) are a bit more involved to compute and analyze, X-LCT 
is more powerful and thus our recommended test for practical usage. 

The survival example from the introduction is used and elaborated upon throughout the paper. We introduce a Cox model in terms of the 
time-varying covariate processes, and we report in Section \ref{sec:simulations} the results from a simulation study based on this model.

The paper is concluded by a discussion in Section \ref{sec:dis}, and 
Appendices \ref{sec:proofs} through \ref{sec:extrafigs} contain:
proofs of results in this paper \eqref{sec:proofs}; 
definitions and results on uniform asymptotics \eqref{app:UniformAsymptotics};  
a uniform version of Rebolledo's martingale CLT \eqref{sec:fclt}; 
an overview of achievable rate results for estimation of nuisance parameters that enter into the LCM estimator \eqref{sec:estimation}; 
and additional results from the simulation study \eqref{sec:extrafigs}.

\section{The Local Covariance Measure} \label{sec:setup}
In this section we present the general framework of the paper,
we define conditional local independence and we introduce the Local Covariance 
Measure as a means to quantify deviations from conditional local independence. 
In Section \ref{sec:statistic} we outline how the Local Covariance Measure 
can be estimated using double machine learning and sample splitting. 
We illustrate the central concepts and methods by an 
example based on Cox's survival model with time-varying covariates. 

We consider a counting process $N = (N_t)$ and another 
real value process $X = (X_t)$, 
both defined on the probability space $(\Omega, \mathbb{F}, \mathbb{P})$. All 
processes are assumed to be defined on a common compact time interval. 
We assume, without loss of generality, that the time interval is $[0,1]$. We will 
assume that $N$ is adapted w.r.t. a right continuous and complete filtration 
$\mathcal{F}_t$, and we denote by $\mathcal{G}_t$ the right continuous 
and complete filtration generated by $\mathcal{F}_t$ and $X_t$. We  
assume throughout that $X$ is \lc{} (that is, has sample paths that are continuous from the left
and with limits from the right), which will ensure bounded sample paths and 
that the process is $\mathcal{G}_t$-predictable.

In the survival example of the introduction, $N_t = \one(T \leq t)$ is the 
indicator of whether death has happened by time $t$, and there can only be one 
event per individual observed. Furthermore, $\mathcal{F}_t = \mathcal{F}_t^{N, Z}$
and $\mathcal{G}_t = \mathcal{F}_t^{N, X, Z}$. 
Our general setup works for any counting process, thus it allows for recurrent 
events and adapted censoring, and the filtration $\mathcal{F}_t$ can 
contain the histories of any number of processes in addition to the history of 
$N$ itself.

\subsection{The hypothesis of conditional local independence} \label{sec:setup-hyp}

The counting process $N$ is assumed to have an $\mathcal{F}_t$-intensity 
$\lambda_t$, that is, $\lambda_t$ is $\mathcal{F}_t$-predictable and with 
$$\Lambda_t = \int_0^t \lambda_s \mathrm{d}s$$
being the compensator of $N$,  
\begin{equation} \label{eq:basicmg}
M_t = N_t - \Lambda_t
\end{equation}
is a local $\mathcal{F}_t$-martingale. Within this framework we can define the 
hypothesis of conditional local independence precisely.

\begin{dfn}[Conditional local independence] \label{dfn:cli}
We say that $N_t$ is conditionally 
locally independent of $X_t$ given $\mathcal{F}_t$ if 
the local $\mathcal{F}_t$-martingale $M_t$ defined by \eqref{eq:basicmg} is also a 
local $\mathcal{G}_t$-martingale.
\end{dfn}

For simplicity, we will refer to this hypothesis as \emph{local independence}
and write 
\begin{equation} \label{eq:H0}
H_0: M_t = N_t - \Lambda_t \text{ is a local } \mathcal{G}_t \text{-martingale}.
\end{equation}
As argued in the introduction, the hypothesis of local
independence is the hypothesis that observing $X$ on $[0,t]$ does not add any
information to $\mathcal{F}_{t-}$ about whether an $N$-event will happen in an
infinitesimal  time interval $[t, t + \mathrm{d}t)$. Definition \ref{dfn:cli} 
captures this interpretation by requiring that the $\mathcal{F}_t$-compensator, 
$\Lambda$, of $N$ is also the $\mathcal{G}_t$-compensator. Thus, $\lambda$ is also the $\mathcal{G}_t$-intensity under $H_0$. 

If $N$ has $\mathcal{G}_t$-intensity $\blambda$, the innovation theorem, Theorem II.T14 
in \cite{Bremaud:1981}, gives that the predictable projection $\lambda_t = E(\blambda_t \mid \mathcal{F}_{t-})$ is the (predictable) $\mathcal{F}_t$-intensity. 
Local independence follows if $\blambda$ is $\mathcal{F}_t$-predictable.  
Intensities are, however, only unique almost surely, and we can have local independence 
even if $\blambda$ is not \emph{a priori} $\mathcal{F}_t$-predictable but have 
an $\mathcal{F}_t$-predictable version. When 
$N$ has $\mathcal{G}_t$-intensity $\blambda$, 
$H_0$ is thus equivalent to $\blambda$ having an $\mathcal{F}_t$-predictable version.
We find Definition \ref{dfn:cli} preferable because it directly 
gives an operational criterion for determining whether $N$ has an 
$\mathcal{F}_t$-predictable version of a $\mathcal{G}_t$-intensity.

Since $X$ is assumed \lc, and thus especially $\mathcal{G}_t$-predictable, the stochastic
integral 
\begin{equation} \label{eq:stochint}
\int_0^t X_{s} \mathrm{d}{M}_s,
\end{equation}
is under $H_0$ a local $\mathcal{G}_t$-martingale. A test could be based on 
detecting whether \eqref{eq:stochint} is, indeed, a local martingale.
We will take a slightly different approach where 
we replace the integrand $X$ by a residual process as defined below.
We do so for two reasons. First, to achieve a $\sqrt{n}$-rate via double 
machine learning we need the integrand to fulfill \eqref{eq:residualorthogonal} below. 
Second, other choices of integrands than $X$ could potentially lead to 
more powerful tests. 
\begin{dfn}[Residual Process] \label{dfn:rescond}
A residual process $G = (G_t)_{t\in[0,1]}$ of $X_t$ given $\mathcal{F}_t$ is a \lc{} 
stochastic process that is $\mathcal{G}_t$-adapted and satisfies 
\begin{align}\label{eq:residualorthogonal}
    \ex(G_t \mid \mathcal{F}_{t-}) = 0, \qquad 
    t\in [0,1].
\end{align}
\end{dfn}
The geometric interpretation is that the residual process evolves such that $G_t$ is orthogonal to $L_2(\mathcal{F}_{t-})$ within $L_2(\mathcal{G}_{t-})$ at each time $t$. One 
obvious residual process is the \emph{additive residual process} given by
\begin{align*}
    G_t =  X_t - \Pi_t = X_t - \ex(X_{t} \mid \mathcal{F}_{t-}).
\end{align*}
where $\Pi_t = E(X_t \mid \mathcal{F}_{t-})$ denotes the 
predictable projection of the \lc{} process $X_{t}$, 
 see Theorem VI.19.2 in \citep{rogers2000}. The additive residual projects $X_t$ onto the orthogonal complement of $L_2(\mathcal{F}_{t-})$, but this may not necessarily remove all $\mathcal{F}_{t}$-predictable information from $X_t$. An alternative 
choice that does so under sufficient regularity conditions is the \emph{quantile residual process} given by
$$
    G_t = F_t(X_t) - \frac{1}{2},
$$
where $F_t$ is the conditional distribution function given by $F_t(x) = \mathbb{P}(X_t \leq x \mid \mathcal{F}_{t-})$.
The quantile residual process satisfies \eqref{eq:residualorthogonal} provided that $(t,x) \mapsto F_t(x)$ is continuous. In Section~\ref{sec:neyman} we discuss additional transformations of  
$X$ that can also be applied before any residualization procedure.

We will formulate the general results in terms of an abstract residual process,
but we focus on the additive residual process in the examples. Any non-degenerate residual process will contain a predictive model of (aspects of) $X_t$ given $\mathcal{F}_{t-}$ in order to satisfy \eqref{eq:residualorthogonal}. We use $\hat{G}_t$ to
denote the residual obtained by plugging in an estimate of that predictive
model. For the additive residual process, the predictive model is $\Pi_t$ and $\hat{G}_t =
X_t - \hat{\Pi}_t$. For the quantile residual process, the predictive model is $F_t$ and $\hat{G}_t = \hat{F}_t(X_t)- \frac12$.

We can now define our functional target parameter of interest, which we call 
the \emph{Local Covariance Measure}.
\begin{dfn}[Local Covariance Measure] \label{dfn:lcm}
With $G_t$ a residual process, define for $t \in [0,1]$
\begin{equation} \label{eq:gamma}
    \gamma_t = \ex\left( I_t \right),
    \qquad
    \text{where}
    \quad
    I_t = \int_0^t G_s \mathrm{d} M_s,
\end{equation}
whenever the expectation is well defined. We call the function $t \mapsto \gamma_t$ 
the Local Covariance Measure (LCM).
\end{dfn}
The following propositions illuminate how $\gamma$ relates to the null hypothesis of $N_t$ being conditionally locally independent of $X_t$ given $\cF_t$.

\begin{prop} \label{prop:cli-mg} 
    Under $H_0$, the process $I = (I_t)$ is a local $\mathcal{G}_t$-martingale with $I_0 = 0$. If $I$ is a martingale, then $\gamma_t = 0$ for $t \in [0,1]$.
\end{prop}

To interpret $\gamma$ in the alternative, we assume that $N$ has $\cG_t$-intensity $\blambda$. 
\begin{prop} \label{prop:gammaalternative}
    If 
    $
        \int_0^1 \ex(|G_s| (\blambda_s + \lambda_s)) \mathrm{d}s
        <\infty
    $, then for every $t\in[0,1]$,
    \begin{align*}
        \gamma_t = \int_0^t \cov(G_s, \blambda_s - \lambda_s) \mathrm{d}s.
    \end{align*}
    In particular, $\gamma$ is the zero-function if and only if $\cov(G_s, \blambda_s - \lambda_s) = 0$ for almost all $s\in [0,1]$.
\end{prop}

We note that under $H_0$, the condition $\int_0^1 \ex(|G_s| \lambda_s) \mathrm{d}s < \infty$ 
is sufficient to ensure that $I$ is a martingale and $\gamma_t = 0$ for 
all $t \in [0,1]$. By Proposition \ref{prop:gammaalternative}, the LCM quantifies 
deviations from $H_0$ in terms of the covariance between the residual process and the difference of the $\cF_{t}$- and $\cG_t$-intensities. To this end, note that if $X$ happens to be $\mathcal{F}_t$-adapted, then 
$\mathcal{G}_t = \mathcal{F}_t$ and $N$ is trivially locally independent of
$X$. 
The hypothesis of local independence is only of interest when
$\mathcal{G}_t$ is a strictly larger filtration than $\mathcal{F}_t$, that is,
when $X$ provides information not already in $\mathcal{F}_t$. 

For the additive residual process, where $G_t = X_t - \Pi_t$, 
\begin{align*}
    \gamma_t & 
    = \ex\left(\int_0^t G_s \mathrm{d} M_s \right) 
    = \ex\left(\int_0^t X_s \mathrm{d} M_s\right) - 
    \ex\left(\int_0^t \Pi_s \mathrm{d} M_s \right)
\end{align*}
provided that the expectations are well defined. 
Since the predictable projection $\Pi_t$ has a \lc{} version and is  
$\mathcal{F}_t$-predictable, and since $M_t$ is a local $\mathcal{F}_t$-martingale,
$\int_0^t \Pi_s \mathrm{d} M_s$
is  a local $\mathcal{F}_t$-martingale. If it is a martingale, it is a 
mean zero martingale, and
\begin{equation} \label{eq:gammarep}
\gamma_t = \ex\left(\int_0^t X_s \mathrm{d} M_s\right) 
=
\ex \left( \sum_{\tau \leq t: \Delta N_\tau = 1} X_\tau 
- 
\int_0^t X_s \lambda_s \mathrm{d} s \right).
\end{equation} 
The computation above shows that the additive residual process defines the same 
functional target parameter $\gamma_t$ as the stochastic integral \eqref{eq:stochint} would. It is, however, the representation of $\gamma_t$ as the
expectation of the residualized stochastic integral that will allow us
to achieve a $\sqrt{n}$-rate of convergence of the 
estimator of $\gamma_t$ in cases where the estimator of $\lambda_t$ converges at a slower rate. 

\subsection{A Cox model with a partially observed covariate process} \label{sec:introex}
To further illustrate the hypothesis of conditional local independence and 
the Local Covariance Measure we consider an example based on Cox's
survival model with time dependent covariates. This is an extension of 
the example from the introduction with $T$ being 
the time to death of an individual, and with $X$ and $Z$ being time-varying processes. 
There is, moreover, one additional time-varying process $Y$ in the full model. 

An interpretation of the processes is as follows:
\begin{align*}
    X & = \text{Pension savings} \\
    Y & = \text{Blood pressure} \\
    Z & = \text{BMI} 
\end{align*}
Periods of overweight or obesity may influence blood pressure in the long term,
and due to, e.g., job market discrimination, high BMI could influence pension
savings negatively. Death risk is influenced directly by BMI and blood pressure
but not the size of your pension savings. Figure \ref{fig:lig} illustrates 
two possible dependence structures among the three processes and the 
death time as local independence graphs, and we will use these two graphs 
to discuss the concept of conditional local independence of pension 
savings on time to death.

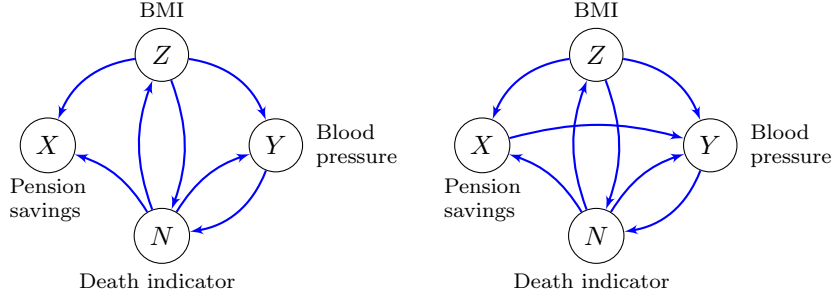
\begin{figure}
    \begin{tabular}{cc}
\begin{tikzpicture}[scale=0.6]
\tikzset{vertex/.style = {shape=circle,draw,minimum size=2em}}
\tikzset{vertexHid/.style = {shape=rectangle,draw,minimum size=2em}}
\tikzset{edge/.style = {->,> = latex', thick}}
\node[vertex] (Z) at  (1.5,2) {$Z$};
\node at (1.5,3) {\footnotesize{BMI}};
\node[vertex] (Y) at  (4,0) {$Y$};
\node[text width=1.1cm]  at (5.8,0) {\footnotesize{Blood} \vskip -1mm 
\footnotesize{pressure}};
\node[vertex] (X) at  (-1,0) {$X$};
\node[text width=1cm] at (-1,-1.2) {\footnotesize{Pension} \vskip -1mm 
\footnotesize{savings}};
\node[vertex] (T) at  (1.5,-2) {$N$};
\node at (1.4,-3) {\footnotesize{Death indicator}};

\draw[edge, bend left=20, color = blue] (Z) to (T);
\draw[edge, bend right=20, color = blue] (T) to (X);
\draw[edge, bend left=20, color = blue] (T) to (Z);
\draw[edge, bend left=20, color = blue] (T) to (Y);
\draw[edge, bend left, color = blue] (Z) to (Y);
\draw[edge, bend right, color = blue] (Z) to (X);
\draw[edge, bend left, color = blue] (Y) to (T);
\end{tikzpicture} & 
\begin{tikzpicture}[scale=0.6]
    \tikzset{vertex/.style = {shape=circle,draw,minimum size=2em}}
    \tikzset{vertexHid/.style = {shape=rectangle,draw,minimum size=2em}}
    \tikzset{edge/.style = {->,> = latex', thick}}
    \node[vertex] (Z) at  (1.5,2) {$Z$};
    \node at (1.5,3) {\footnotesize{BMI}};
    \node[vertex] (Y) at  (4,0) {$Y$};
    \node[text width=1.1cm]  at (5.8,0) {\footnotesize{Blood} \vskip -1mm 
    \footnotesize{pressure}};
    \node[vertex] (X) at  (-1,0) {$X$};
    \node[text width=1cm] at (-1,-1.2) {\footnotesize{Pension} \vskip -1mm 
    \footnotesize{savings}};
    \node[vertex] (T) at  (1.5,-2) {$N$};
    \node at (1.4,-3) {\footnotesize{Death indicator}};
    
    \draw[edge, bend left=20, color = blue] (Z) to (T);
\draw[edge, bend right=20, color = blue] (T) to (X);
\draw[edge, bend left=20, color = blue] (T) to (Z);
\draw[edge, bend left=20, color = blue] (T) to (Y);
    \draw[edge, bend left, color = blue] (Z) to (Y);
    \draw[edge, bend right, color = blue] (Z) to (X);
    \draw[edge, bend left, color = blue] (Y) to (T);
    \draw[edge, bend left=15, color = blue] (X) to (Y);
    \end{tikzpicture}
\end{tabular}

\caption{Local independence graphs illustrating how the three processes $X$,
$Y$, and $Z$ could affect each other and time of death in the Cox example.
There is no direct influence of $X$ (pension savings) on time of death in either
of the two graphs, but in the left graph the death indicator is furthermore
conditionally locally independent of $X$ given the history of $Z$ and $N$. 
In the right graph, $Z$ and $N$ does
not block all paths from $X$ to $N$, thus conditioning on the history of 
$Z$ and $N$ only would not render $N$ conditionally locally independent of $X$.
\label{fig:lig}}
\end{figure}

We assume that $T \in [0,1]$ and that $X$, $Y$ and $Z$ have continuous sample 
paths. Recall also that $N_t = \one(T \leq t)$ is the death indicator process. 
To maintain some form of realism, all processes are stopped at time of 
death, that is, $X_t = X_T$, $Y_t = Y_T$ and $Z_t = Z_T$ for $t \geq T$. This 
feedback from the death event to the other processes is reflected in Figure \ref{fig:lig} by the edges pointing out of $N$. Recall also that 
$$\mathcal{F}_t^{N,Z} = \sigma(N_s, Z_s; s \leq t)$$  
is the filtration generated by the $N$- and $Z$-processes. We use a similar notation 
for other processes and combinations of processes. For example, $\mathcal{F}_t^{N, X, Y, Z}$ 
is the filtration generated by $N$ and all three $X$-, $Y$-, and $Z$-processes. 
With $\lambda_t^{\text{full}}$ denoting the $\mathcal{F}_t^{N, X, Y, Z}$-intensity 
of time of death based on the history of all processes, we assume in this 
example a Cox model given by 
\begin{equation} \label{eq:coxex}
\lambda_t^{\text{full}} = \one(T\geq t)\lambda_t^0 e^{Y_t + \beta Z_t}
\end{equation}
with $\lambda_t^0$ a deterministic baseline intensity. It is not important that 
$\lambda_t^{\text{full}}$ is a Cox model for our general theory, 
but it allows for certain theoretical computations in this example.

The fact that $\lambda_t^{\text{full}}$ does not depend upon $X_t$ implies 
that $\lambda_t^{\text{full}}$ is also the $\mathcal{F}_t^{N, Y, Z}$-intensity,
and according to Definition \ref{dfn:cli}, $N_t$ is conditionally locally independent of $X_t$ given 
$\mathcal{F}_t^{N, Y, Z}$. This is in agreement with the local independence 
graphs in Figure~\ref{fig:lig} where there is no edge in either of them 
from $X$ to $N$.

We will take an interest in the case where $Y$ is unobserved and test the hypothesis:
\begin{center}
    $H_0:$ $N_t$ is conditionally locally independent of $X_t$ 
    given $\mathcal{F}_t^{N, Z}$.
\end{center}
That is, with $Y$ unobserved 
we want test if the intensity of time to death given the history of 
$N$, $X$ and $Z$ depends on $X$. To simplify notation let
$\mathcal{F}_t = \mathcal{F}_t^{N, Z}$ and 
$\mathcal{G}_t = \mathcal{F}_t^{N, X, Z}$ -- in accordance with the general notation.
The $\mathcal{G}_t$-intensity is by the innovation theorem given as
\begin{equation} \label{eq:lamb-marg}
\blambda_t = \ex(\lambda_t^{\text{full}} \mid \mathcal{G}_{t-})
= \one(T\geq t)\lambda_t^0 e^{\beta Z_t} \ex(e^{Y_t} \mid \mathcal{G}_{t-}),
\end{equation}
while the $\mathcal{F}_t$-intensity is 
\begin{equation} \label{eq:lamb}
\lambda_t = \ex(\lambda_t^{\text{full}} \mid \mathcal{F}_{t-}) = 
\one(T\geq t)\lambda_t^0 e^{\beta Z_t} \ex(e^{Y_t} \mid \mathcal{F}_{t-}),
\end{equation}
and $H_0$ is equivalent to $\lambda_t = \blambda_t$ almost surely.
Comparing \eqref{eq:lamb-marg}
and \eqref{eq:lamb} we see that $H_0$ holds in this example if 
$\ex(e^{Y_t} \mid \mathcal{G}_{t-}) = \ex(e^{Y_t} \mid \mathcal{F}_{t-})$,
and a sufficient condition for this to be the case is 
\begin{equation} \label{eq:exas}
\mathcal{F}_t^X \perp \!\!\!\! \perp  \mathcal{F}_t^Y 
\mid  \mathcal{F}_t.
\end{equation}
The condition \eqref{eq:exas} is in concordance with the left graph in Figure \ref{fig:lig}, see Theorem 2 in \cite{Didelez2008}, but not the right, and it implies $H_0$. We will in 
Section \ref{subsec:ex} elaborate on  condition \eqref{eq:exas}  and give explicit 
examples. 

We recall that $H_0$ can  be reformulated as $\blambda_t$ not depending on $X$,
and we could investigate the hypothesis via a marginal Cox model
\begin{equation} \label{eq:coxmarg}
\blambda_t^{\text{cox}} = \one(T\geq t) \blambda_{t}^0 e^{\alpha_1 X_t + \alpha_2 Z_t}
\end{equation}
and test if $\alpha_1 = 0$. The Cox model is, however, non-collapsible
\citep{martinussen2013collaps}, and the semi-parametric model
\eqref{eq:coxmarg} is quite likely misspecified. Consequently, the test of
$\alpha_1 = 0$ is not equivalent to a test of $H_0$. 

Our proposed nonparametric test of $H_0$ does not rely on a specific (semi-)parametric model of $\blambda_t$. To test $H_0$ we consider the LCM using the additive residual process. Then \eqref{eq:gammarep} implies that 
\begin{align*}
   \gamma_t & = \ex\left( X_{T}N_t - \int_0^{t} X_{s} \lambda_s \mathrm{d}s \right),
\end{align*}
By Proposition \ref{prop:cli-mg}, $\gamma_t = 0$ for $t \in [0,1]$ 
under $H_0$, whence conditional local independence implies $\gamma_t = 0$, and 
we test $H_0$ by estimating $\gamma_t$ and testing if it is constantly equal to $0$.

Before introducing a general estimator of the LCM in Section \ref{sec:statistic} 
we outline how to estimate the end point parameter $\gamma_1$ in this example. Due 
to $T \leq 1$ and the appearance of the indicator $\one(T\geq t)$ in \eqref{eq:lamb},
\begin{align*}
   \gamma_1 & = \ex\left( X_{T} - \int_0^{T} X_{s} \lambda_s \mathrm{d}s \right).
\end{align*}
With i.i.d. observations $(T_1, X_1, Z_1), \ldots, (T_n, X_n, Z_n)$ and 
(nonparametric) estimates, $\hat{\lambda}_{j,t}$, based on 
$(T_1, Z_1), \ldots, (T_n, Z_n)$, we could compute the plug-in estimate 
$$\hat{\gamma}_{1, \text{plug-in}}^{(n)} = \frac{1}{n} \sum_{j=1}^n \left( X_{j,T_j} -
\int_0^{T_j} X_{j,s} \hat{\lambda}_{j,s} \mathrm{d}s \right).$$ 
However,  we cannot expect the plug-in estimator to 
have a $\sqrt{n}$-rate unless $\hat{\lambda}$ has $\sqrt{n}$-rate, which 
effectively requires parametric model assumptions on the intensity. 
\begin{figure}[t]
    \includegraphics[width=.8\linewidth]{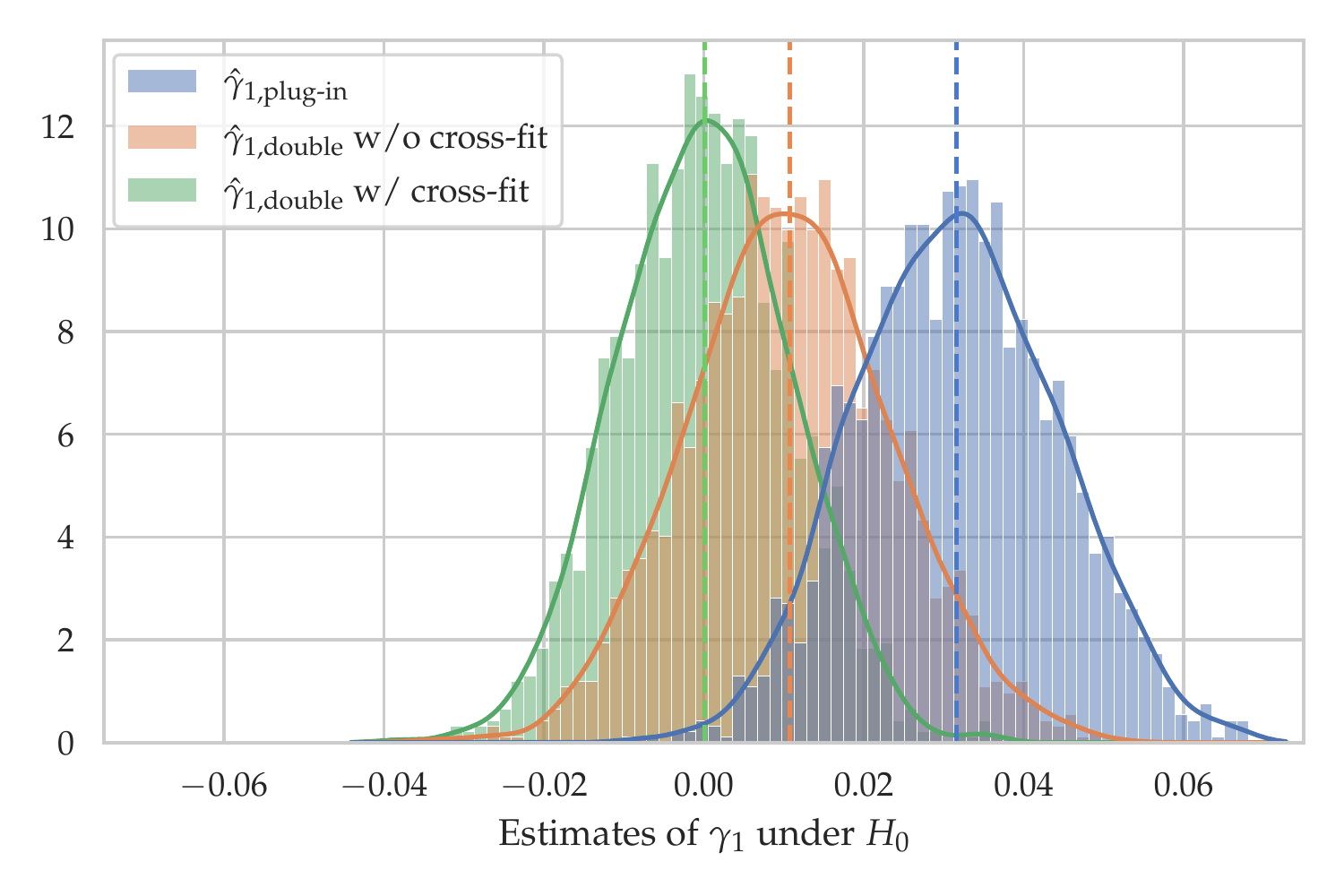}
    \caption{\small Histograms of the distributions of three different 
    estimators of $\gamma_1$.
    Each histogram contains 1000 estimates fitted to samples of size $n=500$. 
    The samples were sampled from a model that satisfies the hypothesis of 
    conditional local independence and hence the ground truth is $\gamma_1=0$. See Section \ref{sec:SamplingScheme} for further details of the data generating process.
    }
    \label{fig:endpoint_example}
\end{figure}
Using the definition of $\gamma_1$ in terms of the additive residual process 
$G_t = X_t - \Pi_t$, we also have that 
\begin{equation} \label{eq:gamrep} 
\gamma_1
= \ex\left( X_{T} - \Pi_{T} - \int_0^T (X_{s} - \Pi_s) \lambda_s \mathrm{d}s
\right).
\end{equation} 
A double machine learning estimator based on the ideas by
\cite{Chernozhukov:2018} is therefore obtained by plugging in two nonparametric estimators: 
$$\hat{\gamma}_{1,\text{double}}^{(n)} = \frac{1}{n} \sum_{j=1}^n 
\left( X_{j,T_j} - \hat{\Pi}_{j,T_j} - \int_0^{T_j} (X_{j,s} - \hat{\Pi}_{j,s}) \hat{\lambda}_{j,s} \mathrm{d}s \right).$$
To achieve a small bias and a $\sqrt{n}$-rate of convergence, we use 
sample splitting. The nonparametric estimates 
$\hat{\Pi}_j$ and $\hat{\lambda}_{j}$ are based on one part 
of the sample only, and are thus independent of the other part of 
the sample used for testing, see Section \ref{sec:statistic}. To 
obtain a fully efficient estimator, multiple sample splits can be combined,
e.g., via cross-fitting, see Section \ref{sec:cf}.

\cref{fig:endpoint_example} shows the distributions of 
$\hat{\gamma}_{1,\text{plug-in}}^{(500)}$ and 
$\hat{\gamma}_{1,\text{double}}^{(500)}$ for the Cox example  
with $\gamma_1 = 0$, see Section \ref{sec:SamplingScheme} for details on the full 
model specification. 
The latter estimator was computed using cross-fitting but also 
without using any form of sample splitting.
The figure illustrates the bias of $\hat{\gamma}_{1,\text{plug-in}}^{(500)}$, 
which is somewhat diminished by double machine learning without sample splitting and 
mostly eliminated by double machine learning in combination with 
cross-fitting.

\subsection{Estimating the Local Covariance Measure} \label{sec:statistic}

To estimate the LCM we assume that we have observed $n$ i.i.d. replications of the
processes, $(N_1, X_1, \mathcal{F}_1), \ldots, (N_n, X_n, \mathcal{F}_n)$, where
observing $\mathcal{F}_{j} = (\mathcal{F}_{j,t})$ signifies that anything
adapted to the $j$-th filtration is computable from observations. The 
process $N_j$ is adapted to $\mathcal{F}_j$, while $X_j$ is not, and 
$\mathcal{G}_j$ denotes the smallest right continuous and complete filtration 
generated by $X_j$ and $\mathcal{F}_j$.
 
For each $n$, we consider a sample split corresponding to a partition $J_n\cup
J_n^c = \{1,\ldots,n\}$ of the indices into two disjoint sets. We let
$\hat{\lambda}^{(n)}$ and $\hat{G}^{(n)}$ be estimates of the intensity and the
residualization map, respectively, fitted on data indexed by $J_n^c$ only. By an
estimate, $\hat{\lambda}^{(n)}$, of $\lambda$ we mean a (stochastic) function
that can be evaluated on the basis of $\mathcal{F}_{j,t}$ for $j \in J_n$, and
its value, denoted by $\hat{\lambda}_{j, t}^{(n)}$, is interpreted as a
prediction of $\lambda_{j,t}$. The stochasticity in $\hat{\lambda}^{(n)}$ arises from
its dependence on data indexed by $J_n^c$, from which its functional form is
completely determined. Similarly, $\hat{G}^{(n)}$ is a function that can be
evaluated on the basis of $\mathcal{G}_{j,t}$ for $j \in J_n$ to give a prediction $\hat{G}_{j,t}^{(n)}$ of
$G_{j,t}$. In Section \ref{subsec:ex} we illustrate through the Cox example how
$\hat{\lambda}^{(n)}$ and $\hat{G}^{(n)}$ are to be computed in practice when we
use sample splitting. In Section \ref{sec:estimation} we give
more examples of such estimation procedures and discuss their statistical properties 
in greater detail. 

To ease notation, we will throughout assume that $(N, X, \mathcal{F})$ 
denotes one additional process and filtration -- independent of and with the 
same distribution as the observed processes. Then the estimated 
intensity $\hat{\lambda}^{(n)}$ and estimated residual process 
$\hat{G}^{(n)}$ can be evaluated on $(N, X, \mathcal{F})$, and thus we may write 
$\hat{\lambda}_t^{(n)}$ and $\hat{G}_t^{(n)}$ to denote template copies of $\hat{\lambda}^{(n)}_{j,t}$ and $\hat{G}^{(n)}_{j,t}$ for $j \in J_n$. 

In terms of the estimates $\hat{\lambda}^{(n)}$ and $\hat{G}^{(n)}$  
we estimate LCM by the stochastic process $\hat{\gamma}^{(n)}$ given by
\begin{equation} \label{eq:LCM}
\hat{\gamma}_t^{(n)} = \frac{1}{|J_n|} \sum_{j \in J_n} \int_0^t 
\hat{G}_{j,s}^{(n)}
\mathrm{d} \hat{M}^{(n)}_{j, s},
\end{equation}
where $\hat{M}_{j, t}^{(n)} = N_{j,t} - 
\int_0^t \hat{\lambda}_{j, s}^{(n)} \mathrm{d}s$. 
We can regard $\hat{\gamma}_t^{(n)}$ as a double machine learning estimator
of $\gamma_t$, with the observations indexed by $J_n^c$ used to learn models of
$\lambda$ and $G$, and with observations indexed by $J_n$ used to estimate
$\gamma_t$ based on these models. In Section \ref{sec:cf} we define the more
efficient estimator that uses cross-fitting, but it is instructive to study the
simpler estimator based on sample splitting first. 

In practical applications, we do not directly observe the filtration $\cF_j$, but rather samples from the stochastic processes generating the filtration. In accordance with the introductory Cox example, consider $\cF_j$ and $\cG_j$ given by $\cF_{j,t} = 
\sigma(Z_{j,s}, N_{j,s}; s\leq t)$ and $\cG_{j,t} = \sigma(X_{j,s}, Z_{j,s}, N_{j,s}; s\leq t)$ for a third stochastic process $Z_j$, with $Z_j$ possibly being multivariate. 
Within this setup, a general procedure for numerically computing the LCM is described in Algorithm \ref{alg:lcm}. Here, historical regression refers to any method which regresses the outcome at a given time on the history of the regressors up to that time. For example, historical linear regression is discussed in Section \ref{sec:simulations} and various alternative methods are discussed in Appendix \ref{sec:estimation}. The choice of sample split will be discussed further in Section \ref{sec:cf} in the context of cross-fitting.

\begin{algorithm} \caption{Sample split estimator of LCM} \label{alg:lcm}
  \textbf{input}: processes $(N_j,X_j,Z_j)_{j=1,\ldots,n}$, partition $J_n \cup J_n^c$ of indices \;
  \textbf{options}: historical regression methods for estimation of $\lambda$ and $G$ given $N$ and $Z$, 
  
  discrete time grid $0 = t_0 <\cdots< t_k \leq 1$\; %
  \Begin{
    historically regress $(X_j)_{j\in J_n^c}$ on $(N_j,Z_j)_{j\in J_n^c}$ 
    to obtain a fitted model $\hat{G}^{(n)}$ \; 
    historically regress $(N_j)_{j\in J_n^c}$ on $(N_j,Z_j)_{j\in J_n^c}$ 
    to obtain a fitted model $\hat{\lambda}^{(n)}$ \; 
    compute out of sample residuals $\hat{G}_{j,t_i}^{(n)}$ and 
    $\hat{M}_{j,t_i}^{(n)}$ for $j \in J_n$ and $i = 0,\ldots,k$ \;
    for each $i=1,\ldots, k$, compute
    $$
    \widetilde{\gamma}_{t_i}^{(n)} = \frac{1}{|J_n|} \sum_{j \in J_n} 
    \sum_{1\leq l \leq i}
    \hat{G}_{j,t_l}^{(n)}
    (\hat{M}^{(n)}_{j, t_l} - \hat{M}^{(n)}_{j, t_{l-1}})
    $$
  }
  \textbf{output}: Local Covariance Measure $\widetilde{\gamma}^{(n)}$ numerically approximated on grid\;
\end{algorithm}

As in Section \ref{sec:introex} we could suggest estimating the entire function 
$t \mapsto \gamma_t$ by a simple plug-in estimator of $\lambda$ 
using the representation \eqref{eq:gammarep}.
\cref{fig:timevaryingexample} illustrates the distribution 
of estimators of the entire time dependent LCM for this plug-in 
estimator together with the double machine learning estimator with and 
without using cross-fitting. The figure also shows the distribution of 
the endpoint being the same distribution shown in Figure 
\ref{fig:endpoint_example}. The simulation is under $H_0$, 
and we see that only the double machine learning estimator with 
cross-fitting results in estimated sample paths centered around $0$.

\begin{figure}
    \includegraphics[width=0.8\linewidth]{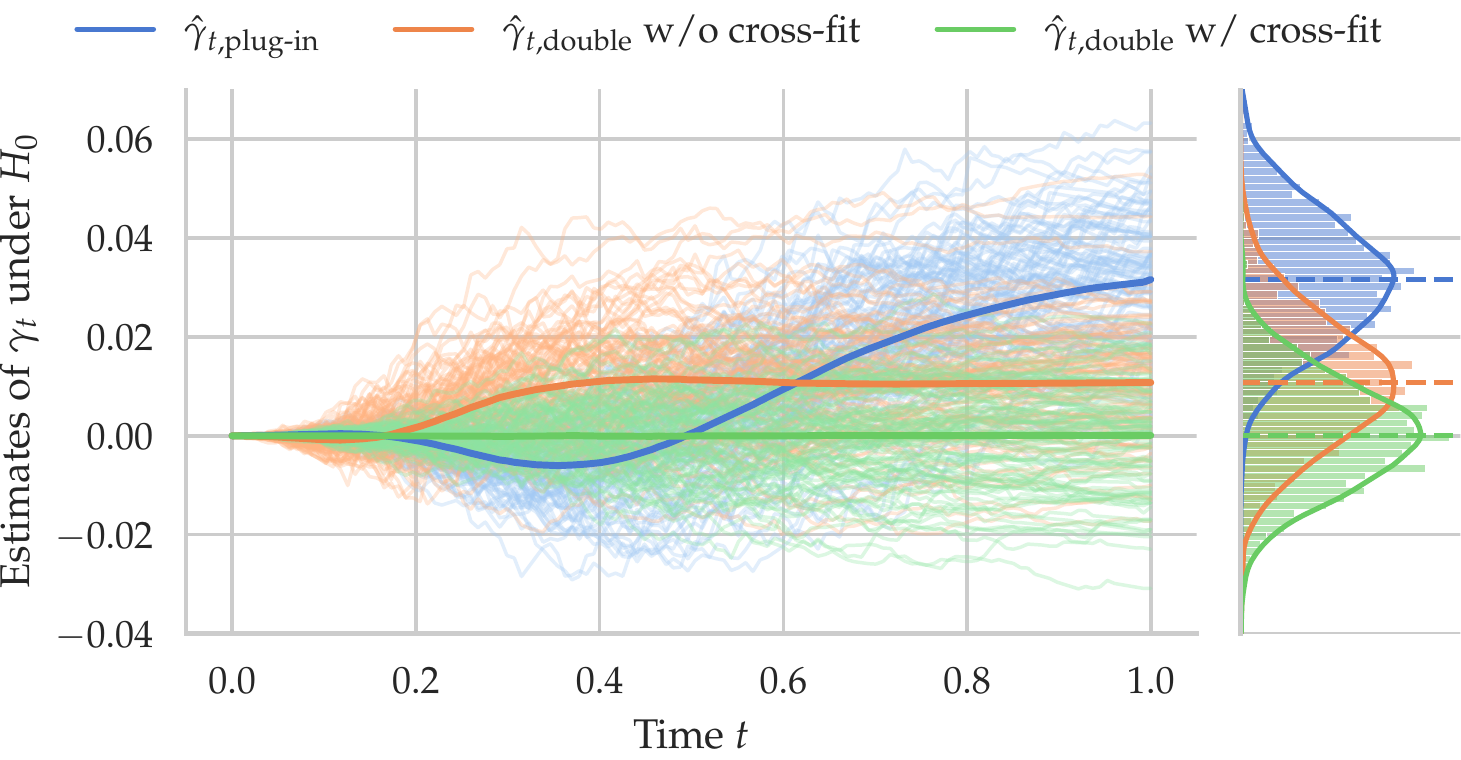}
    \caption{
        A time dependent extension of \cref{fig:endpoint_example}
        showing the distribution of the sample paths  
        $t \mapsto \hat{\gamma}_{t, \mathrm{plug-in}}^{(500)}$ and 
        $t \mapsto \hat{\gamma}_{t, \mathrm{double}}^{(500)}$, 
        the latter with and without using cross-fitting. 
        The data were simulated under $H_0$ where 
        $t\mapsto \gamma_t$ is the zero function.
        See Section \ref{sec:SamplingScheme} for further details of the data generating process.
    }
    \label{fig:timevaryingexample}
\end{figure}
%




\section{Interpretations of the LCM estimator} \label{sec:interpretations}

In this section we provide some additional perspectives on and 
interpretations of the LCM. First 
we show that the LCM estimator can be seen as a Neyman orthogonalization
of the score statistic for a particular one-parameter family. The abstract formulation of the residual process $(G_t)$ permits that we transform $X$ into another $\mathcal{G}_t$-predictable processes. Using this perspective, we may optimize the choice of the process $X$ in terms of power. 

Next we show that when $X$ is independent of time, the 
test statistic reduces in a survival context to a covariance 
between $X$-residuals and Cox-Snell-residuals, which we can link to the 
partial copula between $X$ and the survival time.

\subsection{Neyman orthogonalization of a score statistic} \label{sec:neyman}
Consider the one-parameter family of $\mathcal{G}_t$-intensities
$$\blambda^{\beta}_t = e^{\beta X_t} \lambda_t$$
for $\beta \in \mathbb{R}$. Within this one-parameter family, the hypothesis 
of conditional local independence is equivalent to $H_0 : \beta = 0$.
The normalized log-likelihood with $n$ i.i.d. observations
in the interval $[0,t]$ is  
\begin{align*} 
    \ell_t(\beta) & = \frac{1}{n} \sum_{j=1}^n \left( 
        \int_0^t \log(\blambda_{j,s}^{\beta}) \mathrm{d}N_{j,s} - 
        \int_0^t \blambda_{j,s}^{\beta} \mathrm{d}s\right) \\
    & = \frac{1}{n} \sum_{j=1}^n \left( 
        \int_0^t \beta X_{j,s} + \log(\lambda_{j,s}) \mathrm{d}N_{j,s} - 
        \int_0^t e^{\beta X_{j,s}} \lambda_{j,s} \mathrm{d}s\right).
\end{align*}
Straightforward computations show that 
\begin{align*}
    \partial_{\beta} \ell_t(0) = \frac{1}{n}\sum_{j=1}^n \int_0^t X_{j,s} \mathrm{d} M_{j,s} 
    \quad \text{and} \quad
    - \partial_{\beta}^2 \ell_t(0) = \frac{1}{n} \sum_{j=1}^n \int_0^t X_{j,s}^2 \lambda_{j,s} \mathrm{d}s. 
\end{align*}
If $\lambda$ were known, the score statistic $\partial_{\beta} \ell_t(0)$ 
satisfies $\ex(\partial_{\beta} \ell_t(0)) = \gamma_t$. Moreover, 
under $H_0: \beta = 0$ we have that
$- \partial_{\beta}^2 \ell_t(0) = \langle \partial_{\beta} \ell_t(0) \rangle $ 
is a consistent estimate of the asymptotic variance of the 
mean zero martingale $\partial_{\beta} \ell_t(0)$. The hypothesis of 
local independence -- with $\lambda$ known -- could thus be tested 
using the score test statistic 
$- \partial_{\beta} \ell_t(0)^2 / \partial_{\beta}^2 \ell_t(0)$. 

The nuisance parameter $\lambda$ is, however, unknown and we want to 
avoid restrictive parametric assumptions about $\lambda$. Replacing 
$X_{j,t}$ by the residual process $G_{j,t}$ in the score statistic 
$\partial_{\beta} \ell_t(0)$ gives a \emph{Neyman orthogonalized} 
score 
$$\frac{1}{n}\sum_{j=1}^n \int_0^t G_{j,s} \mathrm{d} M_{j,s}.$$
This score is linear in $\lambda$, and it is not difficult to show 
that it satisfies the Neyman orthogonality condition under $H_0$, cf. Definition 2.1 
in \cite{Chernozhukov:2018}. Indeed, Neyman orthogonality is implicitly 
a central part of the asymptotic results for the LCM estimator 
(in particular Lemma \ref{lem:R1}). The Neyman orthogonalized
score leads directly to the double machine learning 
estimator \eqref{eq:LCM}, and Neyman orthogonality 
in combination with sample splitting are key to showing the $\sqrt{n}$-rate 
of convergence for this estimator. 

The perspective on the LCM estimator as a Neyman orthogonalized score 
statistic suggests that a test based on the LCM has most power against alternatives 
in the one-parameter family $\blambda^{\beta}$. If it happens 
that the most important alternatives are of the form  
$$\blambda^{\beta}_t = e^{\beta \bar{X}_t} \lambda_t$$
for some $\mathcal{G}_t$-predictable process $\bar{X}_t$ different from $X_t$, 
then we should replace $X_t$ by $\bar{X}_t$ in our test statistic, 
that is, in the residualization procedure. Examples of processes 
$\bar{X}_t$ are: 
\begin{itemize}
    \item transformations, $\bar{X}_t = f(X_t)$ for a function $f$
    \item time-shifts, $\bar{X}_t = X_{t-s}$ for $s > 0$
    \item linear filters, $\bar{X}_t = \int_0^t \kappa(t - s) X_s \mathrm{d} s$ for a kernel $\kappa$
    \item non-linear filters, $\bar{X}_t = \phi\left(\int_0^t \kappa(t - s) f(X_s) \mathrm{d} s\right)$
    for a kernel $\kappa$ and functions $f$ and $\phi$.
\end{itemize}
Any finite number of such processes could, of course, also be combined 
into a vector process, and we could, indeed, generalize the LCM estimator \eqref{eq:LCM} to a vector process. 
The generalization is straightforward.

\subsection{Time-independent \texorpdfstring{$X$}{X}} \label{sec:copula}
A different perspective on the test statistic is obtained if $X$ is 
independent of time. If we consider a survival model where $T$ 
is time of death and $N_t = \one(T \leq t)$, then $X$ is
a baseline variable and 
$$\gamma_t = \ex(X (\one(T \leq t) - \Lambda_{t \wedge T})).$$
For $t = 1$, we see that 
$\gamma_1 = \ex(X (1 - \Lambda_T))$ since $T \in [0,1]$ by assumption. Now 
$\Lambda_T$ is exponentially distributed with mean $1$, thus
$$
    \gamma_1 = - \mathrm{cov}(X, \Lambda_T).
$$
Using the additive residual process, the LCM estimator for $t = 1$ is 
$$\hat{\gamma}_1^{(n)} = \frac{1}{|J_n|} \sum_{j \in J_n} 
(X_j - \hat{\Pi}_{j,1})(1 - \hat{\Lambda}_{T_j})$$
which is simply the (negative) empirical covariance between the residuals 
$X_j - \hat{\Pi}_{j,1}$ and the Cox-Snell residuals $\hat{\Lambda}_{T_j}$. 

If we use the quantile residual process $G_t = F_t(X) - \frac{1}{2}$, where 
$F_t(x) = \mathbb{P}(X \leq x \mid \mathcal{F}_{t-})$,
the residual $G_1$ is uniformly distributed and independent of 
$\mathcal{F}_{1-}$ provided $F_1$ is continuous. The LCM estimator for $t = 1$ is 
$$\hat{\gamma}_1^{(n)} = \frac{1}{|J_n|} \sum_{j \in J_n} 
\hat{G}_{j,1} (1 - \hat{\Lambda}_{T_j}),$$
which is again an empirical covariance, but now between the 
generalized residuals $\hat{G}_{j,1}$ and the Cox-Snell 
residuals. This variant of the LCM is closely related to the 
partial copula between $X$ and $T$, which can be estimated as 
$$\frac{1}{|J_n|} \sum_{j \in J_n} 
\hat{G}_{j,1} \left(\frac{1}{2} - \exp(-\hat{\Lambda}_{T_j}) \right).$$
See \cite{Petersen:2021} for further details on the partial 
copula and how this statistic can be used to test the 
(ordinary) conditional independence 
$X \perp \!\!\!\! \perp T \mid \mathcal{F}_{1-}$. 
In contrast to the test based on the partial copula, 
$\hat{\gamma}_1^{(n)}$ extends to the $t$-indexed estimator
$$\hat{\gamma}_t^{(n)} = \frac{1}{|J_n|} \sum_{j \in J_n} 
\hat{G}_{j,t} (\one(T_j \leq t) - \hat{\Lambda}_{t \wedge T_j}),$$
whose asymptotic distribution as a Gaussian martingale follows 
from the general results of this paper. 


\section{General asymptotic results} \label{sec:asymptotics}

In this section we derive uniform asymptotic results regarding the general LCM estimator as a stochastic process. In Section \ref{sec:lct} we discuss how to construct tests of $H_0$ based on the asymptotic results. 

We assume that $N$ has a $\cG_t$-intensity $\blambda_t$, we let $\bLambda_t = \int_0^t \blambda_s\mathrm{d}s$ denote the $\mathcal{G}_t$-compensator of $N$ and let $\mathbf{M}_t = N_t - \bLambda_t$ be the compensated local $\mathcal{G}_t$-martingale. We also recall that $\hat{\gamma}^{(n)}$
denotes the LCM estimator based on sample splitting as defined in Section 
\ref{sec:statistic}. Within this framework we consider the decomposition
\begin{align}\label{eq:decomposition}
    \sqrt{|J_n|}\hat{\gamma}^{(n)}
        = U^{(n)} + R_1^{(n)} + R_2^{(n)} + R_3^{(n)}
            + D_1^{(n)} + D_2^{(n)},
\end{align}
where the processes $U^{(n)}, R_1^{(n)}, R_2^{(n)}, R_3^{(n)}, D_1^{(n)}$, and $D_2^{(n)}$ are given by
    \begin{align}
    U^{(n)}_t & = \frac{1}{\sqrt{|J_n|}} \sum_{j \in J_n} 
    \int_0^t G_{j, s} \mathrm{d} \bM_{j, s}, \label{eq:oracleterm}
    \\
    R^{(n)}_{1, t} & = \frac{1}{\sqrt{|J_n|}} \sum_{j \in J_n}  \int_0^t
    G_{j, s}
    \left( \lambda_{j, s} - \hat{\lambda}_{j, s}^{(n)} \right) 
    \mathrm{d}s,
    \\
    R^{(n)}_{2, t} & = \frac{1}{\sqrt{|J_n|}} \sum_{j \in J_n}  \int_0^t 
    \left( \hat{G}_{j, s}^{(n)} - G_{j, s} \right) \mathrm{d} \bM_{j, s},
    \\
    R^{(n)}_{3, t} & = \frac{1}{\sqrt{|J_n|}} \sum_{j \in J_n}  \int_0^t 
    \left( \hat{G}_{j, s}^{(n)} - G_{j, s} \right) 
    \left( \lambda_{j, s} - \hat{\lambda}_{j, s}^{(n)} \right) 
    \mathrm{d}s,
    \\
    D^{(n)}_{1, t} & = \frac{1}{\sqrt{|J_n|}} \sum_{j \in J_n}  \int_0^t 
    G_{j, s}
    ( \blambda_{j, s} - \lambda_{j, s} ) 
    \mathrm{d}s,
    \\
    D^{(n)}_{2, t} & = \frac{1}{\sqrt{|J_n|}} \sum_{j \in J_n}  \int_0^t 
    (\hat{G}_{j, s}^{(n)} - G_{j, s})
    (\blambda_{j, s} - \lambda_{j, s}) 
    \mathrm{d}s.
    \end{align}
We proceed to show that $U^{(n)}$ and $D_1^{(n)}-\sqrt{|J_n|}\gamma$ each converge in distribution and that the remaining terms converge to the zero-process. This implies that $\sqrt{|J_n|}(\hat{\gamma}^{(n)} - \gamma)$ is stochastically bounded in general, so the LCM estimator will asymptotically detect if the LCM is non-zero. Moreover, recall that $\lambda_t$ is a version of $\blambda_t$ under $H_0$, and hence the processes $D_1^{(n)}$ and $D_2^{(n)}$ are (almost surely) the zero-process in this case. Thus it will follow that $U^{(n)}$ drives the asymptotic limit of the LCM estimator under $H_0$. Based on these general asymptotic results we derive in Section \ref{sec:lct} asymptotic error control for tests based on the LCM estimator.

\subsection{Asymptotics of the LCM estimator}\label{sec:LCMasymptotics}
Our asymptotic results are formulated in terms of uniform stochastic convergence, which has also been discussed extensively in the recent literature on hypothesis testing \citep{shah2020hardness,lundborg2021conditional,lundborg2022projected,scheidegger2022weighted,neykov2021minimax}. Uniform convergence allows us to establish uniform asymptotic level of our proposed test, as well as power under local alternatives. We have collected key definitions and results related to uniform convergence in Appendix \ref{app:UniformAsymptotics}.

To state uniform assumptions and asymptotic results we 
need to indicate a range of possible sampling distributions for which the 
assumptions apply and the results hold. For this purpose, we extend our setup and allow all data to be parametrized by a fixed parameter set $\Theta$. The set $\Theta$ is not 
\emph{a priori} assumed to have any structure, and $\theta \in \Theta$ simply indicates
that $N^{\theta}$, $X^{\theta}$, $\lambda^{\theta}$, 
$G^{\theta}$ etc. have $\theta$-dependent distributions. We generally denote evaluation of processes or derived quantities for a specific $\theta$-value by a superscript, with the LCM, $\gamma^{\theta}$, in particular, depending on $\theta$. The LCM estimator is likewise written as $\hat{\gamma}^{(n),\theta}=(\hat{\gamma}^{(n),\theta}_t)$ for $\theta \in \Theta$ to denote its dependence on the sampling distribution. The superscript notation is, however, heavy and unnecessary in many cases and we will suppress the dependency on $\theta\in \Theta$ whenever it is not needed. Any result that does not explicitly involve $\Theta$ should be understood as a pointwise result for each $\theta \in \Theta$. 

The parametrization allows us to express convergence in distribution and probability uniformly over $\Theta$, which are denoted by $\convUD$ and $\convUP$, respectively. These concepts are defined rigorously in Definition \ref{dfn:UniformConvergence}. We note that uniform convergence reduces to classical (pointwise) convergence if $\Theta$ is a singleton, which corresponds to fixing the sampling distribution. We also introduce the parameter subset
\begin{equation}
    \Theta_0 \coloneqq \{\theta \in \Theta \mid H_0 \text{ is valid}\},
\end{equation}
consisting of all parameter values for which the hypothesis of conditional local independence holds. Correspondingly, we will use $\convUDnull$ and $\convUPnull$ to denote stochastic convergences uniformly over $\Theta_0$.

We are now ready to formulate the underlying assumptions on the data required for our asymptotic results. See the discussion regarding possible relaxations in Section \ref{sec:dis}.
\begin{assump}\label{asm:UniformBounds} 
    There exist constants $C,C'>0$, such that for any parameter value $\theta\in \Theta:$
    \begin{itemize}
        \item[i)] The $\cG_t^\theta$-intensity $\blambda_t^\theta$ of $N^\theta$ is \lc{} with $\sup_{0\leq t \leq 1} \blambda_t^\theta \leq C$ almost surely.
        
        \item[ii)] The residual process $G^\theta$ is \lc{} with $\sup_{0\leq t \leq 1} |G_t^\theta| \leq C'$ almost surely.
    \end{itemize}
\end{assump}

The estimator, $\hat{\lambda}^{(n)}_t$, of $\lambda_t$ and the estimator,
$\hat{G}^{(n)}_t$, of the residual process are assumed to satisfy the same 
bounds as $\lambda_t$ and $G_t$. We note that Assumption \ref{asm:UniformBounds} i) implies that $\bM_t$ is a true $\cG_t$-martingale, and by the innovation theorem, $\lambda_t = \ex[\blambda_t \mid \cF_{t-}]$. As a consequence, the $\cF_t$-intensity $\lambda_t$ inherits the boundedness from the $\cG_t$-intensity $\blambda_t$, and $M_t$ is an $\cF_t$-martingale.
More generally, we have the following proposition ensuring that stochastic integrals 
are true martingales, e.g., that $I_t$ is a martingale under $H_0$.

\begin{prop} \label{lem:truemartingales}
    Under \cref{asm:UniformBounds} it holds that each of the processes
    \begin{equation*}
        \Big(\int_0^t f(G_s)\mathrm{d}\bM_s \Big)_{t\in[0,1]}
        \quad \text{ and } \quad
        \Big(\int_0^t f(\hat{G}_s^{(n)})\mathrm{d}\bM_s \Big)_{t\in[0,1]}
    \end{equation*}
    are mean zero, square integrable $\mathcal{G}_t$-martingales
    for any $f\in C(\mathbb{R})$.  
\end{prop}

To express the asymptotic distribution of $U^{(n)}$ we need its \emph{variance function}. 

\begin{dfn} We define the variance function 
$\mathcal{V}\colon [0,1] \to [0,\infty]$ as
    \begin{align} \label{eq:variance}
    \mathcal{V}(t) = \ex \left(
    \int_0^t G_s^2 \mathrm{d} N_s
    \right).
    \end{align}
\end{dfn}

As everything else, the variance function, $\mathcal{V} = \mathcal{V}^\theta$, is 
also indexed by the parameter $\theta$, which we, for notational simplicity, 
suppress unless explicitly needed.  

By taking $f(x) = x^2$ in Proposition \ref{lem:truemartingales}, 
\cref{asm:UniformBounds} implies that for each $t\in [0,1]$,
    $$
        \mathcal{V}(t) 
        = \ex \left(
            \int_0^t G_s^2 \blambda_s \mathrm{d} s
            \right) < \infty.
    $$ 
Moreover, $\mathcal{V}(t)$ is the variance of  $\int_0^t G_s\mathrm{d}\bM_s$, which under $H_0$ is the same as the variance of $I_t = \int_0^t G_s\mathrm{d}M_s$.

%
%

With the assumptions above we can prove the following proposition 
about the uniform distributional limit of the process 
$U^{(n)}$ in the \emph{Skorokhod space} $D[0,1]$, the space of \cl{} functions from $[0, 1]$ to $\mathbb{R}$ endowed with the Skorokhod topology. A corresponding pointwise result is an application of Rebolledo's classical
martingale CLT. Our generalization to uniform convergence is based on 
a uniform extension of Rebolledo's theorem, see Theorem \ref{thm:URebo} 
in Appendix \ref{sec:fclt}. 
\begin{prop}\label{prop:UasymptoticGaussian}
    Under Assumption \ref{asm:UniformBounds} it holds that
    \begin{align*}
        U^{(n),\theta} \convUD U^\theta
    \end{align*}
    in $D[0,1]$ as $n\to \infty$, where for each $\theta\in \Theta$, $U^\theta$ is a mean zero continuous Gaussian martingale on $[0,1]$ with variance function $\mathcal{V}^\theta$.
\end{prop}

To control the remainder terms in \eqref{eq:decomposition} we will bound 
the estimation errors in terms of the 2-norm, $\vertiii{\cdot}_{2}$, on  $L_2([0,1] \times \Omega)$, i.e.,
    \begin{align*}
    \vertiii{W}_{2}^2 = \ex \left( \int_0^1 W_s^2 \mathrm{d}s
    \right)
    \end{align*}
for any process $W\in L_2([0,1] \times \Omega)$.
We will make the following consistency assumptions on 
$\hat{\lambda}^{(n)}$ and $\hat{G}^{(n)}$. 
\begin{assump} \label{asm:UniformRates}
Assume that $|J_n| \to \infty$ when $n \to \infty$ and let
    \begin{align*}
        g^\theta(n) &= \vertiii{G^\theta - \hat{G}^{(n),\theta}}_{2},
        \\
        h^\theta(n) &= \vertiii{\lambda^\theta - \hat{\lambda}^{(n),\theta}}_{2}.
    \end{align*}
Then each of the sequences $g^\theta(n)$, $h^\theta(n)$, and $\sqrt{|J_n|}g^\theta(n)h^\theta(n)$ converge to zero uniformly over $\Theta$ as $n\to\infty$, i.e.,
\begin{align*}
    \lim_{n\to \infty}\sup_{\theta\in\Theta} \max\{ g^\theta(n), h^\theta(n), 
        \sqrt{|J_n|}g^\theta(n)h^\theta(n) \} = 0.
\end{align*}
\end{assump}

With this assumption we can establish that the remainder terms also converge uniformly
to the zero-process.
\begin{prop}\label{prop:remainderterms}
    Under Assumptions \ref{asm:UniformBounds} and \ref{asm:UniformRates}, it holds that 
    \begin{equation*}
        \sup_{t\in [0,1]} |R_{i,t}^{(n),\theta}| \convUP 0
    \end{equation*}
    as $n\to \infty$ for $i=1,2,3$. 
\end{prop}
%
%

To control the asymptotic behavior of the LCM estimator in the alternative we 
need to control the two terms $D_1^{(n)}$ and $D_2^{(n)}$.

\begin{prop}\label{prop:Dasymptotics}
    Let Assumptions \ref{asm:UniformBounds} and \ref{asm:UniformRates} hold true. 
    \begin{enumerate}
    \item[i)] The stochastic process
    $D_1^{(n),\theta}-\sqrt{|J_n|} \gamma^{\theta} $
    converges in distribution in $(C[0,1],\|\cdot\|_\infty)$ uniformly over $\Theta$ as $n\to \infty$.
    \item[ii)] If $G_t^\theta = X_t^\theta - \Pi_t^\theta$ is the additive residual process, then $D_2^{(n),\theta}\convUP 0$ in $D[0,1]$ as $n\to \infty$.
    \end{enumerate} 
\end{prop}
We note that $D_2^{(n)}$ might not vanish without an assumption like $G_t$ 
being the additive residual process, and it is not clear if $D_2^{(n)}$ will even converge in general. We will not pursue an analysis of the asymptotic behavior of $D_2^{(n)}$ in the general case. We note, however, that if we can estimate $G$ with 
a parametric rate, that is, $\sqrt{|J_n|} g(n) = O(1)$, then 
it follows from the Cauchy-Schwarz 
inequality that $D_2^{(n)}$ is stochastically bounded, and $D_1^{(n)}$ 
still dominates in the alternative where $\gamma \neq 0$. 

We can combine all of the propositions into a single theorem regarding the asymptotics of the LCM estimator, which we consider as our main result.

\begin{thm}\label{thm:main}
    Let Assumptions \ref{asm:UniformBounds} and \ref{asm:UniformRates} hold true.
    \begin{enumerate}
        \item[i)] It holds that
            \begin{align*}
                \sqrt{|J_n|}\hat \gamma^{(n),\theta} \convUDnull U^\theta
            \end{align*}
        in $D[0,1]$ as $n\to \infty$, where for each $\theta\in \Theta_0$, $U^\theta$ is a mean zero continuous Gaussian martingale on $[0,1]$ with variance function $\mathcal{V}^\theta$.

        \item[ii)] For the additive residual process it holds that
        for every $\varepsilon>0$ there exists $K>0$ such that
        \begin{align}\label{eq:stochboundedness}
            \limsup_{n\to \infty}\sup_{\theta\in \Theta} \mathbb{P}\pa{
            \sqrt{|J_n|} \cdot \|\hat{\gamma}^{(n),\theta} - \gamma^\theta\|_\infty > K} < \varepsilon.
        \end{align}
    \end{enumerate}
\end{thm}
Thus we have established the weak asymptotic limit of 
$\sqrt{|J_n|} \hat{\gamma}^{(n)}$ under $H_0$. However, the variance function $\mathcal{V}$ of the limiting Gaussian 
martingale is unknown and must be estimated from data. We 
propose to use the empirical version of \eqref{eq:variance},
    \begin{align}\label{eq:empiricalvariance}
    \hat{\mathcal{V}}_n(t) = 
    \frac{1}{|J_n|} \sum_{j \in J_n} 
    \int_0^t \left(\hat{G}_{j,s}^{(n)}\right)^2 \mathrm{d} N_{j,s} 
    = \frac{1}{|J_n|} \sum_{j \in J_n} 
    \sum_{\tau \leq t: \Delta N_{j,s} = 1} \left(\hat{G}_{j,\tau}^{(n)}\right)^2,
    \end{align}
for which we have the following consistency result.

\begin{prop}\label{prop:varianceconsistent}
    Under Assumptions \ref{asm:UniformBounds} and \ref{asm:UniformRates} it holds that 
    $$
        \sup_{t\in[0,1]}|\hat{\mathcal{V}}_n^\theta(t) - \mathcal{V}^\theta(t)| \convUP 0,
    $$
    as $n \to \infty$.
\end{prop}
We emphasize that $\mathcal{V}$ is only the asymptotic variance function of the 
LCM estimator under $H_0$. It is always the asymptotic variance function of 
$U^{(n)}$, but in the alternative the asymptotic distribution of $\hat{\gamma}^{(n)}$
also involves the asymptotic distribution of $D_1^{(n)}$ and is thus more complicated.

Tests of conditional local independence can now be 
constructed in terms of univariate functionals of 
$\hat{\gamma}^{(n)}$ and $\hat{\mathcal{V}}_n$ that quantify the magnitude of the LCM. The asymptotics of such test statistics under $H_0$ are described in the following corollary, which is essentially an application of the continuous mapping theorem.
\begin{cor} \label{cor:statisticsdistribution}
    Let $\mathcal{J}\colon D[0,1]\times D[0,1] \to \mathbb{R}$ be a functional that is continuous on the closed subset $C[0,1]\times \overline{\{\mathcal{V}^\theta \colon \theta \in \Theta_0\}}$
    with respect the uniform topology, i.e., the topology generated by the norm $\|(f_1,f_2)\|= \max\{\|f_1\|_\infty,\|f_2\|_\infty\}$ for $f_1,f_2\in D[0,1]$. Define the test statistic 
    $$
        \hat{D}_n^\theta = \mathcal{J}\left( 
            \sqrt{|J_n|}\hat{\gamma}^{(n),\theta}, \; \hat{\mathcal{V}}_n^\theta\right).
    $$ 
    Under Assumptions \ref{asm:UniformBounds} and \ref{asm:UniformRates},
    it holds that
        \begin{align}
            \hat{D}_n^\theta \convUDnull \mathcal{J}(U^\theta,\mathcal{V}^\theta),
            \qquad n \to \infty,
        \end{align}
    where $U^\theta$ is a mean zero continuous Gaussian martingale with variance function $\mathcal{V}^\theta$.
\end{cor}

\section{The Local Covariance Test} \label{sec:lct}
In this section we introduce a practically applicable 
test based on the LCM estimator. Using the asymptotic distribution of the
LCM estimator we show that the asymptotic distribution of our proposed test 
is independent of the sampling distribution under $H_0$ and has an explicit representation. 
We show, in addition, uniform asymptotic level of the test, and we give a uniform 
power result for the additive residual process. Finally, we modify the test to be 
based on a cross-fitted estimator of the LCM instead of using sample splitting, 
and we show uniform level of that test.

To construct a test statistic based on the LCM estimator it is beneficial that
its distributional limit does not depend on the variance function. As a simple example,
consider the endpoint test statistic:
    \begin{align} \label{eq:endpointstatistic}
        \big(\hat{\mathcal{V}}_n(1)\big)^{-\frac{1}{2}}\sqrt{|J_n|}\hat{\gamma}^{(n)}_1,
    \end{align}
which under $H_0$ converges in distribution to 
$\mathcal{V}(1)^{-\frac{1}{2}}U_1$ by Corollary \ref{cor:statisticsdistribution}. The distribution of the latter is the standard normal distribution, and in particular it does not depend on $\mathcal{V}$. 

Any test statistic constructed from $\hat{\gamma}^{(n)}$ should  
capture deviations of $\gamma_t$ away from $0$. The test statistic in
\eqref{eq:endpointstatistic} does, however, only consider the endpoint of the
process, and since $\gamma$ is not necessarily monotone, $\gamma_t$ may 
deviate more from $0$ for other $t \in [0,1]$. 
Thus in order to increase power against such alternatives we consider the test statistic
    \begin{align} \label{eq:supremumstatistic}
        \hat{T}_n = 
        \frac{\sqrt{|J_n|}}{\sqrt{\hat{\mathcal{V}}_n(1)}}  \sup_{0 \leq t \leq 1} 
        \big|
            \hat{\gamma}^{(n)}_{t}
        \big|.
    \end{align}
We refer to $\hat{T}_n$ as the \emph{Local Covariance Test statistic} (LCT statistic). We proceed to show that the LCT statistic can be calibrated to obtain a test of $H_0$ with asymptotic level, and which has asymptotic power against any alternative with a non-zero LCM. This is the best we can hope for of any test based on the LCM estimator. 

We note that it might be possible to establish similar results for other norms of the LCM, for example, a statistic based on a weighted $L_2$-norm\footnote{such statistics are known as Anderson-Darling type statistics.}. However, since other norms of the distributional limit $U$ will generally have a distribution with a complicated dependency on $\mathcal{V}$, we believe that the LCT statistic is the simplest to calibrate.

To establish uniform asymptotic level via Corollary
\ref{cor:statisticsdistribution} for tests based on test statistics such as \eqref{eq:supremumstatistic} we need to assume that the asymptotic variances in $t = 1$ 
are uniformly bounded away from zero.

\begin{assump}  \label{asm:UniformVariance}
    There exists a $\delta_1 > 0$ such that for all $\theta \in \Theta$ it holds 
    that $\mathcal{V}^{\theta}(1) \geq \delta_1$.
\end{assump}

\subsection{Type I and type II error control}
We proceed to show that under $H_0$, the LCT statistic is distributed as $\sup_{0\le t \le 1} |B_t|$, where $(B_t)$ is a standard Brownian motion. From this point onwards, we let $S$ denote a random variable with such a distribution and note that its CDF can be written as:
    \begin{align} \label{eq:supdistribution}
        F_S(x) = \mathbb{P}(S \leq x ) = \frac{4}{\pi} \sum_{k=0}^\infty \frac{(-1)^k}{2k+1}
        \exp\left(- \frac{\pi^2(2k+1)^2}{8x^2}\right),
        \qquad x>0.
    \end{align}
See, for example, Section 12.2 in \citet{SchillingPartzsch:2012} where the formula is derived from Lévy's triple law. 

The $p$-value for a test of $H_0$ equals $1 - F_S(\hat{T}_n)$, and since the series in \eqref{eq:supdistribution}
converges at an exponential rate, the $p$-value can be computed with high numerical precision by truncating the series. Given a significance level $\alpha \in (0,1)$, we also let $z_{1-\alpha}$ denote the $(1-\alpha)$-quantile of $F_S$, which exists and is unique since the right-hand side of \eqref{eq:supdistribution} is strictly increasing and continuous. The \emph{Local Covariance Test} (LCT) with significance level $\alpha$ is then defined by
\begin{align}\label{eq:LCT}
    \Psi_n 
    = \Psi_n^\alpha
    = \one(F_S(\hat{T}_n)> 1-{\alpha}) 
    = \one(\hat{T}_n > z_{1-{\alpha}}).
\end{align}

From Theorem \ref{thm:main} we can now deduce the asymptotic properties of the LCT under 
the hypothesis of conditional local independence. 
Recall that $\convUDnull$ denotes uniform convergence in distribution under $H_0$.
\begin{thm}\label{thm:LCTlevel}
    Let Assumptions \ref{asm:UniformBounds}, \ref{asm:UniformRates} and \ref{asm:UniformVariance} hold true. Then it holds that
    $$
        \hat T_n^\theta \convUDnull S
    $$
    as $n\to \infty$. As a consequence, for any $\alpha \in (0,1)$,
    \begin{align*}
        \limsup_{n\to \infty} \sup_{\theta \in \Theta_0} \mathbb{P}(\Psi_n^{\alpha, \theta} =1) \leq \alpha.
    \end{align*}
    In other words, the Local Covariance Test defined in \eqref{eq:LCT} has uniform asymptotic level $\alpha$.
\end{thm}
In general, we cannot expect that the test has power against alternatives to $H_0$ for which the LCM is the zero-function. This is analogous to other types of conditional independence tests based on conditional 
covariances, e.g., GCM \citep{shah2020hardness}. However, we do have the following result that establishes power against local alternatives with $\|\gamma\|_\infty$ decaying at an order of at most $|J_n|^{-1/2}$.
\begin{thm}\label{thm:rootNpower}
    Let Assumptions \ref{asm:UniformBounds} and \ref{asm:UniformRates} hold true. 
    Using the additive residual process it holds that for any $0<\alpha<\beta<1$ there exists $c>0$ such that
    \begin{align*}
        \liminf_{n\to \infty}\inf_{\theta \in \mathcal{A}_{c,n}} \mathbb{P}(\Psi_n^{\alpha,\theta} = 1) \geq \beta,
    \end{align*}
    where $\mathcal{A}_{c,n} = \{\theta \in \Theta \mid \|\gamma^\theta\|_\infty \geq c |J_n|^{-1/2}\}$.
\end{thm}

\subsection{Extension to cross-fitting} \label{sec:cf} 
In Section \ref{sec:asymptotics} we considered sample splitting with observations indexed by
$J_n^c$ used to estimate the two models and with observations indexed by $J_n$
used to estimate $\gamma$. Following \citet{Chernozhukov:2018}, we can improve
efficiency by cross-fitting, i.e., by flipping the roles of $J_n$ and $J_n^c$ to
obtain a second equivalent estimator of $\gamma$. Heuristically, the two
estimators are approximately independent, and thus their average should be a more
efficient estimator. This procedure generalizes directly to a partition $J_n^1
\cup \cdots \cup J_n^K = \{1,\ldots,n\}$ of the indices into $K$ disjoint folds.
The partition is assumed to have a uniform asymptotic density, meaning that
$|J_n^k|/n \to \frac{1}{K}$ as $n \to \infty$ for each $k$. 

We estimate $G$ and $\lambda$ using $(J_n^k)^c=\{1,\ldots, n\}\setminus
J_n^k$ and subsequently estimate $\gamma$ using $J_n^k$. Then the $K$-fold \emph{Cross-fitted LCM estimator}, abbreviated as X-LCM, is defined as the average LCM estimator over the $K$ folds, i.e.,
    \begin{align} \label{eq:KfoldI}
        \check{\gamma}_t^{K,(n)} 
        = \frac{1}{K} \sum_{k=1}^K \frac{1}{|J_n^k|}\sum_{j \in J_n^k} 
        \int_0^t \hat{G}_{j,s}^{k,(n)}\mathrm{d} \hat{M}^{k,(n)}_{j, s},
    \end{align}
where for each $j \in J_n^k$, the processes $\hat{G}_{j}^{k,(n)}$ and $\hat{M}^{k,(n)}_{j}$ are the model predictions of $G_j$ and $M_j$, respectively, based on training data indexed by $(J_n^k)^c$. 
We also define a $K$-fold version of the variance estimator:
    \begin{align} \label{eq:Kfoldsig}
        \check{\mathcal{V}}_{n}^K(t)
        = 
        \frac{1}{K} \sum_{k=1}^K \frac{1}{|J_n^k|} \sum_{j \in J_n^k} 
        \int_0^t \left(
            \hat{G}_{j,s}^{k,(n)}\right)^2 \mathrm{d} N_{j,s}.
    \end{align}
Now, similarly to the LCT statistic, the cross-fitted estimator can be used to construct a test statistic,
    \begin{align} \label{eq:DMLtest}
        \check{T}_{n}^K 
        = \sqrt{\frac{n}{\check{\mathcal{V}}_{K,n}(1)}} \sup_{0\leq t \leq 1}\left| \check{\gamma}_t^{K,(n)} \right|,
    \end{align}
from which we define the following test of conditional local independence.
\begin{dfn} \label{dfn:test} Let $\alpha \in (0,1)$ and let $\check{T}_{n}^K$ be
    the test statistic from \eqref{eq:DMLtest}. The $K$-fold
    \emph{Cross-fitted Local Covariance Test} (X-LCT) with
    significance level $\alpha$ is defined by
        \begin{align*}
            \check{\Psi}_n^K 
            = \one(F_S(\check{T}_{n}^K) > 1-\alpha)
            = \one(\check{T}_{n}^K > z_{1-\alpha}),
        \end{align*}
    where $z_{1-\alpha}$ is the $(1-\alpha)$-quantile of the distribution
    function $F_S$ given in \eqref{eq:supdistribution}.
\end{dfn}

We provide a summary of the computation of the X-LCT in Algorithm \ref{alg:X-lct}. The asymptotic analysis of $\hat{\gamma}^{(n)}$ generalizes to $\check{\gamma}^{K,(n)}$, but we will refrain from restating all results for the $K$-fold cross-fitted estimator. For simplicity, we focus on the fact that the X-LCT is well calibrated.
\begin{thm} \label{thm:LCTXlevel}
    Suppose that \cref{asm:UniformRates} is satisfied for every sample split $J_n^k \cup (J_n^k)^c, k=1,\ldots,K$. Under \cref{asm:UniformBounds,asm:UniformVariance}, the X-LCT statistic satisfies
    $$
        \check{T}_{n}^{K,\theta} \convUDnull S
    $$
    for $n\to \infty$. In particular, the X-LCT has uniform asymptotic level.
\end{thm}
Note that cross-fitting recovers full efficiency in the sense that the scaling
factor is $\sqrt{n}$ rather than $\sqrt{|J_n|}$, which leads to a more powerful test.
Moreover, the asymptotic
distribution of $\check{T}_{n}^K$ does not depend on the number of folds $K$, and any difference between various choices of $K$ can thus be attributed to finite sample errors. Larger values of $K$ will allocate more data to estimation of $G$ and $\lambda$, which intuitively should be the harder estimation problem. Following Remark 3.1 in \citet{Chernozhukov:2018}, we believe that a default choice of $K=4$ or $K=5$ should be reasonable in practice.


\begin{algorithm} \caption{$K$-fold cross-fitted local covariance test (X-LCT)} \label{alg:X-lct}
  \textbf{input}: processes $(N_j,X_j,Z_j)_{j=1,\ldots,n}$, partition $J_n^1 \cup \cdots \cup J_n^K$ of indices into $K$ folds\;
  \textbf{options}: historical regression methods for estimation of $\lambda$ and $G$ given $N$ and $Z$, 
  
  discrete time grid $\mathbb{T}\subset [0,1]$, significance level $\alpha\in (0,1)$\; %
  \Begin{
    \For{$k = 1,\ldots, K$}{
        apply Algorithm \ref{alg:lcm} on sample split $J_n^k \cup (J_n^k)^c$ to compute $\widetilde{\gamma}^{k,(n)}$ on grid $\mathbb{T}$\;
        use Equation \eqref{eq:empiricalvariance} on sample split $J_n^k \cup (J_n^k)^c$ to compute $\widetilde{\mathcal{V}}_{k,n}(1)$ \;
    }
    compute 
    $\check{\gamma}^{K,(n)} 
        = \frac{1}{K} \sum_{k=1}^K \widetilde{\gamma}^{k,(n)}$
    on grid $\mathbb{T}$ \;
    compute  
        $\check{\mathcal{V}}_{K,n}(1) = \frac{1}{K} \sum_{k=1}^K \widetilde{\mathcal{V}}_{k,n}(1)$\; 
    compute the X-LCT statistic
    $\check{T}_n^K = \sqrt{n} \cdot \max_{t\in \mathbb{T}}|\check{\gamma}_t^{K,(n)}| / \sqrt{\check{\mathcal{V}}_{K,n}(1)}$ \;
    compute $p$-value $\check{p}=1-F_S(\check{T}_n^K)$ by truncating the series in Equation \eqref{eq:supdistribution}.
  }
  \textbf{output}: 
  the X-LCT $\check{\Psi}_n^K = \one(\check{p} < \alpha)$, and optionally the $p$-value $\check{p}$\;
\end{algorithm}

\section{Simulation study} \label{sec:simulations} In this section we present
the results from a simulation study based on the Cox example introduced in
Section \ref{sec:introex}. We elaborate in Section \ref{subsec:ex} on the 
full model specification used for the simulation study -- which will also 
illuminate how $\Pi$ and $\lambda$ can be modeled and estimated. 
The results from the simulation study  
focus on the distribution of the X-LCT statistic $\check{T}_{n}^K$ and validate the asymptotic level and power of the X-LCT $\check{\Psi}_n^K$. The latter is also compared to a hazard ratio test based on the marginal Cox model \eqref{eq:coxmarg}. The simulations were implemented in Python and
the code is
available\footnote{\url{https://github.com/AlexanderChristgau/nonparametric-cli-test}}.

\subsection{Cox model continued} \label{subsec:ex}
Consider the same setup as in Section \ref{sec:introex}. To fully specify the 
model we need to specify the distribution of the processes $X$, $Y$ and $Z$.
We suppose that $X$ and $Y$ can be written in terms of $Z$ as
\begin{align}\label{eq:historicallinear}
X_t = \int_0^t Z_s \rho_X(s, t) \mathrm{d}s + V_t,
\qquad \text{and} \qquad
Y_t = \int_0^t Z_s \rho_Y(s, t) \mathrm{d}s + W_t,
\end{align}
where $\rho_X$ and $\rho_Y$ are two functions defined on the triangle $\{(s,t)
\in [0,1]^2 \mid s \leq t \}$, and where $V = (V_t)_{0\leq t \leq 1}$ and $W =
(W_t)_{0\leq t \leq 1}$ are two noise processes with mean zero. The processes
$Z$, $V$ and $W$ are assumed independent, which implies \eqref{eq:exas} and thus 
that $N$ is conditionally locally independent of $X$ given $\mathcal{F}_t = \mathcal{F}_t^{N,Z}$.

The specific dependency of $X$ and $Y$ on $Z$ is known as the \emph{historical functional linear
model} in functional data analysis \citep{Malfait:2003}.
Within this model, 
\begin{equation} \label{eq:piModel}
\Pi_t = \ex(X_t \mid \mathcal{F}_{t-}) = \int_0^t Z_s \rho_X(s, t) \mathrm{d} s,
\end{equation}
and on $(T \geq t)$
\begin{align*}
    \ex(e^{Y_t} \mid \mathcal{F}_t) 
    & = e^{\int_{0}^t Z_s \rho_Y(s, t) \mathrm{d} s} 
    \ex(e^{W_t} \mid T \geq t) \\
    & = e^{\tilde{\beta}_0(t) 
        + \int_{0}^t Z_s \rho_Y(s, t) \mathrm{d} s},
\end{align*}
where $\tilde{\beta}_0(t) = \log (\ex(e^{W_t} \mid T \geq t))$. Since
$$
    \lambda_t 
    = \one(T\geq t)\lambda_t^0 
        e^{\beta Z_t} \ex(e^{Y_t} \mid \mathcal{F}_t)
$$
it follows that on $(T \geq t)$,
\begin{align} 
\log(\lambda_t) 
    & = \log(\lambda_t^0) + \tilde{\beta}_0(t) + \beta Z_t 
        + \int_{0}^t Z_s \rho_Y(s, t) \mathrm{d} s \nonumber \\
    & = \beta_0(t) + \beta Z_t 
        + \int_{0}^t Z_s \rho_Y(s, t) \mathrm{d} s,
\label{eq:lambdaModel}
\end{align}
where the two baseline terms depending only on time have been merged into
$\beta_0$. 

The computations above suggest how the estimators $\hat \lambda^{(n)}$ and $\hat \Pi^{(n)}$ 
could be constructed. That is, $\hat \lambda^{(n)}$ could be based on estimates of
$\beta$, $\beta_0$ and $\rho_Y$ from the
observations $(T_j,Z_j)_{j\in J_n^c}$, and $\hat \Pi^{(n)}$ could be based on estimates of  
$\rho_X$ from $(X_j,Z_j)_{j\in J_n^c}$. We would then have 
$$\hat{\Pi}^{(n)}_{j,t} =  \int_0^t Z_{j,s} \hat{\rho}_X^{(n)}(s, t) \mathrm{d} s$$
for $j \in J_n$ where $\hat{\rho}_{X}^{(n)}$ denotes the estimate of $\rho_X$, 
and similarly for $\hat \lambda^{(n)}_{j,t}$.
Particular choices of estimators $\hat{\rho}_{X}^{(n)}$ and $\hat{\rho}_{Y}^{(n)}$ and their 
theoretical properties are reviewed in \cref{sec:estimation}. Our conclusion from this 
review is that for the historical functional linear model, sufficient rate results 
should be possible but have not yet been established rigorously.

More seriously, we found the available implementations limiting. Specifically, 
the historical linear regression estimator from the \texttt{scikit-fda} library was 
considered initially, but we found that fitting this model was too
computationally expensive for a simulation study with cross-fitting. 
In principle, in our time-continuous setting, we would like to use a functional estimator 
of $\Pi$ that would utilize the regularity along $s$ and $t$. Initial experiments, 
however, suggested that the simpler historical regression described in Section 
\ref{sec:implementation} gave similar results as using the \texttt{scikit-fda} library,
and we went with the less time consuming implementation.

\subsection{Sampling scheme}\label{sec:SamplingScheme}
The actual time-discretized simulations and computations were implemented using an
equidistant grid $\mathbb{T} = (t_i)_{i=1}^{q}$ with $q=128$ time points $0=t_1
< \cdots < t_q = 1$. Inspired by \citet{harezlak2007penalized}, we generated the
processes as follows: let $\xi\in \mathbb{R}^3$ and $V,W,\VVV \in
\mathbb{R}^\mathbb{T}$ be independent random variables such that $\xi
\sim \mathcal{N}(0,\operatorname{I}_3)$ and such that $V,W$, and $\VVV$ are identically
distributed with 
    \begin{align*}
        V_{t_1},V_{t_2} - V_{t_1}, \ldots, V_{t_q}-V_{t_{q-1}} 
        \stackrel{\mathrm{i.i.d.}}{\sim} \mathcal{N}(0,1 / q).
    \end{align*}
Then the process $Z$ is determined by
    \begin{align*}
        Z_t = \xi_1 + \xi_2 t + \sin(2\pi \xi_3 t) + \VVV_t
    \end{align*}
for $t\in \mathbb{T}$. The processes $X$ and $Y$ were then given by the
historical linear model \eqref{eq:historicallinear} with kernels $\rho_X$ and
$\rho_Y$ being one of the following four kernels:
\begin{align*}
    \begin{array}{l l}
        \text{zero: } (s,t)\mapsto 0, 
        & \text{constant: }(s,t)\mapsto 1, \\
        \text{Gaussian: } (s,t)\mapsto e^{-2(t-s)^2}, \qquad
        &\text{sine: } (s,t)\mapsto \sin(4t-20s).
    \end{array}
\end{align*}
To compute $X$ and $Y$, we evaluated the kernels on $\{(s,t)\in \mathbb{T}^2 \mid
s\leq t\}$ and approximated the integrals by Riemann sums. The full intensity for $N_t=\one(T\leq t)$ was specified with a Weibull baseline of the form
\begin{equation*}
    \lambda_t^{\text{full}} = \one(T\geq t) \beta_1 t^2 
        \exp\left(\beta_2 Z_t + Y_t\right),
\end{equation*}
for $\beta_1>0$ and a choice of $\beta_2 \in \{-1,1\}$. 
To sample $T$ we applied the inverse hazard method, which utilizes that $\Lambda_T^{\text{full}}$ is standard exponentially distributed. That is, we sampled $E \sim \mathrm{Exp}(1)$ and numerically computed $T = \max\{ t\in \mathbb{T} \mid \Lambda_t^{\text{full}} < E\}$ as a discretized approximation.
For any given parameter setting, the baseline coefficient $\beta_1$ was chosen sufficiently large to ensure that 
$\Lambda_t^{\text{full}} \geq E$
would occur before time $t=1$ in more that $\frac{q-1}{q}\cdot n$ samples.

With this setup, \cref{asm:UniformBounds} is satisfied if $V$, $W$ 
and $\VVV$ were bounded. Since we use the Gaussian distribution, they are 
technically not bounded, but they could be made bounded by introducing a 
lower and upper cap. Due to the light tails of the Gaussian distribution
such caps would have no noticeable effect on the simulation results, and 
the results we report are generated without a cap. 

The simulation setting used to sample the data in Figures \ref{fig:endpoint_example} and \ref{fig:timevaryingexample} was
\begin{align*}\label{eq:simsettings}
    (\beta_2,\rho_X,\rho_Y)
    = (-1, \text{constant}, \text{constant}).
\end{align*}

\subsection{Implementation of estimators and tests} \label{sec:implementation}
For our proof-of-concept implementation we used two simple off-the-shelf
estimators.

To estimate $\lambda$ we used the \texttt{BoXHED2.0} estimator from
\citet{pakbin2021boxhed2}, based on the works of \citet{wang2020boxhed} and
\citet{lee2021boosted}. In essence, the estimator is a gradient boosted forest
adapted to the setting of hazard estimation with time-dependent covariates. The
maximum depth and number of trees were tuned by 5-fold cross-validation over
the same grid as in \citet{pakbin2021boxhed2}. For computational ease, the
hyperparameters were tuned once on the entire dataset instead of tuning them on
each fold $J_n^k$. In principle, this may invalidate the asymptotic properties of $\check{\Psi}_n^K$ since it breaks the independence between $\hat{\lambda}^{k,(n)}$ and $(T_j,X_j,Z_j)_{j\in J_n^k}$, but we believe that this dependency is negligible.

To estimate the predictable projection $\Pi_t = \ex(X_t \mid \mathcal{F}_{t-})$,
we fitted a series of linear least squares estimators by regressing $X_t$ on
$(Z_s)_{s\in \mathbb{T}:s<t}$ for each $t\in \mathbb{T}$. To stabilize the
estimation error $g(n)$, we added a small $L_2$-penalty with coefficient $0.001$
fixed across all experiments for simplicity. Since $X_t$ was sampled from a discretized historical linear model, the error $g(n)$ should in principle converge with a classical $n^{-1/2}$-rate. However, the finite sample error is expected to be large since it accounts for $q$ linear regressions with up to $q$ predictors.

Based on these estimators, the X-LCT was implemented based on Algorithm \ref{alg:X-lct}. Following the
recommendation by \citet[Remark 3.1.]{Chernozhukov:2018}, we computed the X-LCT with $K=5$ folds. The associated $p$-value was
computed with the series representation of $F_S$ truncated to the first $1000$ terms.

We compared our results for X-LCT with a hazard ratio test in the possibly misspecified 
marginal Cox model given by \eqref{eq:coxmarg}. This test was computed using the lifelines library
\citep{davidson2021lifelines}, specifically the \texttt{CoxTimeVaryingFitter}
model. The model was fitted with an $L_2$-penalty with a coefficient set
to $0.1$ (the default), and as a consequence the hazard ratio test is expected to be conservative.

\begin{figure}
    \includegraphics[width=\linewidth]{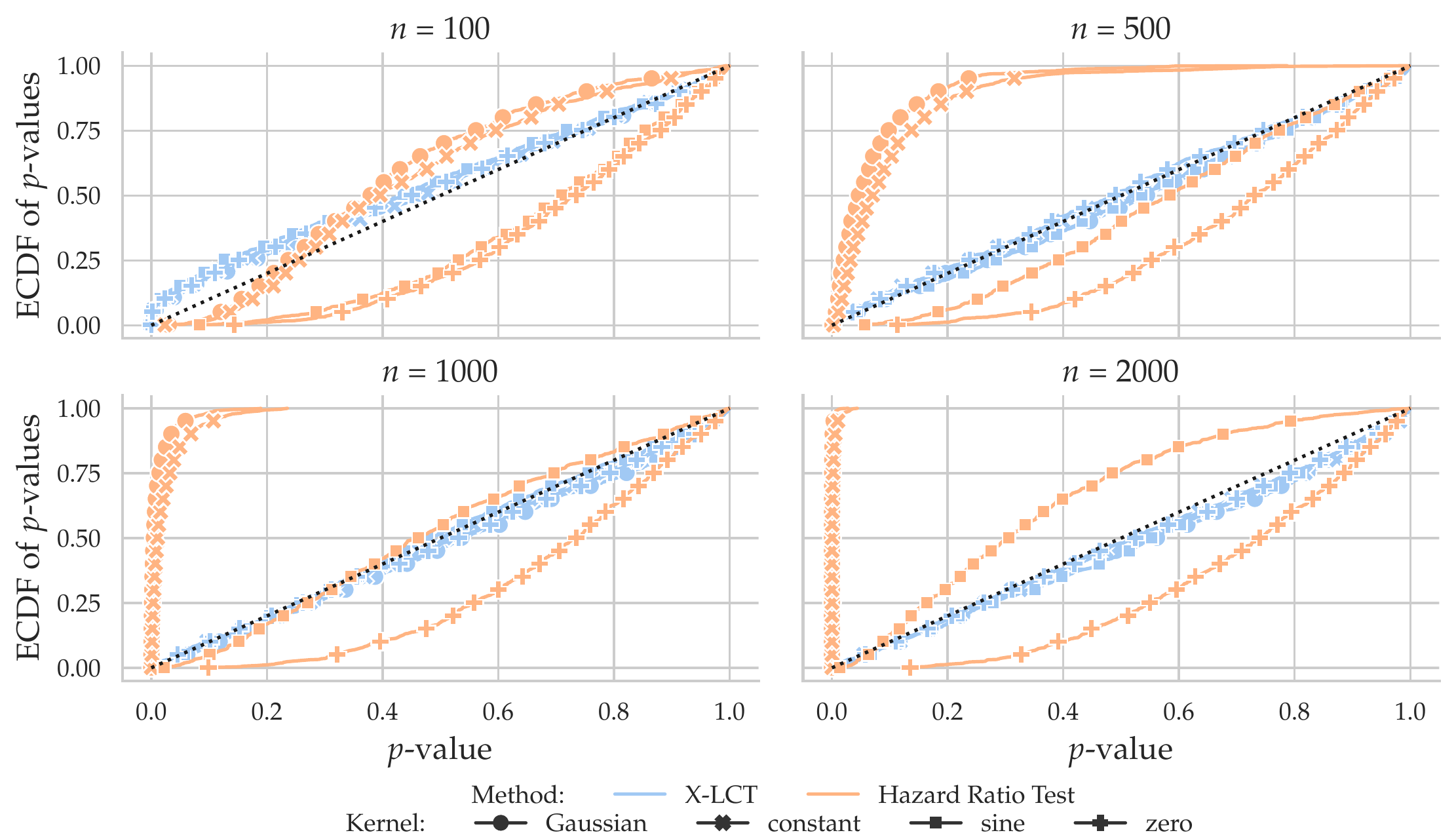}
        \caption{Empirical cumulative distribution functions of simulated
        $p$-values for the cross-fitted local covariance test and the hazard ratio test.
        The simulated data
        satisfies the hypothesis of conditional local independence, so the
        $p$-values are supposed to be uniformly distributed, and the CDF 
        should fall on the diagonal dotted line.}\label{fig:H0pvals2}
\end{figure}

\subsection{Distributions of \texorpdfstring{$p$}{p}-values under \texorpdfstring{$H_0$}{H0}}
We examine the distributional approximation $\check{T}_{n}^K \stackrel{\mathrm{as.}}{\sim} S$, cf. Theorem \ref{thm:LCTXlevel}, by comparing the $p$-values
$1-F_S(\check{T}_{n}^K)$ to a uniform distribution. Figure \ref{fig:H0pvals2}
shows the empirical distribution functions of the $p$-values computed from data
simulated according to the scheme described in the previous section. 
The results are aggregated over the two choices of $\beta_2 \in \{-1,1\}$ since
these two settings were found to be similar. For more detailed results from
the experiment, see Figure \ref{fig:H0pvals_full} in Appendix \ref{sec:extrafigs}, 
which also includes the $p$-values corresponding to the endpoint test statistic.

For the hazard ratio test, Figure \ref{fig:H0pvals2} shows that the $p$-values
are sub-uniform for the zero-kernel. In this case, the marginal Cox model is
correct, and the non-uniformity of the $p$-values can be explained by the
$L_2$-penalization. For the constant and Gaussian kernels the hazard ratio test
fails completely, whereas for the sine kernel, the mediated effect of $Z$ on $T$
through $Y$ is more subtle, and the model misspecification only becomes apparent
for $n=2000$. Overall, these results are consistent with the reasoning in the
Section \ref{sec:introex}: a test based on the misspecified Cox model will wrongly 
reject the hypothesis of conditional local independence. 

For the proposed X-LCT, Figure \ref{fig:H0pvals2} shows that the associated $p$-values are slightly anti-conservative for $n=100$. This is to be expected, and can be
explained by the finite sample errors leading to more extreme values of $\check{T}_{n}^K$ than the approximation by $S$. As $n$ increases, these errors
become smaller -- and for $n=2000$ the $p$-values actually seem to be
sub-uniform. The sub-uniformity may be explained by the time discretization,
since the maximum of the process is taken over $\mathbb{T}$ rather than $[0,1]$.
\cref{fig:supremumapprox} in Appendix \ref{sec:extrafigs} illustrates the asymptotic effect of
the time discretization which supports this claim. Another support of this claim
is that the endpoint test does not appear to give sub-uniform $p$-values for
large $n$, see Figure \ref{fig:H0pvals_full}. We finally note that the 
distributions of the $p$-values for our proposed test is largely unaffected 
by the kernel used to generate the data. 

\begin{figure}
    \includegraphics[width = \linewidth]{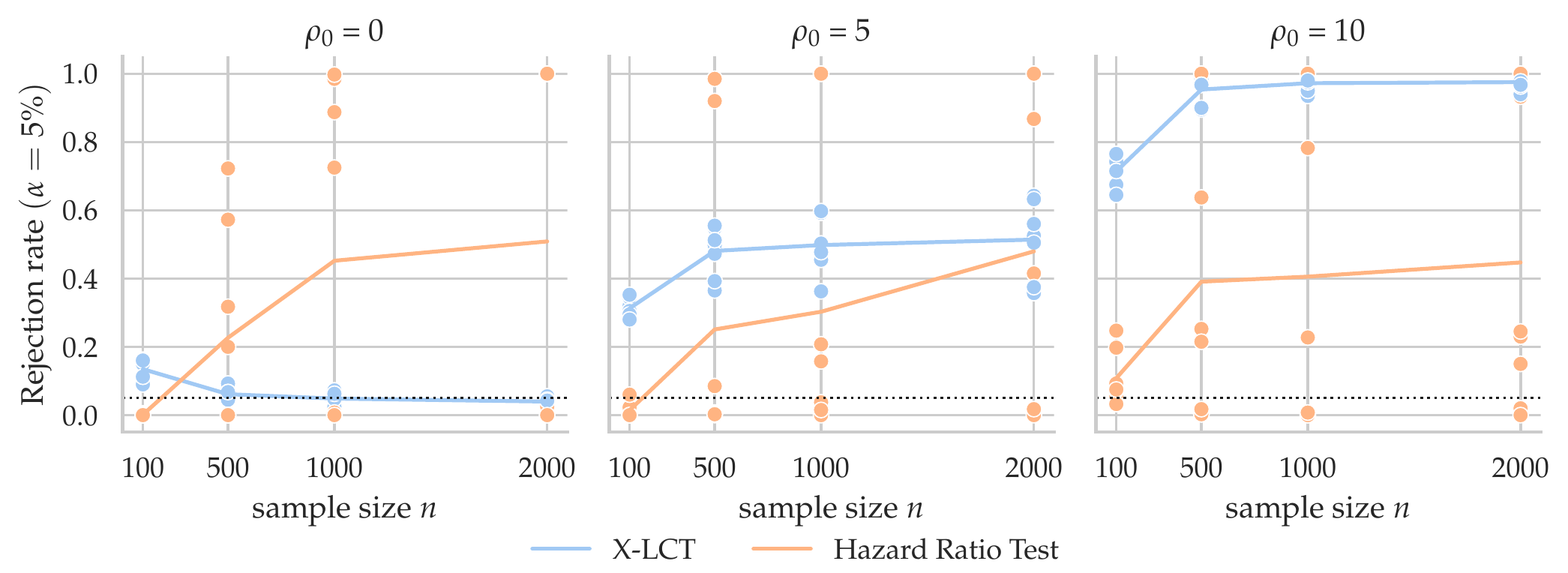}
    \caption{For each $\rho_0 \in \{0,5,10\}$, the lines show the average 
    rejection rates of our proposed test X-LCT (blue) and the hazard ratio test (orange)
    as functions of sample size, with each average taken over $8$ different 
    settings. For each setting, the rejection rate is computed from
    400 simulated datasets at a $5\%$ significance level 
    and the rejection rate is displayed with a dot.}
    \label{fig:alternatives}
\end{figure}

\subsection{Power against local alternatives} \label{subsec:localalternatives}
To investigate the power of the X-LCT we construct local alternatives to $H_0$ in
accordance with the right graph in Figure \ref{fig:lig} by replacing $Y_t$ by
the process $Y_t + \frac{\rho_0}{\sqrt{n}} X_t$. That is, for $\rho_0 \neq 0$,
blood pressure is then directly affected by pension savings, and $N_t$ is no
longer conditionally locally independent of $X_t$ given $\mathcal{F}_t$. In
terms of the full intensity, these local alternatives are equivalent to 
\begin{equation} \label{eq:alternative}
    \lambda_t^{\text{full}} = \one(T\geq t) \beta_1 t^2 
        \exp\left(\beta_2 Z_t + Y_t + \frac{\rho_0}{\sqrt{n}} X_t\right).
\end{equation}
We simulated data for the dependency parameter $\rho_0 \in \{0,5,10\}$. Note
that $\rho_0 = 0$ corresponds to our previous sampling scheme with conditional
local independence. For each of the $96 = 4\times 2 \times 4 \times 3$ choices
of kernel, $\beta_2$, $n$ and $\rho_0$ we ran the tests $400$ times and computed the
$p$-values. For simplicity, we report the rejection rate at an $\alpha = 5\%$
significance level and the results are shown in \cref{fig:alternatives}.

In the leftmost panel, the data was generated under $H_0$ and the plot shows what
we noted previously, namely that the X-LCT holds level for large $n$,
whereas the hazard ratio test does not. 

For the local alternatives, $\rho_0=5$
and $\rho_0 = 10$, we note that the power of the hazard ratio test is quite
sensitive to the simulation settings. For some settings it has no power, while
for others it has some power. 

In contrast, the proposed X-LCT has power against all of the 
local alternatives. The power increases with $n$ initially but 
stabilizes from around $n=1000$. This is similar to the 
behavior observed under the null hypothesis and is not surprising.
We expect that the sample size needs to be sufficiently large for the 
nonparametric estimators to work sufficiently well, and we expect 
the sufficient sample size to be mostly 
unaffected by the value of $\rho_0$. For fixed $n$, we also note that the power 
of $\check{\Psi}_n^K$ is fairly robust with respect to the choice of $\beta_2$ 
and the choice of kernel. Overall, we find that the X-LCT
is applicable in these settings with historical effects: 
it has consistent power against the $\sqrt{n}$ alternatives while maintaining level for $n$ reasonably large.

We now compare the X-LCT, which is based on the uniform norm of the X-LCM, with its endpoint counterpart. More precisely, we consider the test statistic
$$
    \left(\check{\mathcal{V}}_{K,n}(1)\right)^{-\frac{1}{2}}
    \sqrt{n}\check{\gamma}_1^{K,(n)},
$$
which is asymptotically standard normal under $H_0$.
With the simulation settings in Section \ref{subsec:localalternatives}, the X-LCT turns out to be more 
or less indistinguishable from the corresponding endpoint test. 
This is because the alternatives considered have corresponding parameters 
$t \mapsto \gamma_t$, which are most extreme towards $t=1$.
Therefore, the supremum and the endpoint behave similarly in these cases.

\begin{figure}
    \includegraphics[width=.8\linewidth]{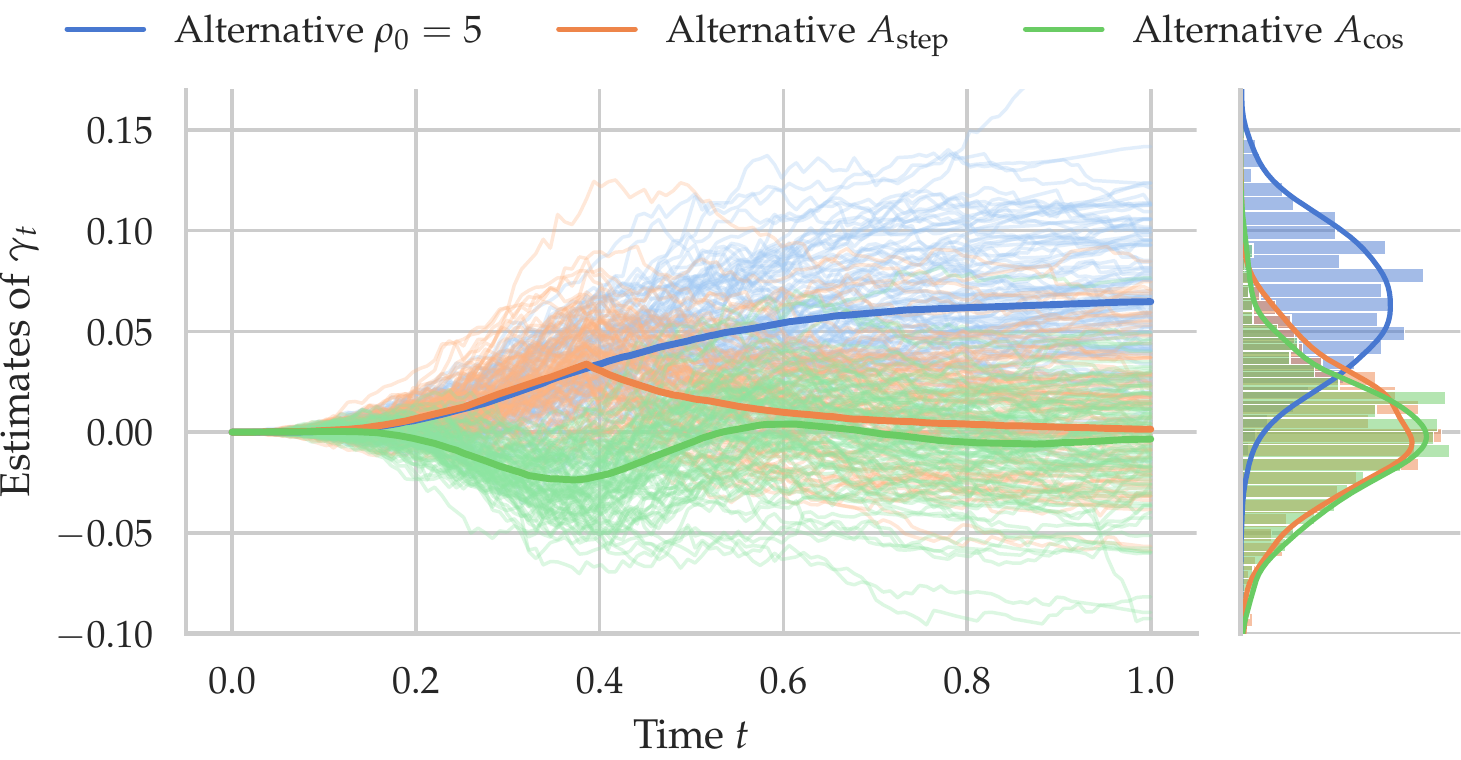}
    \caption{Sample paths of $\check{\gamma}^{K,(500)}$ fitted on data
    sampled from three different alternatives as described in Section 
    \ref{subsec:localalternatives}. 
    Here $(X,Y,Z)$ are sampled from the scheme described in \cref{sec:simulations},
    with both $\rho_X$ and $\rho_Y$ being the constant kernel and with $\beta = -1$.
    For each alternative, 100 paths are shown. The 
    empirical mean functions and the endpoint distributions
    are highlighted and computed based on 500 samples.}
     \label{fig:pathsalternative}
\end{figure}

For this reason we consider local alternatives that result in a non-monotonic 
parameter $t\mapsto \gamma_t$. Using the same expression for the 
intensity \eqref{eq:alternative}, but with a time-varying $\rho_0$, 
we consider the alternatives
    \begin{align*}
        A_\text{step}\colon& 
            \rho_0(t) = 5 \cdot \one(t\leq 0.4) - 5 \cdot \one(t>0.4),\\
        A_\text{cos} \colon& 
            \rho_0(t) = 7 \cdot \cos(4\pi \cdot t).
    \end{align*}
The idea behind the alternative $A_\text{step}$ is that $t\mapsto \gamma_t$
should be increasing on $[0,0.4]$ and decreasing on $(0.4,1]$. 
\cref{fig:pathsalternative} shows sample paths
of $\check{\gamma}^{K,(n)}$ for data simulated under each of the alternatives
$\rho_0=5$, $A_\text{step}$ and $A_\text{cos}$. The figure illustrates that
$t \mapsto |\check{\gamma}_t^{K,(n)}|$ is, indeed, mostly maximal towards $t=1$ 
for the alternative $\rho_0=5$, but not for the time-varying alternatives 
$A_\text{step}$ and $A_\text{cos}$.

With the same sampling scheme for $(X,Y,Z)$ as in Section \ref{sec:SamplingScheme}, we conducted
an analogous experiment with 400 runs for each setting.
\cref{fig:timevaryingrho} shows the rejection rates for the 
two tests.

Under the hypothesis of conditional local independence, the left plot 
in \cref{fig:timevaryingrho} shows that
the endpoint test behaves similarly to $\check{\Psi}_n^K$ as expected. Both
tests have power against the local alternatives, but for $A_\text{step}$ the
power does not seem to stabilize before $n=2000$. This is different from the
previous settings, and can be explained by a slower convergence of the intensity
estimator due to the more complex dependency on $X$. For both of the local
alternatives, we observe that $\check{\Psi}_n^K$ is more powerful than the
endpoint test, with the difference being largest for $A_\text{step}$. 
In conclusion, these results show that the supremum test dominates the 
endpoint test in certain situations.

\begin{figure}
    \includegraphics[width=\linewidth]{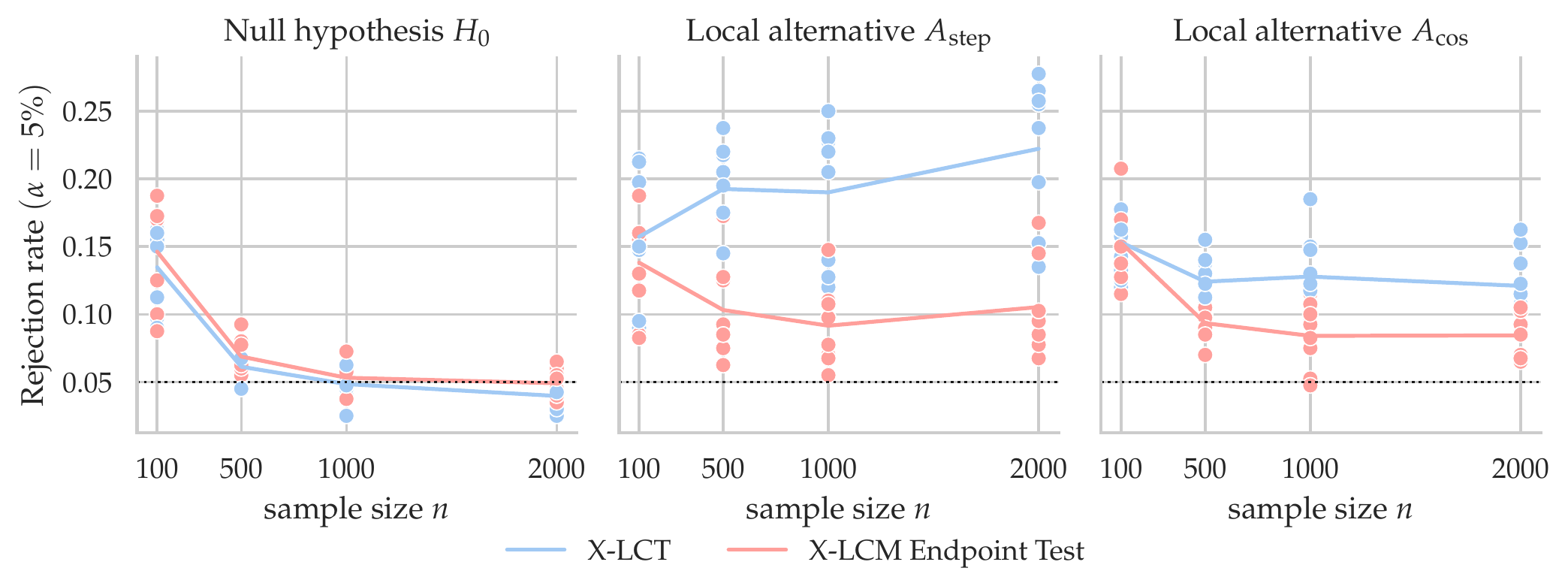}
    \caption{The plots show the average rejection rate of the 
    double machine learning tests based on the supremum statistic (blue)
    and the endpoint statistic (red).} \label{fig:timevaryingrho}
\end{figure}

\section{Discussion} \label{sec:dis} 
The LCM was introduced as a functional parameter that quantifies deviations from the hypothesis $H_0$ of conditional local independence. 
We showed how the parameter may be expressed in several ways, but that it is the representation in terms of the residual process that allows us to estimate the LCM with a $\sqrt{n}$-rate under $H_0$ without parametric model assumptions. The residual process was introduced as an abstract model of $X_t$ for each $t$ given the history up to time $t$, and we showed that such a residualization could be viewed as a form of orthogonalization. Similar ideas have been used recently for classical conditional independence testing, such as GCM \citep{shah2020hardness}, tests based on the partial 
copula \citep{Petersen:2021}, and GHCM \citep{lundborg2021conditional}. It is, however,  not possible to use any of these to test $H_0$, which 
cannot be expressed as a classical conditional independence. Our test based on the LCM 
is the first nonparametric test of conditional local independence with substantial 
theoretical support, and we propose to test $H_0$ in practice by using X-LCT based on the cross-fitted estimator of LCM.

Contrary to the tests of conditional independence mentioned above, we need
sample splitting -- even under $H_0$ -- to achieve our asymptotic results. 
We do not believe that this can be avoided. The standard argument to avoid this uses 
classical conditional independence in a crucial way, which does not translate 
into our framework -- basically because we condition on information that
changes with time. Our simulation study also indicates 
that sample splitting 
or cross-fitting is needed in practice for the LCM estimator to be unbiased 
under $H_0$.

While our cross-fitted estimator of the LCM, the X-LCM, share 
some of the general patterns of other double machine learning procedures -- including 
the overall decomposition \eqref{eq:decomposition} -- our analysis and results required 
a range of generalizations of known results and some novel ideas. 
The asymptotic distribution of the leading term, $U^{(n)}$, is also a well 
known consequence of Rebolledo's CLT, see, e.g., Section V.4 in \citep{AndersenBorganGillKeiding:1993} for related results in the context of survival analysis. However, we generalized this result to uniform convergence in the 
Skorokhod space $D[0,1]$, and we introduced new techniques for handling 
the remainder terms. These novel techniques are made necessary by the 
decomposition \eqref{eq:decomposition} being a decomposition of 
stochastic processes indexed by time. 
We outline below the three most important technical contributions we made.


First, to obtain uniform control of level and power, all asymptotic results in Section \ref{sec:asymptotics} are formulated in terms of uniform stochastic convergence. Since this notion of convergence had not previously been considered on general metric spaces, and especially not on the Skorokhod space, we had to develop the necessary theory. This development could be of 
independent interest, and we have collected the general definitions and main results on uniform stochastic convergence in metric spaces in Appendix \ref{app:UniformAsymptotics}. 
This framework also allowed us to show a uniform version of Rebolledo's martingale CLT 
in Appendix \ref{sec:fclt}.

Second, to establish distributional convergence under $H_0$, we need to control the remainder terms $R_{i,t}^{(n)}$ uniformly over $t$. The third term, $R_{3}^{(n)}$, is simple to bound, and by exploiting Doob's submartingale inequality, the second term, $R_{2}^{(n)}$, can also be bounded. The most difficult first term, $R_{1}^{(n)}$, was controlled using  
stochastic equicontinuity via an exponential tail bound and the use of
the chaining lemma. The necessary general uniform stochastic equicontinuity and chaining arguments are collected in Section \ref{app:UniformChaining} of Appendix \ref{app:UniformAsymptotics}.

Third, to achieve rate results in the alternative, the processes $D_1^{(n)}$ and $D_2^{(n)}$ must be controlled. The process $D_1^{(n)}$ does, like $U^{(n)}$, not involve any estimation, and its distributional convergence follows from a general CLT argument for continuous stochastic processes. The term $D_{2}^{(n)}$ is more difficult to handle, as it may not have mean zero if $G_t$ is not the additive residual process. However,  $X_t$ cancels out in $\hat{G}_t^{(n)} - G_t$ for the additive residual process, which makes the difference $\mathcal{F}_t$-predictable, and $D_{2}^{(n)}$ can then be bounded similarly to $R_1^{(n)}$. For a general residual process, it seems possible for $D_{2}^{(n)}$ to have a bias of order $\sqrt{|J_n|}g(n)$.

Our main result, Theorem \ref{thm:main}, is stated under two assumptions. 
The second, Assumption \ref{asm:UniformRates}, is a straightforward 
generalization to our setup of similar assumptions in the double machine 
learning literature on rates of convergence for the two estimators used. 
Both estimation errors are measured using a $2$-norm, and it is plausible 
that we can relax one norm to a weaker form of convergence if we simultaneously 
strengthen the other norm. The first assumption, Assumption \ref{asm:UniformBounds},
requires uniform bounds on both $\lambda$ and $G$. This is a strong assumption but 
perhaps not particularly problematic from a practical viewpoint. Indeed, 
$G$ is a process we can choose, and we can thus make it bounded if necessary. 
And though many theoretically interesting counting process models have unbounded 
intensities, a large cap on the intensity will make no difference in 
practice. We believe, nevertheless, that it is possible to relax Assumption \ref{asm:UniformBounds} to a weaker form of control on the 
magnitudes of $\lambda$ and $G$ as functions of time, e.g., moment bounds
uniform in $\theta$. However, such a generalization will come at the  
expense of considerably more technical proofs, and we did not pursue this 
line of research.

A major practical question is whether we can estimate $\lambda$ and $G$ 
with sufficient rates, e.g. $n^{- \frac{1}{4} + \epsilon}$. 
In Appendix \ref{sec:estimation} we give an overview of some known and 
some conjectured rate results for specific forms of $\lambda$ and $\Pi$. 
Beyond parametric models we conclude that the existing rate results are scarce, and we regard it is as an independent research project to establish rates for general historical regression methods.

Another question is whether we can replace the counting process $N$ by a more general 
semimartingale. \cite{commenges2009} define conditional local independence for a 
class of special semimartingales, and \cite{Mogensen:2018} and \cite{Mogensen:2022} 
show global Markov properties for local independence graphs of certain Itô processes,
which are, in particular, special semimartingales.
Thus conditional local independence is well defined beyond counting 
processes, and we believe that most definitions and results of this paper would  
generalize beyond $N$ being a counting process. Besides some additional 
technical challenges, the major practical obstacle with such a generalization 
is that we cannot realistically assume to have completely observed sample paths 
of  Itô processes, say. The discrete time nature of the observations should 
then be included in the analysis, and this is beyond the scope of the present
paper. 

Irrespectively of the remaining open problems, the simulation 
study demonstrated some important properties of our proposed test, the X-LCT. 
First, it was fairly simple to implement for the specific example 
considered using some standard estimation techniques that were not 
tailored to the specific model class. Second, it had good level 
and power properties and clearly outperformed the test based on the 
misspecified marginal Cox model. Third, both Neyman orthogonalization 
as well as cross-fitting were pivotal for achieving
the good properties of the test.

\subsection*{Funding}
The work was supported by Novo Nordisk Foundation Grant NNF20OC0062897.

\bibliographystyle{agsm}
\bibliography{bibliography}

\newpage

\section*{Supplementary material}
In Appendix~\ref{sec:proofs},
    we give the proofs of the results of the paper.
    In Appendix~\ref{app:UniformAsymptotics}, we formulate a general uniform asymptotic theory for metric spaces, whereafter we specialize the theory to the Skorokhod space $D[0,1]$ and chaining of stochastic processes.
    In Appendix~\ref{sec:fclt}, we state Rebolledo's martingale central limit theorem, and then we generalize the result to a uniform version that is used in the proofs.
    In Appendix~\ref{sec:estimation}, we discuss estimation of the intensity $\lambda$ and the residual process $G$ in practice. In particular, we compare known rate results with the rates required in Assumption \ref{asm:UniformRates}.
    Finally, Appendix~\ref{sec:extrafigs} contains additional figures from the simulation study.

\appendix

\section{Proofs} \label{sec:proofs}
This appendix contains proofs of the results stated in the paper.

\subsection{Proof of Proposition \ref{prop:cli-mg}}

The process $G_t$ is \lc{} and $\mathcal{G}_t$-predictable by assumption, and the process
$I = (I_t)$ is a stochastic integral of $G_t$ w.r.t. a local
$\mathcal{G}_t$-martingale under the hypothesis $H_0$. It is thus 
also a local $\mathcal{G}_t$-martingale under $H_0$. By definition, 
$I_0 = 0$, and if $I$ is a martingale, $\gamma_t = \ex(I_t) =
\ex(I_0) = 0.$ \hfill $\square$ 

\subsection{Proof of proposition \ref{prop:gammaalternative}}
Suppose that $H$ is non-negative, \lc{} and $\mathcal{G}_t$-predictable, then since $\int_0^t H_s \mathrm{d}\bM_s$ is a local $\cG_t$-martingale it follows by monotone convergence along a localizing sequence that 
\begin{equation} \label{eq:Gid}
\ex\left(\int_0^t H_{s} \mathrm{d}N_{s}\right) = \ex\left(\int_0^t H_{s}\blambda_s\mathrm{d}s\right) = 
\int_0^t \ex(H_{s}\blambda_s) \mathrm{d}s
\end{equation}
for all $t\in[0,1]$. We can apply the identity 
above with $H$ the positive and negative part of $G$, respectively, 
and the integrability assumption ensures that \eqref{eq:Gid} also holds with $H = G$. It follows that
\begin{align*}
    \gamma_t = \ex (I_t)
    = \ex\left(\int_0^t G_{s}(\blambda_s 
        - \lambda_s)\mathrm{d}s\right) 
    = \int_0^t \ex\pa{G_{s}(\blambda_s 
        - \lambda_s)}\mathrm{d}s.
\end{align*}
The latter expectation is indeed a covariance since $\ex(G_s) = \ex(\ex(G_s\mid\cF_{s-})) = 0$. \hfill $\square$

\subsection{Proof of Lemma \ref{lem:truemartingales}}
Before proving Lemma \ref{lem:truemartingales}, we first state general martingale criteria in the context of counting processes.
\begin{lem}\label{lem:countingmartingales}
    Let $(H_t)$ be a locally bounded $\cG_t$-predictable process, 
    let $N$ be a counting process with a $\cG_t$-intensity $\blambda_t$, and let $\bM_t = N_t - \int_0^t \blambda_s \mathrm{d}s$.
    
    If $\int_0^1 \blambda_s \mathrm{d}s$ 
    (or equivalently $N_1$)
    is integrable, then $\bM_t$ and $\bM_t^2-\int_0^t \blambda_s \mathrm{d}s$ are each $\cG_t$-martingales. 
    If, in addition, $\int_0^1 H_s^2 \blambda_s \mathrm{d}s$ is integrable, then
    $\int_0^t H_s \mathrm{d} \bM_s$
    is a mean zero square integrable martingale.
\end{lem}
\begin{proof}
    The first part is Lemma 2.3.2 and Theorem 2.5.3 in \citet{fleming2011counting}. 
    For the second part, assume that $\int_0^1 \blambda_s \mathrm{d}s$
    and $\int_0^1 H_s^2 \blambda_s \mathrm{d}s$ are both integrable.
    In this case, the $\cG_t$-predictable quadratic variation of $\bM$ is $\langle \bM \rangle(t) = \int_0^t\blambda_s\mathrm{d}s$ by the first part.
    Then it remains to note that $(H_t)$ is a locally bounded $\cG_t$-predictable process, so the conditions of Theorem 2.4.4 in \citet{fleming2011counting} are satisfied if $\int_0^1 H_s^2 \blambda_s \mathrm{d}s$ is integrable. This establishes the second part.
\end{proof}
We now return to the proof of Lemma \ref{lem:truemartingales}. 
Let $f\in C(\mathbb{R})$, and we shall prove that $\int_0^t f(G_s)\mathrm{d}\bM_s$ is a mean zero, square integrable $\mathcal{G}_t$-martingale. The proof for the integral with $f(\hat{G}_s^{(n)})$ is identical. 

Continuity of $f$ implies that $C_f \coloneqq \sup_{x\in [-C',C']}|f(x)| <\infty$ and that $(f(G_t))$ is a $\mathcal{G}_t$-predictable process. By Assumption \ref{asm:UniformBounds}, the process $(f(G_t))$ is almost surely bounded by $C_f$ and therefore
\begin{align*}
    \ex\left(\int_0^1 f(G_s)^2 \blambda_s \mathrm{d}s \right) 
    \leq C_f^2 C < \infty.
\end{align*}
Thus we can apply Lemma \ref{lem:countingmartingales} to conclude that $\int_0^t f(G_s)\mathrm{d}\bM_s$ is a mean zero, square integrable $\mathcal{G}_t$-martingale.
\hfill $\square$

\subsection{Proof of Proposition \ref{prop:UasymptoticGaussian}}

As noted elsewhere, the explicit parametrization of all objects by $\theta$ is 
notationally heavy, and there will thus be an implicit parameter value $\theta \in \Theta$
in most of the subsequent constructions and arguments.

To simplify notation we write
    \begin{align*}
    U^{(n)}_t = \sum_{j \in J_n} \int_0^t H_{j, s}^{(n)} \mathrm{d} \bM_{j, s},
    \quad \text{where} \quad 
    H_{j, s}^{(n)} = \frac{G_{j, s}}{\sqrt{|J_n|}}.
    \end{align*}
We will use a uniform extension of Rebolledo's martingale central limit theorem on the sequence $(U^{(n)})_{n \geq 1}$ to show the result. See Appendix \ref{sec:fclt} for a discussion of Rebolledo's CLT and Theorem \ref{thm:URebo} for its uniform extension.

Define $\tilde{\mathcal{G}}_{t}^n$ be the smallest right continuous and complete filtration generated by the filtrations $\{\mathcal{G}_{j, t} \mid j \in J_n\}$. We can apply \cref{lem:truemartingales} to each of the terms of $U^{(n)}$ to conclude that the $j$-th term is a square integrable, mean zero $\mathcal{G}_{j,t}^n$-martingale.
By independence of the observations for each $j$, we can enlarge the filtration for each term and conclude that they are also square integrable, mean zero $\mathcal{\tilde{G}}_t^n$-martingales. Thus $U^{(n)}$ is also a square integrable, mean zero $\mathcal{\tilde{G}}_t^n$-martingale.

To apply Theorem \ref{thm:URebo} first establish that the conditions in Equation \eqref{eq:RebolledoConditions} are fulfilled. By Proposition \ref{prop:Rebolledospecialcase}, we have that
    \begin{align*}
    \big\langle 
    U^{(n)}
    \big\rangle (t) 
    & = 
    \sum_{j \in J_n} \int_0^t \left( H_{j, s}^{(n)} \right)^2 \blambda_{j, s}
    \mathrm{d} s
    = 
    \frac{1}{|J_n|} \sum_{j \in J_n}
        \int_0^t G_{j, s}^2 \blambda_{j, s} \mathrm{d}s.
    \end{align*}
By directly applying the bounds from Assumption \ref{asm:UniformBounds}, we see that the square mean of $\int_0^t G_s^2 \blambda_s \mathrm{d}s$ is bounded by $C^2(C')^4$. 
Thus, for fixed $t\in[0,1]$, the uniform law of large numbers \citep[Lemma 19]{shah2020hardness} gives that
    \begin{align*}
        \big\langle 
        U^{(n)}
        \big\rangle (t)
        = \frac{1}{|J_n|} \sum_{j \in J_n}
        \int_0^t G_{j, s}^2 \blambda_{j, s} \mathrm{d} s 
        \convUP
        \ex \left( \int_0^t G_{s}^2 \blambda_{s}\mathrm{d} s\right) 
        = \mathcal{V}(t)
    \end{align*}
for $n \to \infty$, since the integrals are i.i.d. with the 
same distribution as $\int_0^t G_{s}^2 \blambda_{s}\mathrm{d} s$. 
This establishes the first part of the condition in Equation \eqref{eq:RebolledoConditions}.
For the second part, we also have from Proposition \ref{prop:Rebolledospecialcase} that
    \begin{align} \label{eq:rarefactionterms}
    \big\langle 
    U_{\epsilon}^{(n)}
    \big\rangle (t) 
    & =
    \sum_{j \in J_n} \int_0^t 
    \left( H_{j, s}^{(n)} \right)^2 \one \left( |H_{j, s}^{(n)}| \geq \epsilon \right)
    \mathrm{d} \bLambda_{j, s}
        \nonumber \\
    & = 
    \frac{1}{|J_n|} \sum_{j \in J_n}
    \int_0^t
    G_{j, s}^2 
    \one \left(
    \left|
    G_{j, s}
    \right| \geq \epsilon \sqrt{|J_n|}
    \right)
    \blambda_{j, s} \mathrm{d} s
    \end{align}
for each $t \in [0,1]$ and $\epsilon > 0$. 
From Assumption \ref{asm:UniformBounds}, we note that 
for $n$ sufficiently large such that $|J_n| > (C')^2/\epsilon^2$, it holds that
\(
    \mathbb{P}
    \left(
        \left|
        G_{j, s}
        \right| \geq \epsilon \sqrt{|J_n|}
    \right) = 0
\)
for all $j\in J_n$. As a consequence, the terms in \eqref{eq:rarefactionterms} are almost surely zero for $n$ sufficiently large uniformly over $\Theta$. It follows that
\(
    \big\langle 
    U_{\epsilon}^{(n)}
    \big\rangle (t) \convUP 0
\), which establishes the second part of \eqref{eq:RebolledoConditions}.

We finally note that the collection of variance functions, $( \mathcal{V}^\theta)_{\theta \in \Theta}$, is uniformly equicontinuous and bounded above under Assumption \ref{asm:UniformBounds}. This is established in Lemma \ref{lem:regularityofgammasigma} below. We have thus verified all the conditions of \cref{thm:URebo}, so we conclude that
    \begin{align*}
    U^{(n),\theta} \convUD U^\theta
    \end{align*}
in $D[0,1]$ as $n\to \infty$, where $U^\theta$ is a mean zero continuous Gaussian martingale with variance function $\mathcal{V}^\theta$. 
\hfill $\square$ \\

Note that the convergence of \eqref{eq:rarefactionterms} is established directly from the uniform bounds in Assumption \ref{asm:UniformBounds}. However, the convergence could also be established under a milder conditions with alternative arguments. For example, under the weaker assumption of uniformly bounded variance functions, dominated convergence can be used to establish $L_1$-convergence. 

In the proof above we invoked the following lemma, which we will also use in several proofs in the sequel.

\begin{lem}\label{lem:regularityofgammasigma}
Under Assumption \ref{asm:UniformBounds}, the collections $(\gamma^\theta)_{\theta\in \Theta}$ and $(\mathcal{V}^\theta)_{\theta \in \Theta}$ are each uniformly Lipschitz and in particular uniformly equicontinuous. Moreover, it holds almost surely that
\begin{align*}
        \sup_{t\in [0,1]} |\gamma_t| \leq 2CC' 
        \qquad \text{and} \qquad
        \mathcal{V}(1) = \ex \left(\int_0^1 G_s^2 \blambda_s \mathrm{d}s\right) \leq C(C')^2.
    \end{align*}
\end{lem}
\begin{proof}
For any $0\leq s < t \leq 1$, a direct application of Assumption~\ref{asm:UniformBounds} and Proposition \ref{lem:truemartingales} yields
\begin{align*}
    |\gamma_t-\gamma_s| 
        \leq \ex \left\lvert
                \int_s^t G_u\mathrm{d}N_u 
                    - \int_s^t G_u\lambda_u \mathrm{d}u
            \right \rvert 
        \leq \ex \left(\int_s^t |G_u|(\blambda_u  + \lambda_u)\mathrm{d}u\right)
        \leq 2CC'(t-s),
\end{align*}
and similarly,
\begin{align*}
    \mathcal{V}(t) - \mathcal{V}(s) 
        = \ex \left(\int_s^t G_u^2 \blambda_u \mathrm{d}u\right)
        \leq C(C')^2(t-s).
\end{align*}
This establishes the first part. The bounds follow from inserting $(s,t)=(0,t)$ in the first inequality and $(s,t)=(0,1)$ in the second inequality. 
\end{proof}

\subsection{Proof of Proposition \ref{prop:remainderterms}}

We will divide the proof into three lemmas for each of the remainder terms $R_1^{(n)}, R_2^{(n)}$ and $R_3^{(n)}$, where we establish convergence to the zero-process uniformly over $t$ and $\theta$. However, note that the notion of uniform convergence differs for the process index, $t\in [0,1]$, and the parameter, $\theta \in \Theta$, as we need to show that

\begin{align*}
        \forall i \in \{1,2,3\}\forall \epsilon>0: \quad 
        \lim_{n\to \infty}
        \sup_{\theta \in \Theta} \mathbb{P}\Big(
            \sup_{t\in[0,1]}|R_{i,t}^\theta|>\epsilon
        \Big) = 0.
\end{align*}

For a general discussion of the relation between weak convergence and convergence in probability uniformly as a stochastic process, see \cite{Newey:1991}. For a general discussion of uniform stochastic convergence over a distribution parameter, see Appendix \ref{app:UniformAsymptotics} and the references contained therein. In Appendix \ref{app:UniformChaining}, we discuss the combination of both convergences.

As in the proof of Proposition \ref{prop:UasymptoticGaussian}, 
$\tilde{\mathcal{G}}_t^n$ denotes the smallest right continuous and complete filtration 
generated by the filtrations $\{\mathcal{G}_{j, t} \mid j \in J_n\}$. 
Analogously, we let $\mathcal{\tilde{G}}_t^{n,c}$ be the smallest 
right continuous and complete filtration generated by the filtrations 
$\{ \mathcal{G}_{j, t} \mid j \in J_n^c \}$.
We start by considering $R_3^{(n)}$, since this is the easiest case. 

\begin{lem} \label{lem:R3}
Under \cref{asm:UniformRates} it holds that $\sup_{t\in [0,1]} |R_{3,t}^{(n)}| \convUP 0$. 
\end{lem}

\begin{proof}
We will show the result by showing that
    \begin{align*}
    \sup_{\theta \in \Theta}
    \ex \left(
        \sup_{0 \leq t \leq 1} |R_{3, t}^{(n),\theta} |
    \right) \to 0
    \end{align*}
as $n \to \infty$. Using that the random variables 
    \begin{align*}
            \sup_{0 \leq t \leq 1} \left|
            G_{j, t} - \hat{G}_{j, t}^{(n)}
            \right|
            \cdot 
            \sup_{0 \leq t \leq 1} \left|
            \lambda_{j, t} - \hat{\lambda}_{j, t}^{(n)}
            \right|
    \end{align*}
for $j \in J_n$ are identically distributed for each fixed $n \geq 2$, we have that
    \begin{align*}
        & \ex \left(
        \sup_{0 \leq t \leq 1} | R_{3, t}^{(n)} |
        \right) 
        \\
        & = 
        \ex \left( \sup_{0 \leq t \leq 1} \left|
        \frac{1}{\sqrt{|J_n|}} 
        \sum_{j \in J_n}
        \int_0^t 
        \left( G_{j, s} - \hat{G}_{j, s}^{(n)} \right) 
        \left( \lambda_{j, s} - \hat{\lambda}_{j, s}^{(n)} \right) 
        \mathrm{d}s \right|
        \right)
        \\
        &   \leq \frac{1}{\sqrt{|J_n|}} 
            \sum_{j \in J_n}
            \ex \left( \sup_{0 \leq t \leq 1} 
            \int_0^t \left|
            G_{j, s} - \hat{G}_{j, s}^{(n)} \right| \cdot
            \left| \lambda_{j, s} - \hat{\lambda}_{j, s}^{(n)} \right|
            \mathrm{d}s
            \right) \\
        &   = \sqrt{|J_n|} \ex \left(  
            \int_0^1 \left|
            G_{s} - \hat{G}_{s}^{(n)} \right| \cdot
            \left| \lambda_{s} - \hat{\lambda}_{s}^{(n)} \right|
            \mathrm{d}s
            \right) \\
        &   \leq \sqrt{|J_n|} \ex \left( 
            \sqrt{ \int_0^1 \left( 
                G_{s} - \hat{G}_{s}^{(n)} \right)^2 \mathrm{d}s}
            \sqrt{ \int_0^1 \left( 
                \lambda_{s} - \hat{\lambda}_{s}^{(n)} \right)^2 \mathrm{d}s}
            \right) \\
        &   \leq \sqrt{|J_n|}
            \sqrt{\ex \left( 
                 \int_0^1 \left( 
                    G_{s} - \hat{G}_{s}^{(n)} \right)^2 \mathrm{d}s\right)}
            \sqrt{\ex \left( \int_0^1 \left( 
                    \lambda_{s} - \hat{\lambda}_{s}^{(n)} \right)^2 
                    \mathrm{d}s\right)} \\
        &   = \sqrt{|J_n|} g(n)h(n).
    \end{align*}
By Assumption \ref{asm:UniformRates}, $\sqrt{|J_n|} g(n)h(n)\to 0$ uniformly over $\Theta$ as $n \to \infty$, so the result follows.
\end{proof}


Next we proceed to the remainder process $R_2^{(n)}$.

\begin{lem} \label{lemma:R2}
Under \cref{asm:UniformBounds,asm:UniformRates}, it holds 
that 
\(
    \sup_{t\in [0,1]} |R_{2,t}^{(n)}| \convUP 0
\).
\end{lem}

\begin{proof}
We first write
    \begin{align*}
        R^{(n)}_{2, t} 
        &= 
            \frac{1}{\sqrt{|J_n|}} \sum_{j \in J_n}  
            \int_0^t \left( G_{j, s} - \hat{G}_{j, s}^{(n)} \right) 
            \mathrm{d} \bM_{j, s}, 
    \end{align*}
and note that $R^{(n)}_{2, t}$ is a square integrable, mean zero $\tilde{\mathcal{G}}_t^n$-martingale conditionally on $\tilde{\mathcal{G}}_1^{n,c}$.
This follows by applying \cref{lem:truemartingales} to each of the terms, which are i.i.d. conditionally on $\tilde{\mathcal{G}}_1^{n,c}$. 
We conclude that the squared process
$(R_{2,t}^{(n)} )^2$ is a $\tilde{\mathcal{G}}_t^n$-submartingale conditionally on $\tilde{\mathcal{G}}_1^{n,c}$. By Doob's
submartingale inequality we have that 
    \begin{align*}
    \mathbb{P} \left(
        \sup_{0 \leq t \leq 1} |  R_{2, t}^{(n)} | \geq \epsilon
    \right)
    & =
    \mathbb{P} \left( 
        \sup_{0 \leq t \leq 1} \left( R_{2, t}^{(n)} \right)^2
        \geq \epsilon^2
    \right) 
    \\
    & = 
    \ex \left( 
    \mathbb{P} \left(
        \sup_{0 \leq t \leq 1} \left(  R_{2, t}^{(n)} \right)^2 \geq \epsilon^2
        \mid
        \tilde{\mathcal{G}}_1^{n,c}
    \right)
    \right)
    \\
    & \leq 
    \frac{\ex \left( \var \left( R_{2, 1}^{(n)} \mid
    \tilde{\mathcal{G}}_1^{n,c} \right) \right)}{\epsilon^2} 
    \end{align*}
for $\epsilon > 0$. The collection of random variables
    \begin{align*}
        \left(
            \int_0^1 \left( G_{j, s} - \hat{G}_{j, s}^{(n)} \right) 
        \mathrm{d} \bM_{j, s}
        \right)_{j \in J_n}
    \end{align*}
are i.i.d. conditionally on $\tilde{\mathcal{G}}_1^{n,c}$. Therefore,
    \begin{align*}
        \var \left( R_{2, 1}^{(n)} \mid \tilde{\mathcal{G}}_1^{n,c} \right)
        & = 
        \frac{1}{|J_n|} \sum_{j \in J_n} \var \left( 
            \int_0^1 \left( G_{j, s} - \hat{G}_{j, s}^{(n)} \right) 
            \mathrm{d} \bM_{j, s} 
            \mid \tilde{\mathcal{G}}_1^{n,c}
        \right) 
        \\
        & = 
        \ex \left( \int_0^1
        \left( G_{s} - \hat{G}_{s}^{(n)} \right)^2 
        \mathrm{d} \langle \bM \rangle_s 
        \mid \tilde{\mathcal{G}}_1^{n,c} \right)
        \\
        & =
        \ex \left(
        \int_0^1 \left( G_{s} - \hat{G}_{s}^{(n)} \right)^2 \blambda_{s}
        \mathrm{d}s 
        \mid \tilde{\mathcal{G}}_1^{n,c}
        \right) \\
        & \leq C \cdot \ex \left(
                \int_0^1 \left( G_{s} - \hat{G}_{s}^{(n)} \right)^2
                \mathrm{d}s 
                \mid \tilde{\mathcal{G}}_1^{n,c}
                \right)
    \end{align*}
where we have used that $\blambda_t$ is bounded by \cref{asm:UniformBounds} (i). Thus 
\begin{align*}
    \ex \left(
        \var \left( R_{2, 1}^{(n)} \mid \tilde{\mathcal{G}}_1^{n,c} \right)
    \right) 
    \leq 
    C \cdot \ex \left(
    \int_0^1 \left( G_{s} - \hat{G}_{s}^{(n)} \right)^2 \mathrm{d}s
    \right) 
    = C \cdot g(n)^2,
\end{align*}
and we conclude that 
    \begin{align*}
        \mathbb{P} \left(
        \sup_{0 \leq t \leq 1} | R_{2, t}^{(n)} | \geq \epsilon
    \right)
    \leq 
    \frac{C\cdot g(n)^2}{\epsilon^2} \to 0,
    \end{align*}
as $n \to \infty$ uniformly over $\Theta$ by \cref{asm:UniformRates}.
\end{proof}

Before proving that $R_1^{(n)}$ converges weakly to the zero-process, we will need two auxiliary lemmas. The first is a conditional version of Hoeffding's lemma, which lets us conclude conditional sub-Gaussianity. Recall that a mean zero random variable $A$ is sub-Gaussian with variance factor $\nu>0$ if 
    \begin{align*}
        \log \ex(e^{xA}) \leq \frac{x^2 \nu}{2}
    \end{align*}
for all $x \in \mathbb{R}$. See, for example, \citet{boucheron2013concentration}, Lemma 2.2, for the classical unconditional version.

\begin{lem}[conditional Hoeffding's lemma]\label{lem:conditionalHoeffding}
    Let $Y$ be a random variable taking values on a bounded interval $[a,b]$, satisfying $\ex[Y|\mathcal{G}]=0$ for a $\sigma$-algebra $\mathcal{G}$. 
    
    Then $\log \ex(e^{xY}\mid\mathcal{G}) \leq (b-a)^2 x^2 /8$ almost surely for all $x\in \mathbb{R}$.
\end{lem}
\begin{proof}
    Fix $x\in \mathbb{R}$. By convexity of the exponential function we have
    \begin{align*}
        e^{xy} \leq \frac{b-y}{b-a} e^{xa} + \frac{y-a}{b-a} e^{xb},
        \qquad y\in[a,b].
    \end{align*}
    Inserting $Y$ in place of $y$ and taking the conditional expectation yields
    \begin{align*}
        \ex[e^{xY}\mid\mathcal{G}]
        \leq \frac{b}{b-a} e^{xa} - \frac{a}{b-a} e^{xb}
        = e^{L(x(b-a))}
    \end{align*}
    almost surely,
    where $L(h) = \frac{ha}{b-a} + \log(1+\frac{a-e^ha}{b-a})$.
    Standard calculations show that $L(0)=L'(0)=0$, and the AM-GM inequality implies
    \begin{align*}
        L''(h) = - \frac{abe^h}{(b-ae^h)^2} \leq \frac{1}{4}.
    \end{align*}
    Thus, a second order Taylor expansion yields that $L(h)\leq \frac{1}{8}h^2$, and it follows that
    $\log \ex[e^{xY}\mid\mathcal{G}] \leq 
    \frac{(b-a)^2}{8}x^2$ as desired.  
\end{proof}

For the next lemma define for $s, t \in [0,1]$ with $s < t$ 
    \begin{align*}
        W^{s, t} = \frac{1}{t - s} \int_s^t G_u (\lambda_u - \hat{\lambda}^{(n)}_u)
        \mathrm{d}u. 
    \end{align*}

\begin{lem} \label{lemma:subgaussian}
Let \cref{asm:UniformBounds} hold true. Then, for any $0\leq s < t \leq 1$, it holds that $\ex(W^{s,t}\mid\tilde{\mathcal{G}}_1^{n,c}) = 0$ and 
that $W^{s,t}$ is sub-Gaussian conditionally on $\tilde{\mathcal{G}}_1^{n,c}$ with variance factor $\nu = (2CC')^2$, that is,
    \begin{align*}
        \log \ex (e^{xW^{s, t}} \mid \tilde{\mathcal{G}}_1^{n,c}) 
        \leq 2 (xCC')^2
    \end{align*}
for all $s < t$ and $x \in \mathbb{R}$. 
\end{lem}

\begin{proof}
For fixed $u\in [0,1]$, note that
\begin{align} \label{eq:R1meanzero}
    \ex\left( G_u
    \left( 
    \lambda_u - \hat{\lambda}^{(n)}_u \right) 
    \mid
    \tilde{\mathcal{G}}_1^{n,c}
    \right) 
    & = 
    \ex\left( 
    \ex\left( 
    G_u
    \left( 
    \lambda_u - \hat{\lambda}^{(n)}_u \right) 
    \mid 
    \mathcal{F}_{s-} \vee \tilde{\mathcal{G}}_1^{n,c}
    \right) 
    \mid
    \tilde{\mathcal{G}}_1^{n,c}
    \right) \nonumber \\
    & = 
    \ex\left(
    \ex\left( G_u \mid \mathcal{F}_{s-} \right)
     \left( 
        \lambda_u - \hat{\lambda}^{(n)}_u \right) 
    \mid 
    \tilde{\mathcal{G}}_1^{n,c}
    \right) = 0,
    \end{align}
where we have used that $\lambda_{t} - \hat{\lambda}^{(n)}_{t}$ is 
$\mathcal{F}_{t}$-predictable conditionally on $\tilde{\mathcal{G}}_1^{n,c}$, 
that $G_{t}$ is independent of $\tilde{\mathcal{G}}_1^{n,c}$ since it is $\mathcal{G}_{t}$-predictable, and that 
$\ex\left( G_{s} \mid \mathcal{F}_{s-} \right) = 0$ per definition. 
By applying the conditional Fubini theorem \citep[Theorem 27.17]{schilling2017measures}, we conclude that $\ex(W^{s,t}\mid\tilde{\mathcal{G}}_1^{n,c}) = 0$.

We can now use the conditional version of Hoeffding's lemma formulated in Lemma \ref{lem:conditionalHoeffding}. 
Indeed, we have that for all $s<t$
    \begin{align*}
        |W^{s, t}| & \leq \frac{1}{t - s} \int_s^t | G_u |
        | (\lambda_u - \hat{\lambda}^{(n)}_u) |
        \mathrm{d}u
        \\
        & \leq \sup_{0 \leq u \leq 1}|G_u|
        \sup_{0 \leq u \leq 1}|(\lambda_u - \hat{\lambda}^{(n)}_u)|
        \leq 2 C C' 
    \end{align*}
by \cref{asm:UniformBounds}. Hence, for all $s < t$, Lemma \ref{lem:conditionalHoeffding} lets us conclude that 

\(
        \log \ex (e^{xW^{s, t}} \mid \tilde{\mathcal{G}}_1^{n,c}) 
        \leq 2 (xCC')^2
\), 
$x \in \mathbb{R}$.
\end{proof}

Then we have the following regarding $R_1^{(n)}$.

\begin{lem} \label{lem:R1}
Under \cref{asm:UniformRates,asm:UniformBounds} it holds that $\sup_{t\in [0,1]} |R_{1,t}^{(n)}| \convUP 0$. 
\end{lem}

\begin{proof}
The proof consists of two parts. First we show that for each $t \in [0,1]$
it holds that
    \begin{align*}
    R_{1,t}^{(n)} \convUP 0
    \end{align*}
for $n \to \infty$. Then we show 
\emph{stochastic equicontinuity} of the process 
$R_{1}^{(n)}$ uniformly over $\Theta$, and by Lemma \ref{lem:uniformequcont} it follows
that 
    \begin{align*}
    \sup_{t \in [0,1]} |R_{1,t}^{(n)}| \convUP 0.
    \end{align*}
This is a direct generalization of Theorem 2.1 in \cite{Newey:1991}.
The collection of random variables 
    \begin{align*}
    \left(
        G_{j,s}
        \left( \lambda_{j, s} - \hat{\lambda}^{(n)}_{j, s} \right)
    \right)_{j \in J_n}
    \end{align*}
are i.i.d. conditionally on $\tilde{\mathcal{G}}_1^{n,c}$. Therefore, an application of the conditional Fubini theorem yields
    \begin{align*}
    \ex( R_{1, t} \mid \tilde{\mathcal{G}}_1^{n,c}) 
    & = 
    \frac{1}{\sqrt{|J_n|}} \sum_{j \in J_n} \int_0^t 
    \ex\left( 
    G_{j,s}
    \left( \lambda_{j, s} - \hat{\lambda}^{(n)}_{j, s} \right) 
    \mid 
    \tilde{\mathcal{G}}_1^{n,c}
    \right)
    \mathrm{d}s
    = 0
    \end{align*}
where the last equality follows from the computation in \eqref{eq:R1meanzero}.
Whence $\ex(R^{(n)}_{1, t}) = 0$, and 
$\var(R^{(n)}_{1, t}) = \ex(\var(R^{(n)}_{1, t} \mid \tilde{\mathcal{G}}_1^{n,c}))$, so
    \begin{align*}
    \var(R_{1, t}^{(n)}) 
    & =
    \ex \left( 
    \frac{1}{|J_n|} \sum_{j \in J_n} \var \left( 
        \int_0^t 
        G_{j,s}
        \left( \lambda_{j, s} - \hat{\lambda}^{(n)}_{j, s} \right) \mathrm{d}s 
        \mid 
        \tilde{\mathcal{G}}_1^{n,c}
    \right)
    \right)  
    \\
    & =
    \ex \left(
    \ex \left(
    \left(
        \int_0^t 
        G_{s}
        \left( \lambda_{s} - \hat{\lambda}^{(n)}_{s} \right) \mathrm{d}s 
    \right)^2
    \mid 
    \tilde{\mathcal{G}}_1^{n,c}
    \right)
    \right)
    \\
    & = 
    \ex \left(
    \left(
        \int_0^t 
        G_{s}
        \left( \lambda_{s} - \hat{\lambda}^{(n)}_{s} \right) \mathrm{d}s 
    \right)^2
    \right)
    \\
    & \leq (C')^2 \ex \left( \int_0^t 
        \left( \lambda_{s} - \hat{\lambda}^{(n)}_{s} \right)^2 
            \mathrm{d}s \right)\\
    & \leq (C')^2 h(n)^2
    \end{align*}
where we have used \cref{asm:UniformBounds} (ii). Hence by Chebychev's inequality, it holds for all $\epsilon>0$ that
    \begin{align*}
    \mathbb{P}(|R^{(n)}_{1, t}| > \epsilon)
    \leq 
    \frac{\var(R^{(n)}_{1, t})}{\epsilon^2} 
    \leq \frac{(C')^2 h(n)^2}{\epsilon^2}  \longrightarrow 0
    \end{align*}
as $n \to \infty$ uniformly over $\Theta$ by \cref{asm:UniformRates}. This completes the first part of the proof. 
For the second part, we use a chaining argument based on the exponential inequality in 
\cref{lemma:subgaussian}. We let 
    \begin{align*}
    W^{s, t}_j = \frac{1}{t - s} \int_s^t G_{j, u}
    \left(
    \lambda_{j, u} - \hat{\lambda}_{j, u}^{(n)}
    \right) \mathrm{d}u
    \end{align*}
and 
    \begin{align*}
        A
        =  \frac{1}{\sqrt{|J_n|}} 
            \sum_{j \in J_n} W_j^{s, t}
        = 
        \frac{1}{t - s} (R_{1, t}^{(n)} - R_{1, s}^{(n)}).
    \end{align*}
Using that $(W_j^{s, t})_{j \in J_n}$ are i.i.d. conditionally on $\tilde{\mathcal{G}}_1^{n,c}$ 
we have by \cref{lemma:subgaussian} that
$\ex(A)=0$ and that
    \begin{align*}
    \log \ex \left( e^{x A} \right)
    & = 
    \log \ex \left( \ex \left( e^{x A} \mid \tilde{\mathcal{G}}_1^{n,c} \right) \right) 
    \\
    & = \log  \ex \left( \prod_{j \in J_n}
    \ex\left( e^{\frac{x}{\sqrt{|J_n|}} W_j^{s, t}} 
    \mid 
    \tilde{\mathcal{G}}_1^{n,c}
    \right)
    \right) 
    \\
    & \leq \log  \ex \left( e^{\frac{x^2 \nu}{2}}
    \right) 
    \\
    & = \frac{x^2 \nu}{2}.
    \end{align*}
Hence $A$ is also sub-Gaussian with variance factor $\nu$. This implies that
    \begin{align*}
        \mathbb{P}(|A| > \eta) \leq 2 e^{- \frac{\eta^2 \nu}{2}}
    \end{align*}
for all $\eta > 0$. Rephrased in terms of $R_1^{(n)}$ this bound reads
    \begin{align*}
    \mathbb{P} \left(
    |R_{1, t}^{(n)} - R_{1, s}^{(n)}| > \eta (t - s)
    \right) \leq 2 e^{-\frac{\eta^2 \nu}{2}}
    \end{align*}
for all $\eta > 0$ and $s < t$. It now follows from the chaining lemma, \cite{Pollard:1984} Lemma 
VII.9, that $R_1^{(n)}$ is stochastic equicontinuous. 
Since the variance factor $\nu=(2CC')^2$ does not depend on $\theta \in \Theta$, we have stochastic equicontinuity uniformly over $\Theta$ by Corollary \ref{cor:uniformchaining}. 
This completes the second part of the proof and we are done.
\end{proof}

Note that the second part of the proof above establishes stochastic 
equicontinuity by a bound on the probability that the increments 
of the process are large. This is a well known technique, see, e.g.,
Example 2.2.12 in \cite{Vaart:1996}, from which the same conclusion 
will follow if 
$$\ex(|R_{1, t}^{(n)} - R_{1, s}^{(n)}|^p) \leq K |t - s|^{1 + r}$$
for $K,p,r > 0$.

\cref{prop:remainderterms} now follows from combining the Lemmas \ref{lem:R1}, \ref{lemma:R2}, and \ref{lem:R3}.
\hfill \qed{}

\subsection{Proof of Proposition \ref{prop:Dasymptotics}}
We separate the discussion of $D_1^{(n)}$ and $D_2^{(n)}$ into the Lemmas \ref{lem:D1asymptotics} and \ref{lem:D2asymptotics}, respectively, which together amount to Proposition \ref{prop:Dasymptotics}.

\begin{lem}\label{lem:D1asymptotics}
    Suppose that Assumptions \ref{asm:UniformBounds} and \ref{asm:UniformRates} hold. Then the stochastic process
    $\overline{D}^{(n)} \coloneqq D_1^{(n)}-\sqrt{|J_n|} \cdot \gamma$
    converges in distribution in $C[0,1]$ uniformly over $\Theta$.
\end{lem}
\begin{proof}
Let $\overline{D}^{(n)} \coloneqq D_1^{(n)}-\sqrt{|J_n|} \cdot \gamma$ and note that 
$$
    \overline{D}^{(n)} = |J_n|^{-\frac{1}{2}} \sum_{j\in J_n} W_j,
$$
where $W_j$ is given by $W_{j,t} \coloneqq \int_0^t G_{j,s} (\blambda_{j,s} - \lambda_{j,s})\mathrm{d}s - \gamma_t$ for each $j\in J_n$. By assumption, the variables $\{W_j\colon j\in J_n\}$ are i.i.d. with the same distribution as the process $W$ given by $W_{t} \coloneqq \int_0^t G_{s} (\blambda_{s} - \lambda_{s})\mathrm{d}s - \gamma_t$.
For each $\theta\in \Theta$, let $\Gamma^\theta$ be a Gaussian process with mean zero and covariance function $(s,t)\mapsto \cov(W_s^\theta,W_t^\theta)$, which is well-defined by computations shown below.

We will show that $\overline{D}^{(n),\theta} \convUD \Gamma^\theta$ in $C[0,1]$ by applying Lemma \ref{lem:ProkhorovsPrinciple}, which is an example of Prokhorov's method of "tightness + identification of limit".
We first prove that for any given $k\in \mathbb{N}$ and $0\leq t_1<t_2<\cdots<t_k\leq 1$,
\begin{align*}
    \mathbf{D}^{(n)} \coloneqq
    (\overline{D}_{t_1}^{(n)},\overline{D}_{t_2}^{(n)},\ldots, \overline{D}_{t_k}^{(n)})
    \convUD
    (\Gamma_{t_1}^\theta,\Gamma_{t_2}^\theta,\ldots, \Gamma_{t_k}^\theta).
\end{align*}
To this end we will apply the uniform CLT of \citet[Proposition 19]{lundborg2021conditional}
to the sequence of random vectors $\mathbf{D}^{(n)}\in \mathbb{R}^k$, i.e., the sequence of normalized sums of i.i.d. copies of $\mathbf{W} \coloneqq (W_{t_1},\ldots , W_{t_k})$.
The process $(W_t)$ is mean zero and hence $\mathbf{W}$ is also mean zero. For any $t\in [0,1]$ we observe that
\begin{align*}
    \var(W_{t}) = \var(W_{t}+\gamma_t)
    \leq \ex\left[\pa{\int_0^t |G_s|\cdot |\blambda_s - \lambda_s|\mathrm{d}s}^2
        \right]
    \leq 2C^2(C')^2.
\end{align*}
Therefore the trace of $\var(\mathbf{W})$ is uniformly bounded, which is implies the trace condition in Proposition 19 of \citet{lundborg2021conditional}. 
From Hölder's inequality and Minkowski's inequality, we note that for any $\mathbf{a},\mathbf{b}\in \mathbb{R}^k$
$$
    \|\mathbf{a} + \mathbf{b}\|_2^3 
        \leq k^{3/2}\|\mathbf{a} + \mathbf{b}\|_3^3 
        \leq k^{3/2} (\|\mathbf{a}\|_3 + \|\mathbf{b}\|_3)^3 
        \leq 8 k^{3/2} (\|\mathbf{a}\|_3^3 + \|\mathbf{b}\|_3^3).
$$
Combining the above with Assumption \ref{asm:UniformBounds} and Lemma \ref{lem:regularityofgammasigma} yields that 
$$
    \ex[\|\mathbf{W}\|_2^3] 
    \leq 
    C_k \ex\left[\pa{\int_0^1 |G_s|\cdot|\blambda_s - \lambda_s|\mathrm{d}s}^3\right]
    + C_k \sup_{t\in[0,1]} |\gamma_t|^3
    \leq 16 C_k C^3(C')^3,
$$
where $C_k=8k^{5/2}$. Hence Proposition 19 of \citet{lundborg2021conditional} lets us conclude that $\mathbf{D}^{(n)} \convUD \mathcal{N}(0,\var(\mathbf{W}))$. By definition of $\Gamma^\theta$, this is equivalent to 
$\mathbf{D}^{(n)} \convUD (\Gamma_{t_1}^\theta,\Gamma_{t_2}^\theta,\ldots, \Gamma_{t_k}^\theta)$.
 
We now argue that $(\overline{D}^{(n)})$ and $(\Gamma^\theta)$ are stochastically equicontinuous uniformly over $\Theta$. From the definition of $\Gamma^\theta$ and by Assumption \ref{asm:UniformBounds}, it follows that 
\begin{align} \label{eq:squareincrement}
    \ex[(\Gamma_t^\theta-\Gamma_s^\theta)^2] 
    = \ex[(W_t-W_s)^2] 
    \leq (2CC' (t-s))^2.
\end{align}
Hence $\frac{1}{t-s}(\Gamma_t^\theta-\Gamma_s^\theta)$ is Gaussian with a variance bounded over $\Theta$ and $0\leq s<t\leq 1$. In particular, it is sub-Gaussian with a uniform variance factor over $\Theta$ and $0\leq s<t\leq 1$.
Since $W$ is uniformly bounded over $\Theta$, an application of Hoeffding's Lemma yields that $A_j^{t,s}\coloneqq \frac{1}{t-s}(W_{j,t}-W_{j,s})$ is also sub-Gaussian with a variance factor $\nu$ that is uniform over $\Theta$, $0\leq s<t\leq 1$, and $j\in J_n$. Letting $A_\bullet^{s,t} =\frac{1}{t-s}(\overline{D}_t^{(n)}-\overline{D}_s^{(n)})$, we have
\begin{align*}
    \ex e^{x A_\bullet^{s,t}} 
    =   \prod_{j\in J_n}  
        \ex[e^{x|J_n|^{-1/2}A_j^{s,t}}]
    \leq \prod_{j\in J_n} e^{\frac{x^2\nu}{2|J_n|}} 
    = e^{x^2\nu/2}.
\end{align*}
Hence $A_\bullet^{s,t}$ is also sub-Gaussian with a variance factor uniformly over $\Theta$ and $0\leq s<t\leq 1$.

From the uniform chaining lemma, Corollary \ref{cor:uniformchaining}, we now conclude that both $(\Gamma^\theta)$ and $(\overline{D}^{(n)})$ are stochastically equicontinuous uniformly over $\Theta$. By Proposition~\ref{prop:equicontinuityistightness}, this means that the collection $(\overline{D}^{(n),\theta})$ is sequentially tight
and that $(\Gamma^\theta)$, which is constant in $n$, is uniformly tight. 

Now we have shown convergence of the finite-dimensional marginals and appropriate tightness conditions, so Lemma \ref{lem:ProkhorovsPrinciple} lets us conclude that $\overline{D}^{(n)} \convUD \Gamma^\theta$ weakly in $C[0,1]$.
\end{proof}

Before moving on to the term $D_2^{(n)}$, we first note that Lemma \ref{lem:D1asymptotics} implies that stochastic boundedness, as we will use this result in the proof of Theorem \ref{thm:main}.

\begin{lem}\label{lem:D1bounded}
    Suppose that Assumptions \ref{asm:UniformBounds} and \ref{asm:UniformRates} hold.
    Then $\overline{D}^{(n)} \coloneqq D_1^{(n)}-\sqrt{|J_n|} \cdot \gamma$ is stochastically bounded uniformly over $\Theta$, i.e., 
    for every $\varepsilon>0$ there exists $K>0$ such that
    \begin{align*}
        \limsup_{n\to \infty}\sup_{\theta\in \Theta} \mathbb{P}\pa{
        \|\overline{D}^{(n),\theta}\|_\infty > K} < \varepsilon.
    \end{align*}
\end{lem}
\begin{proof}
    We have established in the proof of Lemma \ref{lem:D1asymptotics}, under the same conditions, that $\overline{D}^{(n),\theta} \convUD \Gamma^\theta$ weakly in $C[0,1]$.
    By the uniform continuous mapping theorem formulated in Proposition \ref{prop:Uctsmapping}, it follows that $\|\overline{D}^{(n),\theta}\|_\infty \convUD \|\Gamma^\theta\|_\infty$.
    From \citet{bengs2019uniform} Theorem 4.1 we then obtain that
    \begin{align*}
        \limsup_{n\to \infty} \sup_{\theta \in \Theta}
            \mathbb{P}( \|\overline{D}^{(n),\theta}\|_\infty > K) 
        \leq  \sup_{\theta \in \Theta}\mathbb{P}(\|\Gamma^\theta\|_\infty > K) 
        \leq \frac{\ex\|\Gamma^\theta\|_\infty}{K}.
    \end{align*}
    Hence it suffices to argue that $\ex \|\Gamma^\theta\|_\infty$ is uniformly bounded over $\Theta$.
    To this end, we note that Equation \eqref{eq:squareincrement} 
    shows that square means of the increments of $\Gamma^\theta$
    are smaller that those of a standard Brownian motion scaled by $2CC'$. Then the Sudakov–Fernique comparison inequality \citep[Theorem 2.2.3]{adler2007random} allows us to leverage this relationship to the expected uniform norms, i.e., $\ex\|\Gamma^\theta\|_\infty \leq 2CC' \ex (\sup_{t\in[0,1]}|B_t|)$. It can be verified that $\ex (\sup_{t\in[0,1]}|B_t|)$ is finite, and in fact, equal to $\sqrt{\pi/2}$ as shown in \citet{saz2019expectedsup}. 
\end{proof}

\begin{lem}\label{lem:D2asymptotics}
    Suppose that Assumptions \ref{asm:UniformBounds} and \ref{asm:UniformRates}
    hold, and that $G_t = X_t - \Pi_t$ is the additive residual process. 
    Then $D_2^{(n)}\convUP 0$ in $D[0,1]$ as $n\to \infty$.
\end{lem}
\begin{proof}
Note first that the terms in $D_2^{(n)}$ are i.i.d. conditionally on $\tilde{\cG}_1^{n,c}$, with the same distribution as the process $\xi$ given by
\begin{align*}
    \xi_t = \frac{1}{\sqrt{|J_n|}}\int_0^t (\hat{G}_s^{(n)} - G_s)(\blambda_s - \lambda_s)\mathrm{d}s.
\end{align*}
Since $\blambda_t$ is independent of $\tilde{\cG}_1^{n,c}$, we have from the innovation theorem that
$$
    \ex(\blambda_t \mid \mathcal{F}_{t-}\vee \tilde{\cG}_1^{n,c})
    = \ex(\blambda_t \mid \mathcal{F}_{t-}) = \lambda_t.
$$
For the additive residual process we also note that
$G_t-\hat{G}_t^{(n)} = \hat{\Pi}_t^{(n)} - \Pi_t$ is $\mathcal{F}_{t}$-predictable conditionally on $\tilde{\cG}_1^{n,c}$. It now follows that
\begin{align*}
    \sqrt{|J_n|} \cdot \ex[\xi_t\mid\tilde{\cG}_1^{n,c}]
    &= 
        \int_0^t \ex[(\hat{G}_s^{(n)} - G_s)(\blambda_s - \lambda_s)\mid\tilde{\cG}^{n,c}]\mathrm{d}s \\
    &= 
        \int_0^t \ex[(\hat{G}_s^{(n)} - G_s)(\ex[\blambda_s\mid\mathcal{F}_{s-}\vee \tilde{\cG}^{n,c}]- \lambda_s)\mid\tilde{\cG}^{n,c}]\mathrm{d}s
    = 0.
\end{align*}
We can therefore conclude that $D_2^{(n)}$ is mean zero conditionally on $\tilde{\cG}^{n,c}$. 
Using that the terms of $D_2^{(n)}$ are i.i.d. conditionally on $\tilde{\cG}^{n,c}$ once more, we now obtain that
\begin{align*}
    \var(D_{2,t}^{(n)}\mid\tilde{\cG}^{n,c}) 
    = |J_n| \cdot \var( \xi_t \mid\tilde{\cG}^{n,c} )
    &= \ex\left[\pa{
        \int_0^t (\hat{G}_s^{(n)} - G_s)(\blambda_s - \lambda_s)\mathrm{d}s
        }^2 \Big|\tilde{\cG}^{n,c} \right] \\
    &\leq 4C^2 \cdot \ex\left(
        \int_0^1 (\hat{G}_s^{(n)} - G_s)^2 \mathrm{d}s
        \Big|\tilde{\cG}^{n,c} 
        \right).
\end{align*}
Taking expectation of the above we have 
$\var(D_{2,t}^{(n)}) = \ex(\var(D_{2,t}^{(n)}\mid\tilde{\cG}^{n,c}))
\leq 4 C^2 g(n)^2$.
By Chebyshev's inequality we get for all $\epsilon>0$
\begin{align*}
    \mathbb{P}\pa{|D_{2,t}^{(n)}|>\epsilon} 
        \leq \frac{4C^2 g(n)^2}{\epsilon^2},
\end{align*}
and by Assumption \ref{asm:UniformRates} we conclude that $D_{2,t}^{(n)} \convUP 0$ for each $t\in [0,1]$. 

We know apply the same chaining argument used in the proofs of Lemma \ref{lem:R1} and Lemma \ref{lem:D1asymptotics}. From Assumption \ref{asm:UniformBounds}, we have for $0\leq s<t\leq 1$ that $|\xi_t - \xi_s| \leq 4 \sqrt{|J_n|}CC' (t-s)$. Hence the conditional Hoeffding's lemma (Lemma \ref{lem:conditionalHoeffding}) yields that
$$
    A_j^{s,t} 
        = \frac{1}{t-s}
        \int_s^t (\hat{G}_{j,s}^{(n)} - G_{j,s})(\blambda_{j,s} - \lambda_{j,s})\mathrm{d}s
$$
is sub-Gaussian conditionally on $\tilde{\cG}_1^{n,c}$ with a variance factor $\nu$ that is uniform over $\Theta$ and $s<t$ (cf. the proof of Lemma \ref{lemma:subgaussian}). Letting $A_\bullet^{s,t} =\frac{1}{t-s}(D_{2,t}^{(n)}-D_{2,s}^{(n)})$, we have for any $x\in \mathbb{R}$
\begin{align*}
    \ex \left(e^{x A_\bullet^{s,t}} \right)
    &= \ex\left(\ex\left[e^{x A_\bullet^{s,t}}
        \mid\tilde{\cG}_1^{n,c}\right]\right) \\
    &= \ex\bigg(\prod_{j\in J_n}  
            \ex\left[e^{x|J_n|^{-1/2}A_j^{s,t}}\mid\tilde{\cG}_1^{n,c}
        \right]\bigg)
    \leq \prod_{j\in J_n} e^{\frac{x^2\nu}{2|J_n|}} 
    = e^{x^2\nu/2},
\end{align*}
so $A_\bullet^{s,t}$ is also sub-Gaussian uniformly over $s<t$ and 
$\Theta$. In terms of $D_2^{(n)}$, this means that we can apply the uniform chaining lemma, Corollary \ref{cor:uniformchaining}, and conclude that it is stochastically equicontinuous uniformly over $\Theta$. 

Since $D_{2,t}^{(n)} \convUP 0$ for each $t\in [0,1]$ and $(D_2^{(n)})$ is stochastically equicontinuous uniformly over $\Theta$, Lemma \ref{lem:uniformequcont} now lets us conclude that $\sup_{t\in [0,1]} |D_{2,t}^{(n)}| \convUP 0$ and we are done.
\end{proof}

\subsection{Proof of Theorem \ref{thm:main}}
Before proving Theorem \ref{thm:main}, we first prove that the collection of Gaussian martingales from Proposition \ref{prop:UasymptoticGaussian} is tight in $C[0,1]$ (see Definition \ref{dfn:tight}). 
\begin{lem}\label{lem:Ulimitistight}
    Let $(U^\theta)_{\theta \in \Theta}$ be the collection of Gaussian martingales from Proposition \ref{prop:UasymptoticGaussian}, i.e., $U^\theta$ is a mean zero continuous Gaussian martingale with variance function $\mathcal{V}^\theta$. 
    Under Assumption \ref{asm:UniformBounds}, $(U^\theta)_{\theta \in \Theta}$ is uniformly tight in $C[0,1]$.
\end{lem}
\begin{proof}
We will use Theorem 7.3 in \citet{billingsley2013convergence}, which characterizes tightness of measures in $C[0,1]$. The first condition of the theorem is trivially satisfied for $(U^\theta)_{\theta \in \Theta}$ since $\mathbb{P}(U_0^\theta = 0)=1$ for all $\theta \in \Theta$. 

By Proposition \ref{prop:BMrepresentation}, $U^\theta$ has a distributional representation as a time-transformed Brownian motion such that $(U_t^\theta)_{t\in[0,1]} \overset{\mathcal{D}}{=} (B_{\mathcal{V}^\theta(t)})_{t\in [0,1]}$, where $B$ is a Brownian motion.
Recall that Brownian motion is $\alpha$-Hölder continuous for $\alpha \in (0,\frac12)$, which means 
that 
$$K(\alpha) = \sup_{s \neq t} \frac{|B_t - B_s|}{|t - s|^\alpha} < \infty.$$
Note also that the collection of variance functions is uniformly Lipschitz by Lemma \ref{lem:regularityofgammasigma} with uniform Lipschitz constant $C_0$, say. It follows that for every $\epsilon>0$,
\begin{align*}
    \lim_{\delta\to 0^+} \sup_{\theta \in \Theta}
    \mathbb{P}\Big(\sup_{|t-s|<\delta}|U_t^\theta 
        - U_s^\theta|>\epsilon \Big)
    &= \lim_{\delta\to 0^+} \sup_{\theta \in \Theta}\mathbb{P}\Big(\sup_{|t-s|<\delta}|B_{\mathcal{V}^\theta(t)} 
        - B_{\mathcal{V}^\theta(s)}|>\epsilon \Big)\\
    &\leq 
    \lim_{\delta\to 0^+} \sup_{\theta \in \Theta}\mathbb{P}\Big(K(\alpha) \sup_{|t-s|<\delta}|\mathcal{V}^\theta(t) 
        - \mathcal{V}^\theta(s)|^\alpha >\epsilon \Big) \\
    &= \lim_{\delta\to 0^+} \mathbb{P}\Big(K(\alpha) C_0^\alpha \delta^\alpha >\epsilon \Big) = 0.
\end{align*}
This establishes the second condition of Theorem 7.3 in \citet{billingsley2013convergence}, and we thus conclude that $(U^\theta)_{\theta \in \Theta}$ is uniformly tight in $C[0,1]$. 
\end{proof}

We now return to the proof of Theorem \ref{thm:main}.

For part i), we first note that under $H_0$ we can take $\blambda_t = \lambda_t$,
which implies that both $D_1^{(n)}$ and $D_2^{(n)}$ equal the zero-process. 

Combining Propositions \ref{prop:UasymptoticGaussian} and \ref{prop:remainderterms} with the uniform version of Slutsky's theorem formulated in Lemma \ref{lem:GeneralSlutsky}, we conclude that
\begin{align*}
    \sqrt{|J_n|}\hat{\gamma}^{(n)}
        = 
        \underbrace{U^{(n)}}_{\convUDnull U^\theta} 
        + \underbrace{R_1^{(n)}+R_2^{(n)}+R_3^{(n)}}_{
            \convUP 0}
        + \underbrace{D_1^{(n)}+D_2^{(n)}}_{
            =0 \, \text{under } H_0}
        \convUPnull U^\theta,
\end{align*}
in $D[0,1]$ as $n\to \infty$, where $U^\theta$ is the Gaussian martingale from Proposition \ref{prop:UasymptoticGaussian}.

For part ii) we can, in addition to Propositions \ref{prop:UasymptoticGaussian} and \ref{prop:remainderterms}, apply Proposition \ref{prop:Dasymptotics} and Lemma~\ref{lem:D1bounded}. Using the triangle inequality on the decomposition \eqref{eq:decomposition} yields that
\begin{align*}
    \sqrt{|J_n|}\cdot \|\hat{\gamma}^{(n)} - \gamma\|_\infty
        \leq
        &\|U^{(n)}\|_\infty + \|D_1^{(n)} - \sqrt{|J_n|}\gamma \|_\infty \\
            &+ \|R_1^{(n)}\|_\infty 
            + \|R_2^{(n)}\|_\infty 
            + \|R_3^{(n)}\|_\infty 
            + \|D_2^{(n)}\|_\infty.
\end{align*}
All the terms in the second line converge in probability to zero uniformly over $\Theta$. Combined with the convergences established in Proposition \ref{prop:UasymptoticGaussian} and Lemma \ref{lem:D1asymptotics}, we obtain that
\begin{align}\label{eq:LCMboundednes}
    &\limsup_{n\to \infty}\sup_{\theta\in \Theta} \mathbb{P}\pa{
    \sqrt{|J_n|} \cdot \|\hat{\gamma}^{(n),\theta} - \gamma^\theta\|_\infty > K} \nonumber \\
    & \qquad \leq 
    \sup_{\theta\in \Theta} 
        \mathbb{P}\pa{ \| U^\theta \|_\infty > K/6}
    + \sup_{\theta\in \Theta} 
        \mathbb{P}\pa{ \| \Gamma^\theta \|_\infty > K/6},
\end{align}
where $\Gamma^\theta$ is the limiting Gaussian process from (the proof of) Lemma \ref{lem:D1asymptotics}. 
The last term in \eqref{eq:LCMboundednes} can be made arbitrarily small for $K$ sufficiently large by Lemma \ref{lem:D1bounded}. 
Lemma~\ref{lem:Ulimitistight} states that the family $(U^\theta)_{\theta \in \Theta}$ is tight in $C[0,1]$, and hence the family $(\|U^\theta\|_\infty)_{\theta \in \Theta}$ is tight in $\mathbb{R}_{\geq 0}$. This implies that the first term in \eqref{eq:LCMboundednes} can also be made arbitrarily small for $K$ sufficiently large.
This establishes \eqref{eq:stochboundedness} and we are done.
\hfill \qed{}

\subsection{Proof of Proposition \ref{prop:varianceconsistent}}
Consider the decomposition of the variance function estimator given by
    \begin{align*}
        \hat{\mathcal{V}}_n(t) & = A^{(n)}_t + B^{(n)}_t + 2 C^{(n)}_t
    \end{align*}
where 
    \begin{align*}
        A^{(n)}_t & =
        \frac{1}{|J_n|} \sum_{j \in J_n} 
        \int_0^t G_{j,s}^2 \mathrm{d} N_{j,s},
        \\
        B^{(n)}_t & = 
        \frac{1}{|J_n|} \sum_{j \in J_n} 
        \int_0^t \left( G_{j,s} - \hat{G}_{j,s}^{(n)} \right)^2 \mathrm{d} N_{j,s}, 
        \\
        C^{(n)}_t & = 
        \frac{1}{|J_n|} \sum_{j \in J_n}
        \int_0^t G_{j,s} \left(G_{j,s} - \hat{G}_{j,s}^{(n)} \right) 
        \mathrm{d} N_{j,s}.
    \end{align*}
We first consider the asymptotic limit of $A^{(n)}$, which is the empirical mean of $|J_n|$ i.i.d. samples of the process $\int_0^t G_s^2\mathrm{d}N_s$.
Under Assumption \ref{asm:UniformBounds}, we can apply the first part of Lemma \ref{lem:countingmartingales} which states $\bM_t^2 - \bLambda_t$ is a martingale. We use this fact to note that
\begin{align*}
    \ex(N_1^2) 
        = \ex((\bM_1 + \bLambda_1)^2)
        \leq 2\left(\ex(\bM_1^2) + \ex(\bLambda_1^2)\right)
        = 4\, \ex\left(\Big(\int_0^1 \blambda_s \mathrm{d}s\Big)^2\right)
        \leq 4 C^2.
\end{align*}
Now, another use of Assumption \ref{asm:UniformBounds} shows that $\int_0^t G_s^2\mathrm{d}N_s$ has a second moment bounded by $4(CC')^2$. 
Thus we can apply the uniform law of large numbers \cite[Lemma 19]{shah2020hardness} to conclude for each $t\in [0,1]$,
\begin{align*}
    A_t^{(n)} 
        = \frac{1}{|J_n|} \sum_{j \in J_n} 
            \int_0^t G_{j,s}^2 \mathrm{d} N_{j,s}
        \convUP 
        \ex\left(\int_0^t G_{s}^2 \mathrm{d} N_s \right) = \mathcal{V}(t).
\end{align*}
Note also that $A^{(n)}$ and $\mathcal{V}$ are non-decreasing and that the collection $(\mathcal{V}^\theta)_{\theta\in\Theta}$ is uniformly equicontinuous by Lemma \ref{lem:regularityofgammasigma}. These are exactly the conditions for Lemma \ref{lem:convergenceofincreasing}, so we can automatically conclude that $\sup_{t\in [0,1]}|A_t^{(n)} - \mathcal{V}(t)|\convUP 0$.

Next we show that the remainder terms  
$B^{(n)}$ and $C^{(n)}$ converge uniformly 
to zero in expectation. Similarly to the proof of \cref{lemma:R2}, we have under Assumptions \ref{asm:UniformBounds} and \ref{asm:UniformRates},
    \begin{align*}
        \ex \left( \sup_{0 \leq t \leq 1} B_t^{(n)} \right) 
        = \ex (B_1^{(n)}) 
        & = \ex \left(
            \frac{1}{|J_n|} \sum_{j \in J_n} 
            \int_0^1 \left( G_{j,s} - \hat{G}_{j,s}^{(n)} \right)^2 \blambda_{j, s} \mathrm{d}s
        \right) 
        \\
        & = \ex \left(
        \ex \left( 
        \frac{1}{|J_n|} \sum_{j \in J_n} 
            \int_0^1 \left( G_{j,s} - \hat{G}_{j,s}^{(n)} \right)^2 
            \blambda_{j, s} \mathrm{d}s
        \mid 
        \tilde{\mathcal{G}}_1^c
        \right)
        \right)
        \\
        & = 
        \ex \left( 
            \int_0^1 \left( G_{s} - \hat{G}_{s}^{(n)} \right)^2 \blambda_{s} \mathrm{d}s
        \right)
        \\
        & \leq C \cdot g(n)^2 \longrightarrow 0
    \end{align*}
as $n \to \infty$ uniformly over $\Theta$. 
Lastly, we see that
    \begin{align*}
        \ex \left| \sup_{0 \leq t \leq 1} C_t^{(n)} \right| 
        & \leq 
        \ex \left( \sup_{0 \leq t \leq 1} |C_t^{(n)}| \right)
        \\
        & \leq 
        \ex \left( 
            \frac{1}{|J_n|} \sum_{j \in J_n}
            \sup_{0 \leq t \leq 1}
            \int_0^t |G_{j,s}| |G_{j,s} - \hat{G}_{j,s}^{(n)}| \blambda_{j, s}
            \mathrm{d} s
        \right)
        \\
        & =
        \ex \left(
            \frac{1}{|J_n|} \sum_{j \in J_n}
            \int_0^1 |G_{j,s}| |G_{j,s} - \hat{G}_{j,s}^{(n)}| \blambda_{j, s}
            \mathrm{d} s
        \right)
        \\
        & = 
        \ex \left( 
        \ex \left( 
            \frac{1}{|J_n|} \sum_{j \in J_n}
            \int_0^1 |G_{j,s}| |G_{j,s} - \hat{G}_{j,s}^{(n)}| \blambda_{j, s}
            \mathrm{d} s
        \mid
        \tilde{\mathcal{G}}_1^c
        \right) 
        \right) 
        \\
        & = 
        \ex \left( 
            \int_0^1 |G_{s}||G_{s} - \hat{G}_{s}^{(n)}| \blambda_{s}
            \mathrm{d} s
        \right)
        \\
        & \leq CC' \ex \left( 
                \int_0^1 |G_{s} - \hat{G}_{s}^{(n)}|
                \mathrm{d} s
            \right) \\
        & \leq CC' \cdot g(n) \longrightarrow 0
    \end{align*}
as $n \to \infty$ uniformly over $\Theta$ by \cref{asm:UniformRates}. 
Combining the convergences established for $A^{(n)}$, $B^{(n)}$, and $C^{(n)}$, we get by a generalized Slutsky (Lemma \ref{lem:SkorokhodSlutsky}) that 
$$
    \sup_{t\in[0,1]}|\hat{\mathcal{V}}_n(t) - \mathcal{V}(t)| \convUP 0.
$$
\hfill \qed{}

\subsection{Proof of Corollary \ref{cor:statisticsdistribution}}
Under \cref{asm:UniformBounds,asm:UniformRates} we know by \cref{thm:main} and \cref{prop:varianceconsistent} that
    \begin{align}\label{eq:LCMandVarconvergent}
    \sqrt{|J_n|} \hat{\gamma}^{(n),\theta} \convUDnull U^\theta
    \qquad \text{and} \qquad
    \hat{\mathcal{V}}_n^\theta \convUPnull \mathcal{V}^\theta
    \end{align}
in $D[0,1]$ as $n \to \infty$. If we were to show pointwise convergence of the test statistic, this would now be a straightforward consequence of the continuous mapping theorem. 
However, to show uniform convergence, we will need an additional tightness argument. 

Let $(\theta_n)_{n\in\mathbb{N}}\subset \Theta_0$ be an arbitrary sequence. Proposition \ref{prop:seqUni} then states that it suffices to show that there exists a subsequence $(\theta_{k(n)})_{n\in \mathbb{N}} \subseteq (\theta_{n})_{n\in \mathbb{N}}$, with $k\colon \mathbb{N} \to \mathbb{N}$ strictly increasing, such that
\begin{equation}\label{eq:ProofOfTeststatistic}
    \lim_{n\to \infty} d_{BL}\big(\hat{D}_{k(n)}^{\theta_{k(n)}}, 
    \mathcal{J}(U^{\theta_{k(n)}},\mathcal{V}^{\theta_{k(n)}})\big)
    = 0.
\end{equation}
Here $d_{BL}$ denotes the bounded Lipschitz metric defined in Appendix \ref{app:UniformAsymptotics}. By Lemma~\ref{lem:Ulimitistight}, the collection $(U^\theta)_{\theta\in \Theta}$ is tight in $C[0,1]$ under Assumption \ref{asm:UniformBounds}. Therefore, Prokhorov's theorem \citep[Theorem 23.2]{kallenberg2021foundations} asserts that there exists a subsequence $(\theta_{a(n)}) \subset (\theta_n)$, and a $C[0,1]$-valued random variable $\tilde{U}$ such that $U^{\theta_{a(n)}} \xrightarrow{\mathcal{D}} \tilde U$ in $C[0,1]$.

Likewise, Lemma \ref{lem:regularityofgammasigma} states that the collection $(V^\theta)_{\theta\in \Theta}$ is uniformly bounded and uniformly equicontinuous under Assumption \ref{asm:UniformBounds}. Thus the Arzelà-Ascoli theorem yields that there exists a further subsequence $(\theta_{b(n)})\subset (\theta_{a(n)})$ and a function $\tilde{\mathcal{V}}\in C[0,1]$ such that $\|\mathcal{V}^{\theta_{b(n)}} - \tilde{\mathcal{V}}\|_\infty \to 0$. 

Combining the convergences of $U^{\theta_{b(n)}}$ and $\mathcal{V}^{\theta_{b(n)}}$ with those in Equation \eqref{eq:LCMandVarconvergent}, it follows from the triangle inequality of the metric $d_{BL}$ that also
    \begin{align*}
    \sqrt{|J_{b(n)}|} \hat{\gamma}^{(b(n)),{\theta_{b(n)}}} \xrightarrow{\mathcal{D}} \tilde{U}
    \qquad \text{and} \qquad
    \hat{\mathcal{V}}_{b(n)}^{\theta_{b(n)}}\xrightarrow{P} \tilde{\mathcal{V}},
    \end{align*}
in $D[0,1]$ as $n\to \infty$. Now we may use that convergence in Skorokhod topology is equivalent to convergence in uniform topology whenever the limit variable continuous, see e.g. \citet[Theorem 23.9]{kallenberg2021foundations}. Hence the convergences above also hold in $(D[0,1], \|\cdot\|_\infty)$.

Since $\mathcal{V}$ is deterministic, this implies the joint convergences
    \begin{align*}
        (U^{\theta_{b(n)}},\mathcal{V}^{\theta_{b(n)}})
        \xrightarrow{\mathcal{D}}
        (\tilde{U},\tilde{\mathcal{V}})
            \qquad \text{and} \qquad
        \Big(\sqrt{|J_{b(n)}|} \hat{\gamma}^{(b(n)),{\theta_{b(n)}}}, \hat{\mathcal{V}}_{b(n)}^{\theta_{b(n)}}\Big) 
        \xrightarrow{\mathcal{D}} 
        (\tilde{U},\tilde{\mathcal{V}})
    \end{align*}
in the product space $D[0,1]\times D[0,1]$ endowed with the uniform topology. 
Since $(\tilde{U},\tilde{\mathcal{V}}) \in C[0,1]\times \overline{\{\mathcal{V}^\theta \colon \theta \in \Theta_0\}}$ takes values in the continuity set of $\mathcal{J}$ by assumption, the classical continuous mapping theorem lets us conclude that
\begin{align*}
    \mathcal{J}(U^{\theta_{b(n)}},\mathcal{V}^{\theta_{b(n)}})
    \xrightarrow{\mathcal{D}}
    \mathcal{J}(\tilde{U},\tilde{\mathcal{V}})
        \quad \text{and} \quad
    \hat{D}_{b(n)}^{\theta_{b(n)}}=\mathcal{J}\Big(\sqrt{|J_{b(n)}|} \hat{\gamma}^{(b(n)),{\theta_{b(n)}}}, \hat{\mathcal{V}}_{b(n)}^{\theta_{b(n)}}\Big) 
    \xrightarrow{\mathcal{D}} 
    \mathcal{J}(\tilde{U},\tilde{\mathcal{V}})
\end{align*}
as $n \to \infty$. Now another application of the triangle inequality with $\mathcal{J}(\tilde{U},\tilde{\mathcal{V}})$ as intermediate value shows that \eqref{eq:ProofOfTeststatistic} holds with $k(n)=b(n)$, so we are done.
\hfill $\square$


\subsection{Proof of Theorem \ref{thm:LCTlevel}}
We will apply Corollary \ref{cor:statisticsdistribution} with the functional $\mathcal{J}$ given by
$$
    \mathcal{J}(f_1,f_2) = 
    \one(f_2 \neq 0)\frac{\|f_1\|_\infty }{\sqrt{|f_2(1)|}}, \qquad f_1,f_2 \in D[0,1].
$$
Under Assumption \ref{asm:UniformVariance}, it suffices to check continuity of $\mathcal{J}$ on the set $\Upsilon$ given by
$$ 
    \Upsilon \coloneqq 
    C[0,1] \times \{f \in C[0,1] \mid \delta_1 \leq |f(1)| \} 
    \supset 
    C[0,1]\times \overline{\{\mathcal{V}^\theta \colon \theta \in \Theta_0\}} .
$$
To see that $\mathcal{J}$ is continuous on $\Upsilon$ in the uniform topology, we note that it can be written as a composition of the continuous maps
\begin{align*}
    \Upsilon \longrightarrow [0,\infty)\times [\delta_1, \infty), &
        \qquad (f_1,f_2) \mapsto (\|f_1\|_\infty,|f_2(1)|), \\
    [0,\infty)\times [\delta_1, \infty) \longrightarrow \mathbb{R}, &
        \qquad (x_1,x_2)\mapsto \frac{x_1}{\sqrt{x_2}}.
\end{align*}
Thus it follows from Corollary \ref{cor:statisticsdistribution} that
\begin{align*}
    \hat T_n = \frac{\sqrt{|J_n|}\sup_{t\in [0,1]} |\hat \gamma_t^{(n)}|
                }{\sqrt{\hat{\mathcal{V}}_n(1)}}
     = \mathcal{J}\left( 
        \sqrt{|J_n|}\hat{\gamma}^{(n)}, \; \hat{\mathcal{V}}_n\right)
    \convUDnull  
    \mathcal{J}(U, \mathcal{V})
    = \frac{\|U\|_\infty}{\mathcal{V}(1)}.
\end{align*}
With $(B_u)$ a Brownian motion it follows by Proposition \ref{prop:BMrepresentation} 
that 
\begin{align}\label{eq:UnormalizedIsS}
    \dfrac{\|U\|_\infty}{\sqrt{\mathcal{V}(1)}}
    \overset{\mathcal{D}}{=}
        \dfrac{\sup_{0\leq t\leq 1}|B_{\mathcal{V}(t)}|}{\sqrt{\mathcal{V}(1)}}
    =   \dfrac{\sup_{0\leq u\leq \mathcal{V}(1)}|B_{u}|}{\sqrt{\mathcal{V}(1)}}
    \stackrel{\mathcal{D}}{=} 
        \sup_{0\leq t\leq 1}|B_{t}| \stackrel{\mathcal{D}}{=} S,
\end{align}
where we have used that $\mathcal{V}$ is continuous and that Brownian motion is scale invariant. This establishes the first part of the theorem.

For the second part, we first note that the distribution of $S$ is absolutely continuous with respect to Lebesgue measure, which follows from Equation \eqref{eq:supdistribution}.
Then we can use Theorem 4.1 of \citet{bengs2019uniform} to conclude that
$$
    \limsup_{n\to \infty} \sup_{\theta \in \Theta} 
    |\mathbb{P}(\hat T_n \leq z_{1-\alpha}) - (1-\alpha)| = 0.
$$
It follows from the triangle inequality that
\begin{align*}
    \limsup_{n\to \infty}\sup_{\theta \in \Theta} \mathbb{P}(\Psi_n^\alpha = 1)
    = \limsup_{n\to \infty} \sup_{\theta \in \Theta} \mathbb{P}(\hat T_n > z_{1-\alpha})
    \leq \alpha.
\end{align*}
\hfill $\square$

\subsection{Proof of Theorem \ref{thm:rootNpower}}
Let $0<\alpha<\beta<1$ be given.
The second part of Theorem~\ref{thm:main} permits us to choose $K>0$ sufficiently large such that
\begin{align}
    \limsup_{n\to \infty}\sup_{\theta\in \Theta} \mathbb{P}\pa{
    (\sqrt{|J_n|}\|\hat{\gamma}^{(n),\theta} - \gamma^\theta\|_\infty )> K} < 1 - \beta.
\end{align}
We then choose $c > K + z_{1-\alpha}\sqrt{1+C(C')^2}$ such that for all $\theta \in \mathcal{A}_{c,n}$, it holds that
\begin{align*}
    \sqrt{|J_n|}\|\gamma^\theta\|_\infty-z_{1-\alpha}\sqrt{1+\mathcal{V}^\theta(1)} 
    \geq c - z_{1-\alpha} \sqrt{1+C(C')^2} > K,
\end{align*}
where we have used Lemma \ref{lem:regularityofgammasigma} in the first inequality. The (reverse) triangle inequality now yields that for any $\theta \in \mathcal{A}_{c,n}$
\begin{align*}
    (\Psi_n^\theta = 0)
    = (\hat{T}_n^\theta \leq z_{1-\alpha}) 
    &= \pa{\|\hat\gamma^{(n),\theta}\|_\infty 
        \leq \sqrt{\hat{\mathcal{V}}_n^\theta(1)}\frac{z_{1-\alpha}}{\sqrt{|J_n|}}} \\
    &\subseteq \pa{
        \|\gamma^\theta\|_\infty
        -
        \left\|
            \hat\gamma^{(n),\theta}
                - \gamma^\theta
        \right\|_\infty
        \leq \sqrt{\hat{\mathcal{V}}_n^\theta(1)}\frac{z_{1-\alpha}}{\sqrt{|J_n|}}} \\
    &\subseteq E_1^{(n),\theta} \cup E_2^{(n),\theta},    
\end{align*}
where
\begin{align*}
    E_1^{(n),\theta} &= \pa{\sqrt{|J_n|} \left\| \hat\gamma^{(n),\theta} 
            - \gamma^\theta\right\|_\infty > K},\\
    E_2^{(n),\theta} &= \pa{\hat{\mathcal{V}}_n^\theta(1) 
        > 1+\mathcal{V}^\theta(1)}
        \subseteq
            \pa{|\hat{\mathcal{V}}_n^\theta(1)-\mathcal{V}^\theta(1)| > 1}.
\end{align*}
From Proposition \ref{prop:varianceconsistent} we know that $\limsup_{n\to \infty} \sup_{\theta \in \Theta} \mathbb{P}(E_2^{(n),\theta}) = 0$, so from the choice of $K$ we conclude that
\begin{align*}
    \limsup_{n\to \infty}\sup_{\theta \in \Theta} \mathbb{P}(\Psi_n=0)
    \leq \limsup_{n\to \infty} \sup_{\theta \in \Theta} \mathbb{P}(E_1^{(n),\theta})
    < 1- \beta.
\end{align*}
The desired statement follows from substituting $\mathbb{P}(\Psi_n=0) = 1 - \mathbb{P}(\Psi_n=1)$ into the above equation and simplifying. \hfill \qedsymbol

\subsection{Proof of Theorem \ref{thm:LCTXlevel}}
Assume that $H_0$ holds and note that Assumptions \ref{asm:UniformBounds} and \ref{asm:UniformRates} are satisfied for every sample split $J_n^k \cup (J_n^k)^c$, $k=1,\ldots,K$. 

We consider the decomposition in Equation \eqref{eq:decomposition} for each sample split $J_n^k \cup (J_n^k)^c$, and denote the corresponding processes by
$U^{k,(n)}$, $R_1^{k,(n)}$, $R_2^{k,(n)}$, $R_3^{k,(n)}$, $D_1^{k,(n)}$, and $D_2^{k,(n)}$. 
For each fold $k\in\{1,\ldots,K\}$, we can then apply the results in Section \ref{sec:asymptotics} for a single data split:
\begin{itemize}
    \item By \cref{prop:UasymptoticGaussian}, we have that $U^{k,(n)} \convUD U$ in $D[0,1]$, where $U$ is a mean zero continuous Gaussian martingale with variance function $\mathcal{V}$. 

    \item By \cref{prop:remainderterms}, 
    $R_\ell^{k,(n)} \convUP 0$ in $(D[0,1],\|\cdot\|_\infty)$ as $n \to \infty$.
    
    \item Under $H_0$, the processes $D_1^{k,(n)}$ and $D_2^{k,(n)}$ are equal to the zero process almost surely.
\end{itemize}
Recall that the folds are assumed to have uniform asymptotic density, which is equivalent to $\frac{\sqrt{n}}{\sqrt{K|J_n^k|}} \to 1$ as $n\to \infty$.
Thus we may also conclude that for each fixed $k$ and $\ell$,
\begin{align*}
    \frac{\sqrt{n}}{\sqrt{K|J_n^k|}} U^{k,(n)}\convUD U 
    \quad \text{and} \quad 
    \frac{\sqrt{n}}{K\sqrt{|J_n^k|}} R_\ell^{k,(n)} \convUP 0,
\end{align*}
where the convergences hold in the Skorokhod and uniform topology, respectively. Now the key observation is that 
\begin{align*}
    U^{1,(n)} 
    \perp\!\!\!\!\perp \cdots \perp\!\!\!\!\perp 
    U^{K,(n)}.
\end{align*}
To see this, note that $U^{k,(n)}$ is constructed from $(G_j,M_j)_{j\in J^k}$ only, and by the i.i.d. assumption of the data, the collections $(G_j,M_j)_{j\in J^1},\ldots,(G_j,M_j)_{j\in J_n^K}$ are jointly independent. 
We can therefore apply Lemma \ref{lem:sumofindependent} iteratively to the sequences 
$$
    \frac{\sqrt{n}}{\sqrt{K|J_n^1|}} U^{1,(n)}, \ldots, \frac{\sqrt{n}}{\sqrt{K|J_n^K|}} U^{K,(n)}
$$ to conclude that their sum is uniformly convergent to the sum of $K$ independent copies of $U$. Using the convolution property of the Gaussian distribution, it therefore follows that 
\begin{align*}
    \check{U}^{K,(n)}
    \coloneqq
    \frac{1}{\sqrt{K}}\sum_{k=1}^K 
        \frac{\sqrt{n}}{\sqrt{K|J_n^k|}} U^{k,(n)}
        \convUD U
\end{align*}
in $D[0,1]$ as $n\to \infty$. By the uniform Slutsky theorem formulated in Lemma \ref{lem:SkorokhodSlutsky}, we can therefore conclude that
    \begin{align*}
        \sqrt{n}\check{\gamma}^{K,(n)}      
        = 
        \check{U}^{K,(n)}
        +   \sum_{k=1}^K 
            \frac{\sqrt{n}}{K\sqrt{|J_n^k|}} \left( 
            R_1^{k,(n)} + R_2^{k,(n)} + R_3^{k,(n)} + 
            D_1^{k,(n)} + D_2^{k,(n)} 
        \right)
        \convUDnull U
    \end{align*}
in $D[0,1]$ as $n\to \infty$. Note that we use $\theta\in \Theta_0$ to ensure that $D_1^{k,(n)} + D_2^{k,(n)}$ is equal to the zero process almost surely. 
Since the limit $(U^\theta)_{\theta \in \Theta_0}$ is tight in $C[0,1]$ by Lemma~\ref{lem:Ulimitistight}, 
Proposition~\ref{prop:convergencetoCont2} lets us conclude that $\sqrt{n}\|\check{\gamma}^{K,(n)}\|_\infty \convUDnull \|U\|_\infty$.

Consider now the cross-fitted variance estimator at its endpoint
\begin{equation*}
    \check{\mathcal{V}}_{K,n}(1) = \frac{1}{K} \sum_{k=1}^K \frac{1}{|J_n^k|}\sum_{j \in J_n^k} \int_0^1 \left(\hat{G}_{j,s}^{k,(n)}\right)^2 \mathrm{d} N_{j,s}.
\end{equation*}
From \cref{prop:varianceconsistent}, we see that $\check{\mathcal{V}}_{K,n}(1)$ is an average of $K$ variables converging uniformly in probability to $\mathcal{V}(1)$ in the uniform topology. Hence $\check{\mathcal{V}}_{K,n}(1)$ also converges uniformly in probability to $\mathcal{V}(1)$ in the uniform topology. We can then apply Theorem 6.3 of \citet{bengs2019uniform}, which is a uniform version of Slutsky's theorem, to conclude that
\begin{align*}
    \check{T}_{n}^K = \frac{\sqrt{n} \|\check{\gamma}^{K,(n)}\|_\infty}{\check{\mathcal{V}}_{K,n}(1)} 
    \convUDnull
    \frac{\|U\|_\infty}{\mathcal{V}(1)} \stackrel{\mathcal{D}}{=} S,
\end{align*}
as $n\to \infty$, where last equality in distribution was established in \eqref{eq:UnormalizedIsS}. 

Following the second part of the proof of Theorem \ref{thm:LCTlevel}, we conclude in the X-LCT has uniform asymptotic level. \hfill $\square$

\section{Uniform stochastic convergence}\label{app:UniformAsymptotics}
In this section, we discuss weak convergence of random variables with values in a metric space uniformly over a parameter set $\Theta$. The uniformity over the parameter set can be used, for example, to establish uniform asymptotic level as well as power under local alternatives.

The content of this section extends the works of \citet{bengs2019uniform} and \citet{kasy2019uniformity}, and we especially build upon Appendix B of \citet{lundborg2021conditional}, in which uniform stochastic convergence is considered in \emph{separable} Banach spaces and Hilbert spaces. 
The space space $(D[0,1], \|\cdot\|_\infty)$ of \cl{} functions endowed with the uniform norm is a Banach space, but it is unfortunately not separable. Therefore we extend the notion of uniform stochastic convergence to random variables in metric spaces, with the condition that the limit is supported on a separable set. This allows to consider uniform weak convergence in two important special cases: 
i) convergence in $(D[0,1], \|\cdot\|_\infty)$ towards variables in $(C[0,1], \|\cdot\|_\infty)$, and ii) convergence in $D[0,1]$ endowed with the Skorokhod metric. 

The Skorokhod space $D[0,1]$ is, if not specified otherwise, equipped with the complete Skorokhod metric $d^\circ$, which makes it a Polish space, i.e., a complete and separable metric space. See for example Section 12 in \citet{billingsley2013convergence} for a discussion of the Skorokhod space and in particular Equation (12.16) for a definition of $d^\circ$.

\subsection{Uniform stochastic convergence in metric spaces}
Throughout this section we consider a background probability space $(\Omega, \mathbb{F}, \mathbb{P})$ and let $(\mathbb{D},d_{\mathbb{D}})$ denote a generic metric space. 
We define $BL_1(\mathbb{D})$ as the set of real-valued functions on $\mathbb{D}$ with Lipschitz norm bounded by $1$, that is, functions $f\colon \mathbb{D} \to \mathbb{R}$ with 
$\|f\|_\infty \leq 1$ and $|f(x)-f(y)| \leq d_{\mathbb{D}}(x,y)$ for every $x,y\in \mathbb{D}$.
Let $\mathcal{M}_1(\mathbb{D})$ denote the set of Borel probability measures on $\mathbb{D}$. We then define the \emph{bounded Lipschitz metric} on $\mathcal{M}_1(\mathbb{D})$ by 
\begin{align*}
    d_{BL}(\mu,\nu) \coloneqq 
    \sup_{f\in BL_1(\mathbb{D})}
    \Big|\int f \mathrm{d}\mu -\int f \mathrm{d}\nu\Big|,
    \qquad \mu,\nu \in \mathcal{M}_1(\mathbb{D}).
\end{align*}
For any pair $(X,Y)$ of $\mathbb{D}$-valued random variables we use the shorthand notation
\begin{align*}
    d_{BL} (X,Y)
    = d_{BL} (X(\mathbb{P}),Y(\mathbb{P}))
    = \sup_{f\in BL_1(\mathbb{D})} \lvert \mathbb{E}(f(X)-f(Y))\rvert.
\end{align*}
If the underlying metric space is ambiguous for $d_{BL}$, we will specify that it is the bounded Lipschitz metric on $\mathcal{M}_1(\mathbb{D})$ by writing $d_{BL(\mathbb{D})}$. 
Our interest in the bounded Lipschitz metric is due to its characterization of weak convergence.

\begin{prop}\label{prop:boundedLipconv}
    Let $X, X_1, X_2, \ldots$ be a sequence of $\mathbb{D}$-valued random variables. Assume that there exists a separable subset $\mathbb{D}_0 \subseteq \mathbb{D}$
    such that $\mathbb{P}(X \in \mathbb{D}_0)=1$. Then the following are equivalent:
    \begin{itemize}
        \item The sequence $(X_n)_{n\geq 1}$ converges in distribution to $X$, i.e., for all $f\in C_b(\mathbb{D})$ it holds that
        $\mathbb{E}[f(X_n)] \to \mathbb{E}[f(X)]$ as $n\to \infty$.
        \item It holds that $d_{BL}(X_n,X) \to 0$ as $n\to \infty$.
    \end{itemize}
\end{prop}
\begin{proof}
    See Theorem 1.12.2, Addendum 1.12.3, and the following discussion in \citet{Vaart:1996}.
\end{proof}
To discuss uniform stochastic convergence, we will for the remaining part of this section let $\Theta$ be fixed set, which is used as a (possible) parameter set for every random variable. We say that a collection $(X^{\theta})_{\theta \in \Theta}$ of $\mathbb{D}$-valued random variables is \emph{separable} if there exists a separable subset $\mathbb{D}_0 \subseteq \mathbb{D}$ such that $\mathbb{P}(X^\theta \in \mathbb{D}_0)=1$ for all $\theta \in \Theta$. If $\mathbb{D}$ is a separable metric space, then any collection of $\mathbb{D}$-valued random variables is automatically separable.

Now Lemma~\ref{prop:boundedLipconv} justifies the following generalization of weak convergence uniformly over $\Theta$:
\begin{dfn}\label{dfn:UniformConvergence}
    Let $(X_{n}^\theta)_{n\in \mathbb{N}, \theta \in \Theta}$ and $(X^{\theta})_{\theta \in \Theta}$ be collections of $\mathbb{D}$-valued random variables and assume that $(X^{\theta})_{\theta \in \Theta}$ is separable. We say that:
    \begin{enumerate}[label=(\roman*)]
        \item \emph{$X_{n}^\theta$ converges uniformly in distribution over $\Theta$ to $X^\theta$ in $\mathbb{D}$}, and write $X_{n}^\theta \convUD X^\theta$, if
        \begin{align*}
            \lim_{n\to \infty}\sup_{\theta \in \Theta}
            d_{BL(\mathbb{D})}(X_{n}^\theta,X^\theta)
            = 0.
        \end{align*}
        
        \item \emph{$X_{n}^\theta$ converges uniformly in probability over 
        $\Theta$ to $X^\theta$ in $\mathbb{D}$}, and write $X_{n}^\theta \convUP X^\theta$, if 
        \begin{align*}
            \lim_{n\to \infty}\sup_{\theta \in \Theta}\mathbb{P}(d_{\mathbb{D}}(X_{n}^\theta, X^\theta)>\epsilon)
            = 0
        \end{align*}
        for every $\epsilon > 0$.
    \end{enumerate}
\end{dfn}
If for some $\mu\in \mathcal{M}_1(\mathbb{D})$, it holds that $X_n^\theta \convUD X^\theta$ with $X^\theta(\mathbb{P}) = \mu$ for all $\theta \in \Theta$, we also write $X_n^\theta \convUD \mu$. Similarly, we may replace the limit random variable $X^\theta$ by a point $x\in \mathbb{D}$ by interpreting $x$ as the constant map $(\theta,\omega)\mapsto x$ for $\theta \in \Theta$ and $\omega \in \Omega$.

Note that if the parameter set $\Theta = \{\theta_0\}$ is a singleton, then each type of uniform convergence reduces to the corresponding classical definition of convergence in distribution or probability. 
If $\mathbb{D}$ is a separable Banach space, we note that Definition \ref{dfn:UniformConvergence} coincides with Definition 3 in \citet{lundborg2021conditional}.
\begin{prop}\label{prop:seqUni}
    Let $(X_n^\theta)_{n\in \mathbb{N}, \theta \in \Theta}$ and $(X^\theta)_{\theta \in \Theta}$ be collections of \(\mathbb{D}\)-valued random variables and assume $(X^\theta)_{\theta \in \Theta}$ is separable. Then the following are equivalent:
    \begin{itemize}
        \item[a)] $X_n^\theta \convUD X^\theta$ as $n\to \infty$.

        \item[b)] For any sequence $(\theta_n)_{n\in\mathbb{N}}\subseteq \Theta$ it holds that $d_{BL}(X_{n}^{\theta_{n}},X^{\theta_{n}}) \to 0$ as $n\to 0$.
        
        \item[c)] For any sequence $(\theta_n)_{n\in\mathbb{N}}\subseteq \Theta$ there exists a subsequence $(\theta_{k(n)})_{n\in\mathbb{N}}$, with $k\colon \mathbb{N} \to \mathbb{N}$ strictly increasing,
        such that
        \begin{align*}
            \lim_{k\to \infty}d_{BL}(X_{k(n)}^{\theta_{k(n)}},X^{\theta_{k(n)}})
            =0.
        \end{align*}
    \end{itemize}
    Moreover, $X_n^\theta \convUP X^\theta$ if and only if for any sequence $(\theta_n)_{n\in\mathbb{N}}\subseteq \Theta$ and any $\epsilon>0$ it holds that
        \begin{align*}
            \lim_{n\to \infty}\mathbb{P}(d_{\mathbb{D}}(X_n^{\theta_n}, X^{\theta_n})>\epsilon) = 0.
        \end{align*}
\end{prop}
\begin{proof}
    This is essentially Lemma 1 in \citet{kasy2019uniformity} for $\mathbb{D}$-valued random variables, except that we have added the equivalent condition c). The proof for the characterization of uniform convergence in probability is identical to the one given by \citet{kasy2019uniformity}, so we focus on the equivalence between a), b), and c).
    To this end, we to prove that $a) \implies b) \implies c) \implies a)$.
    
    The fact that a) implies b) follows directly from applying the bound 
    \[
       d_{BL}(X_n^{\theta_n},X^{\theta_{n}})
            \leq \sup_{\theta \in \Theta} d_{BL}(X_n^\theta,X^\theta)
    \]
    and taking the limit as $n\to \infty$. We also see that b) implies c) since any sequence is a subsequence of itself. 
    
    We show that c) implies a) by contraposition.
    Assume the negation of a), that is, there exists an $\epsilon>0$ and a sequence $(\theta_n)_{n\in \mathbb{N}}\subseteq \Theta$ such that 
    \begin{align*}
        d_{BL}(X_{n}^{\theta_{n}},X^{\theta_{n}}) > \epsilon
    \end{align*}
    for all $n\in \mathbb{N}$. Then, for all subsequences $(\theta_{k(n)})$ of $(\theta_n)$, it holds that $d_{BL}(X_{k(n)}^{\theta_{k(n)}},X^{\theta_{k(n)}})$ does not converge to zero. This implies the negation of c).
\end{proof}

Proposition \ref{prop:seqUni} will allow us to extend many results for classical stochastic convergence to uniform stochastic convergence.
\begin{cor}\label{cor:ConvUDtoconstant}
    Let $(X_n^\theta)_{n\in \mathbb{N},\theta \in \Theta}$ be a collection of \(\mathbb{D}\)-valued random variables and let $x\in \mathbb{D}$. Then $X_n^\theta \convUD x$ if and only if $X_n^\theta \convUP x$.
\end{cor}
\begin{proof}
    For any sequence $(\theta_n)_{n\in\mathbb{N}}\subseteq \Theta$, recall that $X_n^{\theta_n} \xrightarrow{\mathcal{D}} x$ if and only if $X_n^{\theta_n} \xrightarrow{P} x$, see e.g. Lemma 5.1 in \citet{kallenberg2021foundations}.  Hence the statement follows directly from Proposition \ref{prop:seqUni} (combined with 
    Proposition \ref{prop:boundedLipconv}).
\end{proof}

Our goal is to prove uniform versions of Slutsky's theorem for $D[0,1]$, Rebolledo's central limit theorem, and the chaining lemma for stochastic processes. To prove Slutsky's lemma for $D[0,1]$, we first prove a general result for metric spaces.
\begin{lem}\label{lem:GeneralSlutsky}
    Let $(X^\theta, X_n^\theta, Y_n^\theta)_{n\in \mathbb{N},\theta \in \Theta}$ be a collection of $\mathbb{D}$-valued random variables and assume that $(X^\theta)_{\theta \in \Theta}$ is separable. 
    If $X_n^\theta \convUD X^\theta$ and $d_{\mathbb{D}}(X_n^\theta,Y_n^\theta) \convUP 0$,
    then it also holds that $Y_n^\theta \convUD X^\theta$.
\end{lem}
\begin{proof}
    By the triangle inequality of the bounded Lipschitz metric, we observe that
    \begin{align*}
        \sup_{\theta \in \Theta}d_{BL}(Y_n^\theta, X^\theta) 
            \leq \sup_{\theta \in \Theta} d_{BL}(Y_n^\theta, X_n^\theta) 
                + \sup_{\theta \in \Theta} d_{BL}(X_n^\theta, X^\theta).
    \end{align*}
    The last term converges to zero by the assumption of $X_n^\theta \convUD X^\theta$. 
    For the other term, let $\epsilon>0$ and use the partition
    $$
        (d_{\mathbb{D}}(X_n^\theta,Y_n^\theta)\leq \epsilon) \cup (d_{\mathbb{D}}(X_n^\theta,Y_n^\theta)>\epsilon)
    $$
    to obtain that
    \begin{align*}
        d_{BL}(X_n^\theta, Y_n^\theta) 
            &= \sup_{f\in BL_1(\mathbb{D})} |\mathbb{E}[f(X_n^\theta)-f(Y_n^\theta)]|\\
            &\leq \epsilon 
                + \sup_{f\in BL_1(\mathbb{D})} 
             \mathbb{E}[|f(X_n^\theta)-f(Y_n^\theta)|\one(d_{\mathbb{D}}(X_n^\theta,Y_n^\theta)>\epsilon)]\\
            &\leq \epsilon 
            + \mathbb{P}(d_{\mathbb{D}}(X_n^\theta,Y_n^\theta)>\epsilon).
    \end{align*}
    Taking the supremum over $\Theta$ and the limit superior for $n\to \infty$ finishes the proof.
\end{proof}
The following formulation of the continuous mapping theorem is analogous to Theorem 1 in \citet{kasy2019uniformity}. The proof is almost identical, but we repeat it here for completeness.

\begin{prop}\label{prop:Uctsmapping}
    Let $(\mathbb{D}_1,d_1)$ and $(\mathbb{D}_2,d_2)$ be metric spaces, and let $\Phi \colon \mathbb{D}_1 \longrightarrow \mathbb{D}_2$ be a Lipschitz continuous map.
    Let $(X_n^\theta)_{n\in \mathbb{N}, \theta \in \Theta}$ and $(X^\theta)_{\theta \in \Theta}$ be collections of $\mathbb{D}_1$-valued random variables, and assume $(X^\theta)_{\theta \in \Theta}$ is separable. 
    
    If $X_n^\theta\convUD X^\theta$ in $\mathbb{D}_1$, then $\Phi(X_n^\theta)\convUD \Phi(X^\theta)$ in $\mathbb{D}_2$.
\end{prop}
\begin{proof}
    Note first that if $X^\theta$ is in a separable subset $\mathbb{D}_0 \subseteq \mathbb{D}_1$, then the variables $\Phi(X_\theta)$ for $\theta \in \Theta$ 
    are all in the separable subset $\Phi(\mathbb{D}_0)\subseteq \mathbb{D}_2$. 
    Hence it is well-defined to consider uniform convergence in distribution towards $(\Phi(X^\theta))_{\theta \in \Theta}$. Let $f\in BL_1(\mathbb{D}_2)$ and let $K$ be the Lipschitz constant of $\Phi$. Consider the map 
    \[
        g\colon \mathbb{D}_1\longrightarrow \mathbb{R}, \qquad g(x) = \min(1,K^{-1}) f(\Phi(x)).
    \]
    Then $\|g\|_\infty\leq \|f\|_\infty \leq 1$ and for all $x,y\in \mathbb{D}_1$,
    \begin{align*}
        |g(x)-g(y)| &\leq \min(1,K^{-1}) d_2(\Phi(x),\Phi(y)) \\
        &\leq \min(1,K^{-1})K d_1(x,y)\leq d_1(x,y)
    \end{align*}
    Hence $g\in BL_1(\mathbb{D}_1)$. It follows that
    \begin{align*}
        d_{BL_1(\mathbb{D}_2)}(\Phi(X_n^\theta),\Phi(X^\theta)) 
        &= \sup_{f\in BL_1(\mathbb{D}_2)} |\mathbb{E}[f(\Phi(X_n^\theta))-f(\Phi(X^\theta))]|\\
        &\leq \frac{1}{\min(1,K^{-1})} \sup_{g\in BL_1(\mathbb{D}_1)} |\mathbb{E}[g(X_n^\theta)-g(X^\theta)]|\\
        &\leq \max(1,K) \cdot d_{BL_1(\mathbb{D}_1)}(X_n^\theta,X^\theta) 
    \end{align*}
    Taking the supremum over $\Theta$ and the limit superior as $n\to \infty$ finish the proof.
\end{proof}

We will also need the following two notions of tightness.

\begin{dfn} \label{dfn:tight}
Let $(\mu^\theta)_{\theta \in \Theta}$ be a family of probability measures on $\mathbb{D}$, and let $(X^\theta)_{\theta \in \Theta}$ and $(X_n^\theta)_{n \in \mathbb{N}, \theta \in \Theta}$ be collections of $\mathbb{D}$-valued random variables.
\begin{itemize}
    \item[i)] We say that $(\mu^\theta)_{\theta \in \Theta}$ is \emph{tight} if for any $\varepsilon > 0$, there exists a compact set $K\subseteq \mathbb{D}$ such that
    $\sup_{\theta \in \Theta} \mu^\theta(K^c)  < \varepsilon$.
    We say that $(X^\theta)_{\theta \in \Theta}$ is \emph{uniformly tight} if the collection of distributions $(X^\theta(\mathbb{P}))_{\theta \in \Theta}$ is tight.
    
    \item[ii)] The sequence $((X_n^\theta)_{\theta \in \Theta})_{n\in \mathbb{N}}$ of collections is said to be \emph{sequentially tight} if for any sequence $(\theta_n)_{n \in \mathbb{N}} \subset \Theta$, the sequence of distributions $(X_{n}^{\theta_n}(\mathbb{P}))_{n \in \mathbb{N}}$ is tight.
\end{itemize}
\end{dfn}
Definition \ref{dfn:tight} i) is a classical concept, whereas sequential tightness was introduced by \citet{lundborg2021conditional} and relaxes uniform tightness for sequences of variables parametrized over an infinite set.

The importance of tightness is mainly due to Prokhorov's theorem \citep[Theorem 23.2]{kallenberg2021foundations}, which states that if $\mathbb{D}$ is a Polish space\footnote{The `only if' part does not require separability nor completeness.}, then $(\mu^\theta)_{\theta \in \Theta}$ is tight if and only if all sequences in $(\mu^\theta)_{\theta \in \Theta}$ have a weakly convergent subsequence.

The continuous mapping theorem in Proposition \ref{prop:Uctsmapping} is more restrictive than the classical theorem as it requires Lipschitz continuity. However, we also have an alternative version of uniform continuous mapping when the limit variable is tight.
\begin{prop}\label{prop:UctsmappingT}
    Let $(\mathbb{D}_1,d_1)$ and $(\mathbb{D}_2,d_2)$ be Polish spaces, and let $(X_n^\theta)_{n\in \mathbb{N}, \theta \in \Theta}$ and $(X^\theta)_{\theta \in \Theta}$ be collections of $\mathbb{D}_1$-valued random variables. Assume $(X^\theta)_{\theta \in \Theta}$ is uniformly tight, and let $\Phi \colon \mathbb{D}_1 \longrightarrow \mathbb{D}_2$ be a map that is continuous on the support of $(X^\theta)_{\theta \in \Theta}$.
    
    If $X_n^\theta\convUD X^\theta$ in $\mathbb{D}_1$, then $\Phi(X_n^\theta)\convUD \Phi(X^\theta)$ in $\mathbb{D}_2$.
\end{prop}
\begin{proof}
    Same as the proof of Proposition 10 in \citet{lundborg2021conditional}, but with norms of differences replaced by metric distances.
\end{proof}

\subsection{Uniform stochastic convergence in Skorokhod space}
In this section we consider the special case where $(\mathbb{D},d_\mathbb{D})$ is the Skorokhod space $(D[0,1],d^\circ)$.  We can also equip $D[0,1]$ with the uniform norm, $\|x\|_\infty = \sup_{t\in [0,1]}|x_t|$, and it known that weak convergence based on either $\|\cdot\|_\infty$ or $d^\circ$ are equivalent when the limit is continuous. We now extend this result to stochastic convergence uniformly over $\Theta$. 

\begin{prop}[Skorokhod equivalence]
\label{prop:convergencetoCont2}
    Let $(X_n^\theta)_{n\in \mathbb{N},\theta \in \Theta}$ be a collection of $D[0,1]$-valued random variables and let $(X^\theta)_{\theta \in \Theta}$ be a uniformly tight collection of $C[0,1]$-valued random variables. Then $X_n^\theta \convUD X^\theta$ in $(D[0,1], d^\circ)$ if and only if $X_n^\theta \convUD X^\theta$ in $(D[0,1], \|\cdot\|_\infty)$. In the affirmative, $\|X_n^\theta\|_\infty \convUD \|X\|_\infty$.
\end{prop}
\begin{proof}
    To avoid ambiguity in the topology on $D[0,1]$, we will throughout this proof use $\mathbb{D}^\circ$ to denote the metric space $(D[0,1],d^\circ)$ and we use $\mathbb{D}_\infty$ to denote the Banach space $(D[0,1], \|\cdot\|_\infty)$. Note also that $C[0,1]$ is separable within $\mathbb{D}_\infty$, so $(X^\theta)_{\theta\in\Theta}$ is separable, and hence the convergence $X_n^\theta \convUD X^\theta$ is well-defined in the non-separable space $\mathbb{D}_\infty$. 
    
    The `if' part is clear since $d^\circ(x,y)\leq \|x-y\|_\infty$ for all $x,y\in D[0,1]$. 
    
    For the `only if' part, assume that $X_n^\theta \convUD X^\theta$ in $\mathbb{D}^\circ$ and let $(\theta_n)\subseteq \Theta$ be an arbitrary sequence. Since $(X^{\theta_n}(\mathbb{P}))$ is tight, Prokhorov's Theorem asserts that there exists a subsequence 
    $(\theta_{k(n)})$ and a probability distribution $\mu$ on $C[0,1]$ such that $X^{\theta_{k(n)}}(\mathbb{P}) \xrightarrow{wk} \mu$ in $\mathbb{D}^\circ$. By the triangle inequality
    \begin{align*}
        d_{BL(\mathbb{D}^\circ)}(X_{k(n)}^{\theta_{k(n)}},\mu) 
            \leq 
            d_{BL(\mathbb{D}^\circ)}(X_{k(n)}^{\theta_{k(n)}},X^{\theta_{k(n)}}) 
            + d_{BL(\mathbb{D}^\circ)}(X^{\theta_{k(n)}},\mu) \to 0, \quad n \to 0.
    \end{align*}
    This shows that also $X_{k(n)}^{\theta_{k(n)}}(\mathbb{P})\xrightarrow{wk} \mu$ in $\mathbb{D}^\circ$. Now we can use that weak convergence in the Skorokhod topology and the uniform topology are equivalent when the limit is continuous \citep[Theorem 23.9 (iii)]{kallenberg2021foundations}.
    We therefore conclude that the convergences $X^{\theta_{k(n)}}(\mathbb{P}) \xrightarrow{wk} \mu$ and $X_{k(n)}^{\theta_{k(n)}}(\mathbb{P})\xrightarrow{wk} \mu$ also hold in $\mathbb{D}_\infty$. 
    But then another use of the triangle inequality shows that 
    $$
        d_{BL(\mathbb{D}_\infty)}(X_{k(n)}^{\theta_{k(n)}},X^{\theta_{k(n)}})  
        \leq 
        d_{BL(\mathbb{D}_\infty)}(X_{k(n)}^{\theta_{k(n)}},\mu)
        + d_{BL(\mathbb{D}_\infty)}(\mu,X^{\theta_{k(n)}}) \to 0.
    $$
    Since $(\theta_{k(n)})$ is a subsequence of the arbitrarily chosen sequence $(\theta_n)$, we conclude that $X_n^\theta \convUD X^\theta$ in $\mathbb{D}_\infty$ by Proposition \ref{prop:seqUni}.
    
    Finally, as the uniform norm is Lipschitz continuous as a map from $\mathbb{D}_\infty$ to $\mathbb{R}$, the continuous mapping theorem formulated in Proposition \ref{prop:Uctsmapping}
    yields that
    \begin{align*}
        X_n^\theta \convUD X^\theta 
            \: \text{  in  } \: \mathbb{D}_\infty
        \qquad \implies \qquad 
        \|X_n^\theta\|_\infty \convUD \|X\|_\infty.
    \end{align*}
    This establishes the last part of the lemma.
\end{proof}

Using $\|\mu\|_\infty$ to denote the pushforward measure for any $\mu \in \mathcal{M}_1(D([0,1]))$ we restate the result above for a fixed limit distribution.

\begin{cor}\label{cor:convergencetoCont}
    Let $(X_n^\theta)_{n\in \mathbb{N},\theta \in \Theta}$ be a collection of $D[0,1]$-valued random variables and let $\mu$ be a probability measure on $C[0,1]$.
    Then $X_n^\theta \convUD \mu$ in $(D[0,1], d^\circ)$ if and only if $X_n^\theta \convUD \mu$ in $(D[0,1], \|\cdot\|_\infty)$. In the affirmative, $\|X_n^\theta\|_\infty \convUD \|\mu\|_\infty$.
\end{cor}
\begin{proof}
    Since $\mu$ is a probability measure on the Polish space $C[0,1]$, it is, in particular, tight \citep[Theorem 1.3]{billingsley2013convergence}. Hence the statement is a special case of Proposition~\ref{prop:convergencetoCont2}. 
\end{proof}

Now we are ready to prove a uniform version of Slutsky's theorem in the Skorokhod space.

\begin{lem}[Uniform Slutsky in Skorokhod space]\label{lem:SkorokhodSlutsky}
    Let $(X^\theta, X_n^\theta, Y_n^\theta)_{n\in \mathbb{N},\theta \in \Theta}$ be a collection of $D[0,1]$-valued random variables such that $Y_n^\theta \convUP 0$ and $X_n^\theta\convUD X^\theta$ in $D[0,1]$. 
    Then it holds that $X_n^\theta+Y_n^\theta \convUD X^\theta$.
\end{lem}
\begin{proof}
    Since $Y_n^\theta \convUP 0$, Corollary \ref{cor:convergencetoCont} implies that $\|Y_n^\theta\|_\infty \convUD 0$, and Corollary \ref{cor:ConvUDtoconstant} implies that $\|Y_n^\theta\|_\infty \convUP 0$. 
    Using the trivial estimate $d^\circ(x+y,x) \leq \|(x+y)-x\|_\infty = \|y\|_\infty$ for $x,y\in D[0,1]$, it follows that 
    $d^\circ(X_n^\theta+Y_n^\theta,X_n^\theta) \convUP 0$.
    Combining the latter with $X_n^\theta\convUD X^\theta$, the desired conclusion now follows from Lemma \ref{lem:GeneralSlutsky}.
\end{proof}

We also have a related result for sums of independent sequences.
\begin{lem}\label{lem:sumofindependent}
    Let $(X_n^\theta, Y_n^\theta)_{n\in\mathbb{N},\theta \in \Theta}$ be a collection of $D[0,1]$-valued random variables and let $(X^\theta)_{\theta \in \Theta}$ and $(Y^\theta)_{\theta \in \Theta}$ be uniformly tight collections of $C[0,1]$-valued random variables.
    Assume that $X_n^\theta \convUD X^\theta$ and $Y_n^\theta \convUD Y^\theta$ in $D[0,1]$, and that for each $\theta \in \Theta$ and $n\in \mathbb{N}$, it holds that 
    $X_n^\theta \perp \!\!\!\! \perp Y_n^\theta$.
    Let $Z^\theta$ have distribution 
    $X^\theta(\mathbb{P})* Y^\theta(\mathbb{P})$, that is, the same distribution as the sum of two independent copies of each of $X^\theta$ and $Y^\theta$. 
    
    Then it also holds that 
    $X_n^\theta + Y_n^\theta \convUD Z^\theta $
    in $(D[0,1],\|\cdot\|_\infty)$.
\end{lem}
\begin{proof}
    We may assume without loss of generality that $X^\theta \perp \!\!\!\! \perp Y^\theta$ and that $Z^\theta = X^\theta + Y^\theta$. Let $(\theta_n)\subseteq\Theta$ be an arbitrary sequence. By tightness of $(X^\theta)_{\theta \in \Theta}$ and $(Y^\theta)_{\theta \in \Theta}$, we can apply Prokhorov's theorem twice to obtain probability measures $\mu$ and $\nu$ on $C[0,1]$, and a subsequence $(\theta_{k(n)})$, such that 
    $X^{\theta_{k(n)}}(\mathbb{P}) \xrightarrow{wk} \mu$
    and 
    $Y^{\theta_{k(n)}}(\mathbb{P}) \xrightarrow{wk} \nu$.
    Hence the product measures converge,
    $$
    X^{\theta_{k(n)}}(\mathbb{P}) \otimes Y^{\theta_{k(n)}}(\mathbb{P}) \xrightarrow{wk} \mu \otimes \nu,
    $$
    in $C[0,1]$ as $n\to \infty$, see, for example, Theorem 2.8 (ii) in \citet{billingsley2013convergence}. 

    Since $X_n^\theta \convUD X^\theta$ in $D[0,1]$ by assumption and $(X^\theta)$ is uniformly tight in $C[0,1]$, Proposition \ref{prop:convergencetoCont2} implies that the convergence also holds in $(D[0,1],\|\cdot\|_\infty)$.
    The triangle inequality now yields
    \begin{align*}
        d_{BL}(X_{k(n)}^{\theta_{k(n)}}, \mu) 
            \leq d_{BL}(X_{k(n)}^{\theta_{k(n)}}, X^{\theta_{k(n)}}) 
            + d_{BL}( X^{\theta_{k(n)}}, \mu) \to 0,
    \end{align*}
    so also $X_{k(n)}^{\theta_{k(n)}}(\mathbb{P}) \xrightarrow{wk} \mu$ in $(D[0,1],\|\cdot\|_\infty)$. An analogous computation shows that $Y_{k(n)}^{\theta_{k(n)}}(\mathbb{P}) \xrightarrow{wk} \nu$, and hence also
    \begin{align*}
        X_{k(n)}^{\theta_{k(n)}}(\mathbb{P}) \otimes Y_{k(n)}^{\theta_{k(n)}}(\mathbb{P}) 
        \xrightarrow{wk} \mu \otimes \nu
    \end{align*}
    in the product space $D[0,1]\times D[0,1]$ endowed with the uniform product topology. From the independence statements 
    $X^\theta \perp \!\!\!\! \perp Y^\theta$ and 
    $X_n^\theta \perp \!\!\!\! \perp Y_n^\theta$, we have thus shown that
    \begin{equation*}
        (X^{\theta_{k(n)}},Y^{\theta_{k(n)}}) \xrightarrow{\mathcal{D}} \mu \otimes \nu
        \quad \text{and} \quad
       (X_{k(n)}^{\theta_{k(n)}},Y_{k(n)}^{\theta_{k(n)}})
       \xrightarrow{\mathcal{D}} \mu \otimes \nu
    \end{equation*}
    in the uniform product topology. Since addition $+\colon D[0,1]\times D[0,1] \to D[0,1]$ is continuous with respect to this topology, we conclude by the classical continuous mapping theorem that
    \begin{align*}
        Z^{\theta_{k(n)}} = X^{\theta_{k(n)}}+Y^{\theta_{k(n)}} 
            \xrightarrow{\mathcal{D}}
            \mu * \nu
        \quad \text{and} \quad
        X_{k(n)}^{\theta_{k(n)}}+Y_{k(n)}^{\theta_{k(n)}}
            \xrightarrow{\mathcal{D}} \mu * \nu.
    \end{align*}
    It now follows that
    \begin{align*}
        d_{BL}&(X_{k(n)}^{\theta_{k(n)}}+Y_{k(n)}^{\theta_{k(n)}},Z^{\theta_{k(n)}}) \\
        &\leq 
        d_{BL}(X_{k(n)}^{\theta_{k(n)}}+Y_{k(n)}^{\theta_{k(n)}}, \mu * \nu)
        +
        d_{BL}(\mu * \nu, Z^{\theta_{k(n)}}) \to 0.
    \end{align*}
    Since $(\theta_{k(n)})$ is a subsequence of the arbitrarily chosen sequence $(\theta_n)$, we conclude that $X_n^\theta + Y_n^\theta \convUD Z^\theta $ in $(D[0,1],\|\cdot\|_\infty)$ by Proposition \ref{prop:seqUni}.
\end{proof}

We also need the following lemma, which is a generalization of the classical result: pointwise convergence of a sequence of monotone functions towards a continuous limit is in fact uniform over compact intervals. 

\begin{lem} \label{lem:convergenceofincreasing}
    Let $(X_n^\theta)_{n\in \mathbb{N},\theta \in \Theta}$ be a collection of $D[0,1]$-valued random variables with non-decreasing sample paths. 
    Let $(f^\theta)_{\theta \in \Theta}\subset C[0,1]$ be a uniformly equicontinuous collection of non-decreasing functions. If $X_n^\theta(t)\convUP f^\theta(t)$ for each $t\in [0,1]$, then it also holds that 
    \begin{align*}
        \sup_{t\in [0,1]}|X_n^\theta(t)-f^\theta(t)| \convUP 0.
    \end{align*}
\end{lem}
\begin{proof}
    Let $\epsilon>0$. By uniform equicontinuity we can find 
    $0=t_1<\cdots < t_k = 1$ such that $f^\theta(t_i)-f^\theta(t_{i-1})<\epsilon/2$ for all $\theta$ and $i$. Using that $X_n^\theta$ and $f^\theta$ are non-decreasing, we observe that for $t_{i-1}\leq t \leq t_{i}$:
    \begin{align*}
        X_n^\theta(t)-f^\theta(t) &\leq X_n^\theta(t_i)-f^\theta(t_i) + \epsilon/2, \\
        X_n^\theta(t)-f^\theta(t) &\geq X_n^\theta(t_{i-1})-f^\theta(t_{i-1}) - \epsilon/2.
    \end{align*}
    Combining the inequalities over the entire grid we have 
    \[
            \sup_{t\in [0,1]}|X_n^\theta(t)-f^\theta(t)| \leq \max_{i=0,\ldots,k} |X_n^\theta(t_i)-f^\theta(t_i)| + \epsilon/2.
    \]
    By assumption, $X_n^\theta(t)\convUP f^\theta(t)$ for each $t$, and in particular
    \[
        \max_{i=0,\ldots,k} |X_n^\theta(t_i)-f^\theta(t_i)| \convUP 0
    \]
    as $n\to \infty$. We therefore conclude that
    \begin{align*}
        \sup_{\theta \in \Theta}
        \mathbb{P}
        \Big(\sup_{t\in [0,1]}|X_n^\theta(t)-f^\theta(t)| > \epsilon\Big) 
        \leq 
        \sup_{\theta \in \Theta}
        \mathbb{P}
        \Big(\max_{i=0,\ldots,k} |X_n^\theta(t_i)-f^\theta(t_i)| > \epsilon/2\Big) 
                \longrightarrow 0
    \end{align*}
    as $n\to \infty$.
\end{proof}

The last auxiliary result of this section is an example of Prokhorov's method of ``tightness + identification of limit''.

\begin{lem}\label{lem:ProkhorovsPrinciple}
    Let $(\mathbb{D},d_\mathbb{D})$ be either $(C[0,1],\|\cdot\|_\infty)$ or $(D[0,1],d^\circ)$, and let $(X^\theta, X_n^\theta)_{n\in \mathbb{N},\theta \in \Theta}$ be a collection of $\mathbb{D}$-valued random variables with $(X^\theta)_{\theta \in \Theta}$ separable. Suppose that
    \begin{itemize}
        \item The finite dimensional marginals converge uniformly:        for any $0\leq t_1< \cdots < t_k\leq 1$
            \begin{align*}
                \pi_{t_1,\ldots, t_k}(X_n^\theta) \convUD 
                \pi_{t_1,\ldots, t_k}(X^\theta), 
                    \qquad n\to \infty,
            \end{align*}
            where $\pi_{t_1,\ldots, t_k} \colon \mathbb{D} \to \mathbb{R}^k$ is the projection given by $\pi_{t_1,\ldots, t_k}(x) = (x(t_1), \ldots, x(t_k))$.
        \item $(X_n^\theta)_{n\in \mathbb{N},\theta \in \Theta}$ is sequentially tight.
        \item $(X^\theta)_{n\in \mathbb{N},\theta \in \Theta}$ is uniformly tight.
    \end{itemize}
    Then $X_n^\theta \convUD X^\theta$ as $n\to \infty$.
\end{lem}
\begin{proof}
    The statement is analogous to Proposition 18 in \citet{lundborg2021conditional}, the difference being that the functionals $\langle \cdot, h\rangle$ in 
    \cite{lundborg2021conditional} have been replaced by the functionals $\pi_{t_1,\ldots,t_k}$.
    
    The proof of \citet{lundborg2021conditional} also works in our case, given that the finite dimensional marginals form a separating class for the both the Borel algebra on $C[0,1]$ and the Borel algebra on $D[0,1]$. This is established in \citet{billingsley2013convergence}, Example 1.3 and Theorem 12.5 (iii).
\end{proof}

\subsection{Chaining in time uniformly over a parameter} \label{app:UniformChaining}
We extend the basic chaining arguments to hold uniformly over $\Theta$. Our arguments closely follow those of \citet[Chapter VII.2.]{Pollard:1984} and \citet{Newey:1991}. The results are formulated for processes indexed over a general metric space $T$, but we will only apply the results in the case $T=[0,1]$. We have the following extension of stochastic equicontinuity to the uniform setting.
\begin{dfn}\label{dfn:UniformChaining}
    A collection of sequences
    $$
        \big(Z^{(n),\theta}\big)_{n\in \mathbb{N},\theta \in \Theta}
        =
        \Big(Z_t^{(n),\theta}\Big)_{t\in T, n\in \mathbb{N},\theta \in \Theta}
    $$
    of stochastic processes indexed over a metric space $(T,d)$ is called \emph{stochastically equicontinuous uniformly over $\Theta$} if for all $\epsilon, \eta > 0$ there exists $\delta >0$ such that
    \begin{align*}
        \limsup_{n\to \infty} \sup_{\theta \in \Theta}\mathbb{P}\Big(\sup_{s,t\in T \colon d(s,t)\leq \delta} \big|Z_s^{(n),\theta}-Z_t^{(n),\theta}\big| 
                > \epsilon\Big)
            < \eta.
    \end{align*}
\end{dfn}
In Section 2.8.2 of \citet{Vaart:1996}, the same definition is given in the context of empirical processes.
Recall that we write, e.g., $Z^{(n)}$ as a shorthand for $Z^{(n),\theta}$ and let the dependency on $\theta$ be implicit for notational ease. We also write $\sup_{d(s,t)\leq \delta}$ as a shorthand for $\sup_{s,t\in T \colon d(s,t)\leq \delta}$.
Definition \ref{dfn:UniformChaining} is a direct extension of pointwise stochastic equicontinuity.
Accordingly, Theorem 2.1 from \citet{Newey:1991} generalizes as follows:
\begin{lem}\label{lem:uniformequcont}
    Let $(Z_t^{(n)})_{t\in T, n\in \mathbb{N}}$ be a sequence of stochastic processes indexed by a compact metric space $T$. Assume that $(Z_t^{(n)})$ is stochastically equicontinuous uniformly over $\Theta$ and that for each $t\in T$ it holds that $Z_t^{(n)} \convUP 0$. Then $\sup_{t\in T} |Z_t^{(n)}|\convUP 0$ as $n\to \infty$.
\end{lem}
\begin{proof}
    Let $\varepsilon,\eta>0$ be given, and let $\delta>0$ be the corresponding distance obtained from the uniform stochastic equicontinuity of $(Z^{(n)})$. By compactness of $T$ there exists a finite set $T^*\subseteq T$ such that $T = \bigcup_{t\in T^*}B(t,\delta)$. By the triangle inequality we get that
    \begin{align*}
        \sup_{t\in T} |Z_t^{(n)}|
        = \sup_{t\in T^*}\sup_{s\in B(t,\delta)} |Z_s^{(n)}|
        \leq 
            \sup_{t\in T^*}|Z_t^{(n)}| +
            \sup_{t\in T^*}\sup_{s\in B(t,\delta)} |Z_s^{(n)}-Z_t^{(n)}|.
    \end{align*}
    Since $T^*$ is finite, it follows that $\sup_{t\in T^*} |Z_t^{(n)}|\convUP 0$, which combined with the inequality implies that
    \begin{align*}
        &\limsup_{n \to \infty} \sup_{\theta \in \Theta} \mathbb{P}
            \big(\sup_{t\in T} |Z_t^{(n)}| > 2\varepsilon\big) \\
        &\leq 0 + 
        \limsup_{n \to \infty} \sup_{\theta \in \Theta}
        \mathbb{P}\Big(
            \sup_{t\in T^*}\sup_{s\in B(t,\delta)} |Z_t^{(n)}-Z_t^{(n)}|>\varepsilon
        \Big)
        \leq \eta.
    \end{align*}
    As $\varepsilon,\eta>0$ were chosen arbitrarily, we conclude that $\sup_{t\in T} |Z_t^{(n)}| \convUP 0$.
\end{proof}
To establish uniform stochastic equicontinuity we extend the chaining lemma to a uniform setting. To formulate the theorem we first need some classical definitions related to chaining.
\begin{dfn}
    Let $T$ be a compact metric space. A subset $T^* \subseteq T$ is called a $\delta$-net if $\bigcup_{t\in T^*} B(t,\delta) = T$. 
    The covering number
    $$
        N(\delta) =N(\delta,T) \coloneqq \min\{\,|T^*| \colon
            T^*\subseteq T, T^* \text{ is a } \delta\text{-net}\,\}
    $$
    is the smallest possible cardinality of a $\delta$-net, which is finite by compactness. The associated covering integral is
    \begin{align*}
        J(\delta) = \int_0^\delta \pa{2 \log(N(\epsilon)/\epsilon)}^{\frac{1}{2}}\mathrm{d}\epsilon,
        \qquad 0\leq \delta \leq 1.
    \end{align*}
\end{dfn}

\begin{lem}
    Let $(T,d)$ be a metric space with finite covering integral $J(\cdot)$ and let $(Z_t^\theta)_{t\in T,\theta \in \Theta}$ be a collection of stochastic processes indexed by $T$ with continuous sample paths. Assume there is a uniform constant $\varsigma>0$ such that, for all $s,t\in T$ and $\eta>0$,
    \begin{align*}
        \sup_{\theta \in \Theta} \mathbb{P}\pa{|Z_s^\theta - Z_t^\theta| > \eta \cdot d(s,t)}
            \leq 2 e^{- \frac{\eta^2}{2\varsigma^2}}.
    \end{align*}
    Then, for all $0<\epsilon<1$,
    \begin{align*}
        \sup_{\theta \in \Theta} \mathbb{P}
        \Big(\sup_{d(s,t)\leq \epsilon}|Z_s^\theta - Z_t^\theta| > 26 \varsigma J(\epsilon)\Big)
        \leq 2 \epsilon.
    \end{align*}
\end{lem}
\begin{proof}
    The lemma is a direct consequence of classical chaining lemma \citep[page 144]{Pollard:1984}. For each $\theta \in \Theta$, the conditions of the chaining lemma are met for $(Z_t^\theta)_{t\in T}$ with sub-exponential factor $\varsigma$. This implies, in particular, that for any $\theta \in \Theta$ and $0<\epsilon<1$,
    \begin{align*}
        \mathbb{P}
        \Big(\sup_{d(s,t)\leq \epsilon}|Z_s^\theta - Z_t^\theta| > 26 \varsigma J(\epsilon)\Big)
        \leq 2 \epsilon,
    \end{align*}
    which is equivalent to the conclusion of the lemma.
\end{proof}

This immediately implies the following corollary.
\begin{cor}\label{cor:uniformchaining}
    Let $(T,d)$ be a metric space with finite covering integral $J(\cdot)$ and let $(Z^{(n),\theta})$ be a sequence of stochastic processes on $T$ with continuous sample paths. Assume there exists a constant $\varsigma>0$ such that, for all $s,t\in T$ and $\eta>0$ and $n\in \mathbb{N}$,
    \begin{align*}
        \sup_{\theta \in \Theta} \mathbb{P} \pa{|Z_s^{(n),\theta} - Z_t^{(n),\theta}| > \eta \cdot d(s,t)}
            \leq 2 e^{- \frac{\eta^2}{2\varsigma^2}}.
    \end{align*}
    Then $(Z^{(n)})$ is stochastically equicontinuous uniformly over $\Theta$.
\end{cor}

For stochastic processes with continuous sample paths, stochastic equicontinuity turns out to be equivalent to sequential tightness (Definition \ref{dfn:tight} ii)).
\begin{prop}\label{prop:equicontinuityistightness}
    Let $(Z^{(n),\theta})_{n\in \mathbb{N},\theta\in\Theta}$ be a collection of $C[0,1]$-valued random variables such that $\mathbb{P}(Z_0^{(n),\theta}=0)=1$ all $n\in \mathbb{N}$ and $\theta \in \Theta$.
    The following are equivalent:
    \begin{enumerate}
        \item $(Z^{(n),\theta})$ is stochastically equicontinuous uniformly over $\Theta$.
        
        \item $(Z^{(n),\theta})$ is sequentially tight.
    \end{enumerate}
\end{prop}
\begin{proof}
    The equivalence is a straightforward application of Theorem 7.3 in \citet{billingsley2013convergence}. 
    Condition $(i)$ of the aforementioned theorem is satisfied for any sequence of measures from the collection $(Z^{(n),\theta}(\mathbb{P}))_{n\in \mathbb{N},\theta \in \Theta}$, since $Z_0^{(n),\theta}=0$ almost surely for all $n$ and $\theta$.
    For any sequence $(\theta_n)\subseteq \Theta$, stochastic equicontinuity uniformly over $\Theta$ implies condition $(ii)$ of Theorem 7.3 in \citet{billingsley2013convergence} for the measures $((Z^{(n),\theta_n})(\mathbb{P}))$. 
    We therefore conclude that stochastic equicontinuity uniformly over $\Theta$ implies sequential tightness. 
    
    On the contrary, assume that $(Z^{(n),\theta})$ is sequentially tight and let $\epsilon,\eta>0$ be given. For each $n$, choose $\theta_n$ such that 
    \begin{align*}
        \sup_{\theta\in\Theta} \mathbb{P} \Big(\sup_{|s-t|\leq \delta} \big|Z_s^{(n),\theta}-Z_t^{(n),\theta}\big| 
                \geq \epsilon\Big)
        \leq \mathbb{P}\Big(\sup_{|s-t|\leq \delta} \big|Z_s^{(n),\theta_n}-Z_t^{(n),\theta_n}\big| 
                \geq \epsilon\Big) + \frac{1}{n}.
    \end{align*}
    Since $((Z^{(n),\theta_n})(\mathbb{P}))$ is tight by assumption, condition $(ii)$ of Theorem 7.3 asserts that there exists $\delta, N>0$ such that
    $$ 
        \mathbb{P}\Big(\sup_{|s-t|\leq \delta} \big|Z_s^{(n),\theta_n}-Z_t^{(n),\theta_n}\big| 
                \geq \epsilon\Big) < \eta
    $$ 
    for $n\geq N$. Combining both inequalities and taking the limit superior finish the proof.
\end{proof}


\section{The Functional Martingale CLT} \label{sec:fclt}
In this section we state Rebolledo's martingale CLT \citep{rebolledo1980central} based on its formulation in \citet{AndersenBorganGillKeiding:1993}, and then we extend the result to a uniform version without fixed variance functions.
The one-dimensional case suffices for our purpose, so for simplicity, every local martingale in the following is a real-valued stochastic process. For a local square integrable martingale $(M_t)$, we let $\langle M \rangle(t)$ denote its quadratic characteristic.
The theorem requires a condition on the jumps of the local martingales, for which we will need the following definition.
\begin{dfn}
    Let $M_t$ be a local square integrable $\cF_t$-martingale. For any $\varepsilon>0$, we define $\langle M_\varepsilon \rangle(t)$ to be the quadratic characteristic of the pure jump-process given by 
    $$
        t\mapsto \sum_{0\leq s\leq t} M_s \one(|\Delta M_s|>\varepsilon).
    $$
\end{dfn}
We also need a representation of Gaussian martingales, which ensures their continuity.

\begin{prop} \label{prop:BMrepresentation}
    Let $(B_t)_{t\in[0,\infty)}$ be a Brownian motion on $[0,\infty)$ with continuous sample paths. For every non-decreasing $f\in C[0,1]$, the process $(B_{f(t)})_{t\in [0,1]}$ is a continuous mean zero Gaussian martingale on $[0,1]$ with variance function $f$. 

    Consequently, if $U=(U_t)_{t\in [0,1]}$ is a mean zero Gaussian martingale with a continuous variance function $V$, then $U$ has the distributional representation 
    \begin{equation}\label{eq:BMrepresentation}
            (U_t)_{t\in [0,1]} 
            \stackrel{\mathcal{D}}{=}
            (B_{V(t)})_{t\in [0,1]}.
    \end{equation}
\end{prop}
\begin{proof}
    Let $f\in C[0,1]$ be non-decreasing. From the properties of Brownian motion, it follows directly that the time-transformed process $(B_{f(t)})_{t\in [0,1]}$ is a mean zero Gaussian process with variance function $f$. 
    Since $f$ is continuous, each sample path $t\mapsto B_{f(t)}$ is a composition of continuous functions and thus continuous itself. 
    Since $f$ is non-decreasing, the time-transformation also preserves the martingale property. This establishes the first part.
    
    For the second part, recall that the covariance function of a martingale is determined by its variance function.
    Hence the first part implies that the right-hand side in \eqref{eq:BMrepresentation} is a Gaussian process with the same mean and covariance structure as the left-hand side. Since the distribution of a Gaussian processes is uniquely determined by its mean and covariance structure, the equality in distribution follows.
\end{proof}

Proposition \ref{prop:BMrepresentation} is a simple, distributional variant of the 
Dubins-Schwarz theorem, see \cite{revuz2013continuous}, Chapter V, Theorems 1.6 and 1.7. 
The Dubins-Schwarz theorem implies that, in fact,  $U_t = B_{V(t)}$ for $t\in [0,1]$, 
where $B$ is a Brownian motion on $[0,V(1)]$. For the purpose of this paper we only need 
the simpler, distributional equality \eqref{eq:BMrepresentation}.

We can now formulate Rebolledo's CLT for local martingales. 
To this end, note that Proposition \ref{prop:BMrepresentation} ensures the existence of the continuous 
Gaussian limit martingale $U$ when the variance function $V$ is continuous. 

\begin{thm}[Rebolledo's CLT]\label{thm:fclt}
    Let $(U^{(n)})_{n\in \mathbb{N}}$ be a sequence a local square integrable martingales in $D[0,1]$, possibly defined on different sample spaces and with different filtrations for each $n\in \mathbb{N}$. Let $U$ be a continuous Gaussian martingale with continuous variance function $V\colon [0,1] \to [0,\infty)$, and assume that $U^{(n)}_0 = U_0 = 0$.
    Suppose that for every $t\in [0,1]$ and $\varepsilon>0$,
    \begin{align*}
        \langle U^{(n)} \rangle(t) \xrightarrow{P} V(t)
        \qquad \text{ and } \qquad 
        \langle U_\varepsilon^{(n)} \rangle(t) \xrightarrow{P} 0,
    \end{align*}
    as $n\to \infty$. 
    Then it holds that $U^{(n)} \xrightarrow{\mathcal{D}} U$
    in $D[0,1]$ as $n\to \infty$. 
\end{thm}
\begin{proof}
    This is a special case of Theorem $\mathrm{II}.5.2$ in \citet{AndersenBorganGillKeiding:1993}.
\end{proof}

The general formulation of Rebolledo's CLT above, which allows for $n$-dependent sample spaces and filtrations, can now be leveraged to obtain a uniform version via the sequential characterization of uniform stochastic convergence.

\begin{thm}[Uniform Rebolledo CLT]\label{thm:URebo}
    For each $n \in \mathbb{N}$ and $\theta \in \Theta$:
    \begin{itemize}[leftmargin=15pt]
        \item Let $\mathcal{F}^{(n),\theta}=(\mathcal{F}_t^{(n),\theta})_{t\in [0,1]}$ be a filtration satisfying the usual conditions.
        
        \item Let $U_t^{(n),\theta}$ be a local square integrable $\mathcal{F}_t^{(n),\theta}$-martingale in $D[0,1]$ with $U_0^{(n),\theta}=0$.
        
        \item Let $V^\theta \colon [0,1] \to [0,\infty)$ be a non-decreasing function with $V^\theta(0)=0$.
    \end{itemize}
    Assume that $(V^\theta)_{\theta \in \Theta}$ is uniformly equicontinuous and that $\sup_{\theta \in \Theta} V^\theta(1) < \infty$. 
    Assume further that for every $\varepsilon>0$ and $t\in [0,1]$,
    \begin{align}\label{eq:RebolledoConditions}
        \langle U^{(n),\theta} \rangle(t) \convUP V^\theta(t)
        \qquad \text{ and } \qquad
        \langle U_\varepsilon^{(n),\theta}\rangle(t) \convUP 0,
    \end{align}
    as $n\to \infty$. Then it holds that
    \begin{align*}
        U^{(n),\theta} \convUD U^\theta, \qquad n\to \infty,
    \end{align*}
    in $D[0,1]$ uniformly over $\Theta$, where for each $\theta \in \Theta$, $U^\theta$ is a mean zero continuous Gaussian martingale on $[0,1]$ with variance function $V^\theta$.
\end{thm}
\begin{proof}
    We will use the characterization of uniform convergence as stated in Proposition~\ref{prop:seqUni} c). To this end, let $(\theta_n)\subseteq \Theta$ be an arbitrary sequence. By assumption $(V_{\theta_n})_{n\in \mathbb{N}}$ is a uniformly equicontinuous and bounded sequence of functions on a compact interval, so the Arzelà–Ascoli theorem states that there exists a subsequence $\theta_{k(n)}$, with $k\colon \mathbb{N} \to \mathbb{N}$ strictly increasing, and a function $\tilde{V}\in C[0,1]$ such that 
    \[
        \sup_{t\in [0,1]} |V^{\theta_{k(n)}}(t) - \tilde{V}(t)| 
        \longrightarrow 0, \qquad n \to \infty.
    \]
    Since each function $V^{\theta_{k(n)}}$ is non-decreasing, it follows that $\tilde{V}$ is non-decreasing. It also holds that 
    $\tilde{V}(0)=\lim_{n\to \infty}V^{\theta_{k(n)}}(0)=0$, and therefore $\tilde{V}$ is 
    the variance function of a continuous Gaussian martingale $\tilde{U}$ with $\tilde{U}_0 = 0$. 
    
    By assumption of the convergences in \eqref{eq:RebolledoConditions}, we may conclude that 
    \begin{align*}
        |\langle U^{(k(n)),\theta_{k(n)}} \rangle(t) - \tilde{V}(t)|
        \leq
        \underbrace{
            |\langle U^{(k(n)),\theta_{k(n)}} \rangle(t) - V^{\theta_{k(n)}}(t)|
        }_{\xrightarrow{P} 0}
        +
        \underbrace{
        |V^{\theta_{k(n)}}(t) - \tilde{V}(t)|
        }_{\to 0}
        \xrightarrow{P} 0
    \end{align*}
    and that
    \(
        \langle U_\epsilon^{(k(n)),\theta_{k(n)}} \rangle(t) \to 0
    \) as $n\to \infty$. Thus we have established the conditions of the classical Rebolledo CLT -- Theorem \ref{thm:fclt} -- for the sequence $U^{(k(n)),\theta_{k(n)}}$ and the Gaussian martingale $\tilde{U}$ with variance function $\tilde{V}$. We therefore conclude that
    \begin{align*}
        U^{(k(n)),\theta_{k(n)}} 
            \xrightarrow{\mathcal{D}}
        \tilde{U}
    \end{align*}
    in $D[0,1]$ as $n\to \infty$. 

    We now establish that the sequence $(U^{\theta_{k(n)}})$ also converges in distribution to $\tilde{U}$ in $C[0,1]$, and in particular also in $D[0,1]$. To this end, we use the characterization of convergence in distribution in $C[0,1]$ from Theorem 7.5 in \citet{billingsley2013convergence}, which states that we need to show that
    \begin{enumerate}
        \item For all $0\leq t_1 < \cdots < t_m \leq 1$, it holds that
        \[
            (U_{t_1}^{\theta_{k(n)}}, \ldots, U_{t_m}^{\theta_{k(n)}})
                \xrightarrow{\mathcal{D}}
            (\tilde{U}_{t_1},\ldots, \tilde{U}_{t_m}), 
            \qquad n \to \infty.
        \]
        \item For all $\epsilon>0$
        \begin{align*}
            \lim_{\delta\to 0^+} \limsup_{n \to \infty}
            \mathbb{P}\Big(\sup_{|t-s|<\delta}|U_t^{\theta_{k(n)}}
            - U_s^{\theta_{k(n)}}|>\epsilon \Big)=0.
        \end{align*}
    \end{enumerate}
    The first condition is clear since all the marginals are multivariate Gaussian, and the mean and variance of the sequence converges to the mean and variance of the limit distribution. The second condition follows from the same computation as in the proof of Lemma~\ref{lem:Ulimitistight}. By Theorem 7.5 in \citet{billingsley2013convergence} we therefore conclude that 
    \begin{align*}
        U^{\theta_{k(n)}} \xrightarrow{\mathcal{D}} \tilde{U},
        \qquad \text{for} \:\:  n \to \infty,
    \end{align*}
    in $C[0,1]$, and hence also in $D[0,1]$. 

    We can now apply the triangle inequality for the bounded Lipschitz metric to conclude that
    \begin{align*}
        d_{BL}(U^{(k(n)),\theta_{k(n)}}, U^{\theta_{k(n)}} )
        \leq 
        d_{BL}(U^{(k(n)),\theta_{k(n)}}, \tilde{U} )
        +
        d_{BL}( \tilde{U}, U^{\theta_{k(n)}} )
        \longrightarrow 0.
    \end{align*}
    Since $(\theta_n) \subseteq \Theta$ was an arbitrary sequence, we conclude that $U^{(n),\theta} \convUD U^\theta$ by Proposition~\ref{prop:seqUni}.
\end{proof}

The following proposition gives explicit expressions for the quadratic characteristics that 
appear in Rebolledo's CLT in the special case where the local martingales are given as stochastic integrals with respect to a compensated counting processes. 

\begin{prop}\label{prop:Rebolledospecialcase}
    Let $N_1,\ldots, N_n$ be counting processes and assume that for each $j=1,\ldots,n$,
    $N_j$ has an absolutely continuous $\mathcal{F}_t^{(n)}$-compensator $\Lambda_{j,t}$ such that $M_{j,t}=N_{j,t}-\Lambda_{j,t}$ is a locally square integrable $\mathcal{F}_t^{(n)}$-martingale.
    Let $H_1,\ldots,H_n$ be locally bounded $\mathcal{F}_t^{(n)}$-predictable processes, and define the process $U_t^{(n)} = \sum_{j=1}^n \int_0^t H_{j,s} \mathrm{d}M_{j,s}$. Then $U_t^{(n)}$ is a local square integrable $\mathcal{F}_t^{(n)}$-martingale, and for any $t,\epsilon>0$ it holds that
    \begin{align*}
        \langle U^{(n)} \rangle(t) 
            &= \sum_{j=1}^n \int_0^t H_{j,s}^2 \mathrm{d}\Lambda_{j,s}, \\
        \langle U_\epsilon^{(n)} \rangle(t) 
            &= \sum_{j=1}^n \int_0^t H_{j,s}^2\one(|H_{j,s}|\geq \epsilon)\mathrm{d}\Lambda_{j,s}.
    \end{align*}
\end{prop}
\begin{proof}
    See the discussion following Theorem $\mathrm{II}.5.2$ in \citet{AndersenBorganGillKeiding:1993}, in particular equations $(2.5.6)$ and $(2.5.8)$.
\end{proof}

\section{Estimation of \texorpdfstring{$\lambda$}{the intensity} and \texorpdfstring{$G$}{the residualization process}} \label{sec:estimation} 
The asymptotic theory for estimation of the LCM crucially relies on $\hat \lambda^{(n)}$ and $\hat G^{(n)}$ being consistent,
and more importantly, having a product error decaying at an $n^{-1/2}$-rate.
Therefore, a central question when applying the test, is how to model $\lambda$ and $G$. 

In principle, we could use parametric models to learn $\hat \lambda^{(n)}$ and $\hat G^{(n)}$, and under such models it should be possible to achieve $n^{-1/2}$-rates. For example, if we consider a parametrization $(t,\theta) \mapsto \lambda_t(\theta)$ which is $\kappa(t)$-Lipschitz in $\theta\in \Theta \subseteq \mathbb{R}^p$ for each $t$, then
\begin{align*}
    h(n)^2 
    = E\left(\int_0^1 (\lambda_t(\theta_0)-\lambda_t(\hat \theta^{(n)}))^2 
    \mathrm{d}t \right) 
    \leq 
    \|\kappa\|_{L_2([0,1])}^2 E\| \theta_0 - \hat \theta^{(n)}\|_{\mathbb{R}^p}^2.
\end{align*}
Thus the rates from parametric asymptotic theory can be converted to rates for $g$ and $h$. 

However, it is of greater interest if sufficient rates can be achieved with
nonparametric estimators. Below we give
concrete examples of nonparametric models and discuss which rates are
achievable. For simplicity, we focus on the case where $\mathcal{F}_t =
\mathcal{F}_t^{N,Z}$ and where $G_t = X_t - \Pi_t$ as in the introductory example.

\subsection{Nonparametric functional estimation of \texorpdfstring{$\Pi$}{the predictable projection}}
As seen in \cref{subsec:ex}, assumptions on the form of $\Pi$ turn the
general estimation problem into a concrete problem of estimating a function. 

If the system is Markovian, it can be reasonable to assume a
\emph{functional concurrent model}. The model asserts that $\Pi_t = \mu(t,Z_t)$
for a bivariate function $\mu$, and a survey of methods for estimating $\mu$ is
given by \citet{maity2017nonparametric}. Notably, \citet{jiang2011functional}
achieve an $n^{-1/3}$-rate of $g(n)$ under certain regularity and moment
assumptions, see their Theorem 3.3. That result also holds if $Z$ is 
replaced by a linear predictor $\beta^T\mathbf{Z}$ of several covariates.

Consider again the historical linear regression model from \cref{subsec:ex},
and assume that the effect of $Z$ on $X$ is homogeneous over time. 
That is, $\rho_X(s,t) = \tilde \rho_X(t-s)$ for some function $\tilde \rho_X$.
This submodel is known as the
\emph{functional convolution model}, since $\Pi$ can be written as the
convolution of $Z$ and $\tilde \rho_X$. Applying the Fourier transform converts
it into a (complex) linear concurrent model, so by Plancherel's theorem one can leverage the convergence rates from the concurrent model. \citet[Theorem
16]{manrique2016functional} uses this idea to transfer the $n^{-1/4}$-rate of the
\emph{functional ridge regression estimator} \citep{manrique2018ridge} 
to the convolution model, which holds under modest moment conditions on the data. With additional distributional assumptions, we conjecture that faster rate results for the linear concurrent model can also be leveraged to the convolution model. \citet{csenturk2010functional} consider a similar model under
the assumption that
$$
    \rho_X(s,t) = \one(t-\Delta \leq s\leq t)\tilde \rho_X^1(t)\tilde \rho_X^2(t-s),
$$
for two functions $\tilde \rho_X^1$ and $\tilde \rho_X^2$ and a lag $\Delta>0$.
They establish a pointwise rate result for the response curve, but it is not 
obvious how to cast their result as a polynomial rate for $g(n)$.

For the full historical functional linear model we are not aware of any
published rate results. \citet{yuan2010reproducing,cai2012minimax} establish
rates on the prediction error for scalar-on-function regression, and
\citet{Yao:2005} establish various rates for function-on-function regression,
but in a non-historical setting. Based on the former, we give a heuristic for
which rates are achievable for $g(n)$ in this model. If $\hat{\Pi}$ is based on
a kernel estimate $\hat{\rho}_X^{(n)}$ of $\rho_X$, then Tonelli's theorem yields
    \begin{align*}
        g(n)^2 = \vertiii{ 
            \Pi - \hat{\Pi}^{(n)}
        }_{2}^2
        &= \ex \left(
            \int_0^1 \left(
                \int_0^t (\rho_X(s,t)-\hat{\rho}_X^{(n)}(s,t))Z_s\mathrm{d}s
            \right)^2\mathrm{d}t
        \right) \\
        &= \int_0^1 \ex
            \left(\left(
                \int_0^t (\rho_X(s,t)-\hat{\rho}_X^{(n)}(s,t))Z_s\mathrm{d}s
            \right)^2\right)\mathrm{d}t.
    \end{align*}
Theorem 4 in \citet{cai2012minimax} asserts that we, under certain regularity
conditions, can estimate $\rho_X(\cdot,t)$ such that 
\begin{equation*}
    \ex \left(
    \left(
        \int_0^t (\rho_X(s,t)-\hat{\rho}_X^{(n)}(s,t))Z_s\mathrm{d}s
    \right)^2
    \right)
\end{equation*}
decays at a $n^{-2r_t/(2r_t+1)}$-rate for a fixed $t$. Here $r_t$ is a constant
describing the eigenvalue decay of a certain operator related to the 
autocovariance of $Z$ and the regularity of $\rho_X$. As a concrete example, if
$Z$ is a Wiener process and $\rho_X(\cdot,t)\in \mathcal{W}_2^m([0,t])$ is in
the $m$-th Sobolev space for each $t>0$, then $r_t = 1 + m$ and $g(n)$ will
converge at an $n^{-(1+m)/(2m+3)}$-rate, see the discussion after Corollary 8 in
\cite{yuan2010reproducing}. Based on these arguments, we believe that the
desired $n^{-(1/4+\varepsilon)}$-rate for $g(n)$ is achievable with 
suitable regularity assumptions on $Z$ and $\rho_X$. 

\subsection{Estimation of \texorpdfstring{$\lambda$}{the intensity}}
Within the framework of the Cox model, \cite{Wells:1994} demonstrate 
that the baseline intensity can be estimated with rate $n^{-2/5}$ using 
a standard kernel smoothing technique. With the parametric $n^{-1/2}$-rate 
on the remaining parameters, this translates readily into $h(n) = O(n^{-2/5})$.

As an alternative to the Cox model, \citet{bender2020general}
propose a general framework for nonparametric estimation of Markovian
intensities, i.e., $\lambda_t = \exp(f(t,Z_t))$ for some function $f$. They
survey existing methods such as gradient boosted trees and neural networks and
relate them to this setting. Based on real and synthetic data, they find that
both gradient boosted trees and neural networks outperform the Cox
model in terms of predictive performance as measured by the Brier score. In
essence, the framework relies on discretizing time and approximating the
intensity with successive Poisson regressions. Using the same idea,
\citet{rytgaard2021estimation} argue that $h(n) = o(n^{-1/4})$
can be achieved for time-independent covariates. 

Similarly, \citet{rytgaard2022continuous} mention that $h(n) = o(n^{-1/4})$ can
be achieved for estimation of intensities in a multivariate point process with a
uniformly bounded number of events, which we place into a general modeling 
framework below.

\subsection{Estimation of \texorpdfstring{$\lambda$}{the intensity} and \texorpdfstring{$\Pi$}{the predictable projection} for counting processes}
In \cref{sec:setup,sec:asymptotics} we considered the setup where $N$ was a counting
process adapted to a filtration $\mathcal{F}_t$, which could contain information
on baseline covariates and covariate processes that were not necessarily
counting processes. In this section we explore how our testing framework can be
applied when all stochastic processes of interest are
counting processes.

More specifically, let $(N_t^d)_{d \in [p]}$ be 
a $p$-dimensional counting process. For $a, b \in [p]$ and 
$C \subset [p] \setminus \{b\}$ with 
$a \neq b$ and $a \in C$ we are interested in testing 
the hypothesis that $N^a$ is conditionally locally independent 
of $N^b$ given the filtration, $\mathcal{F}^C_t$, generated 
by $N^C = (N^d)_{d \in C}$.

We can cast this setup in the framework of \cref{sec:setup} as follows.
Naturally, we let $N = N^a$ and $\mathcal{F}_t = \mathcal{F}^C_t$. The auxiliary
process $X$ is chosen to be \lc{} and predictable with respect to the
filtration, $\mathcal{F}^b_t$, generated by $N^b$. For example, we could choose
$X_t = N^b_{t-}$. But $X_t$ could be any functional of $N^b$ such
as $X_t = f(N_{t-}^b)$ for a suitable function $f$ or a linear filter of $N^b$,
    \begin{align*}
        X_t = \int_0^{t-} \kappa(t - s) \mathrm{d} N_s^b,
    \end{align*}
where $\kappa$ is a suitable kernel function, see also Section \ref{sec:neyman}.
In principle, the process $X$ could also depend on the process $N^C$, but it is
important that the filtration, $\mathcal{G}_t$, generated by $\mathcal{F}_t$ and
$X_t$ is strictly larger than $\mathcal{F}_t$, i.e., $X_t$ should depend on
$N^b$, in order to get a non-trivial test as explained in \cref{sec:setup}. 

In the framework of counting processes, we can approach the estimation of both
$\lambda$ and $\Pi$ in a unified and general way as follows: Let $(\tau_j,
z_j)_{j \geq 1}$ be the marked point process associated with the counting
process $N^C$, i.e., $(\tau_j)_{j \geq 1}$ is a sequence of almost surely
strictly increasing event times located at the jumps of $N^C$, and $(z_j)_{j
\geq 1}$ for $z_j \in C$ are the corresponding event types. 

Since both $\lambda_t$ and $\Pi_t$ are real-valued and
$\mathcal{F}^C_{t-}$-measurable for each fixed $t \geq 0$, they can be
represented as measurable functions of $\{(\tau_j, z_j) \mid \tau_j < t, z_j \in
C\}$. Hence, we can model both $\lambda$ and $\Pi$ using any sequence-to-number
model. For the intensity process, \citet{rytgaard2022continuous} propose a
sequence of HAL estimators when the total event count is uniformly bounded. As
an alternative, \citet{Xiao:2019} propose using a recurrent neural network
(LSTM). Unless there is a uniform bound on the total number of events, as
assumed by \cite{rytgaard2022continuous}, there are currently no published
results available on the rates of convergence for nonparametric estimation of
sequence-to-number functions. 

\clearpage
\section{Additional simulation figures} \label{sec:extrafigs}
This section contains additional figures related to the simulations of \cref{sec:simulations}.

\begin{figure}[ht]
    \includegraphics[width=\linewidth]{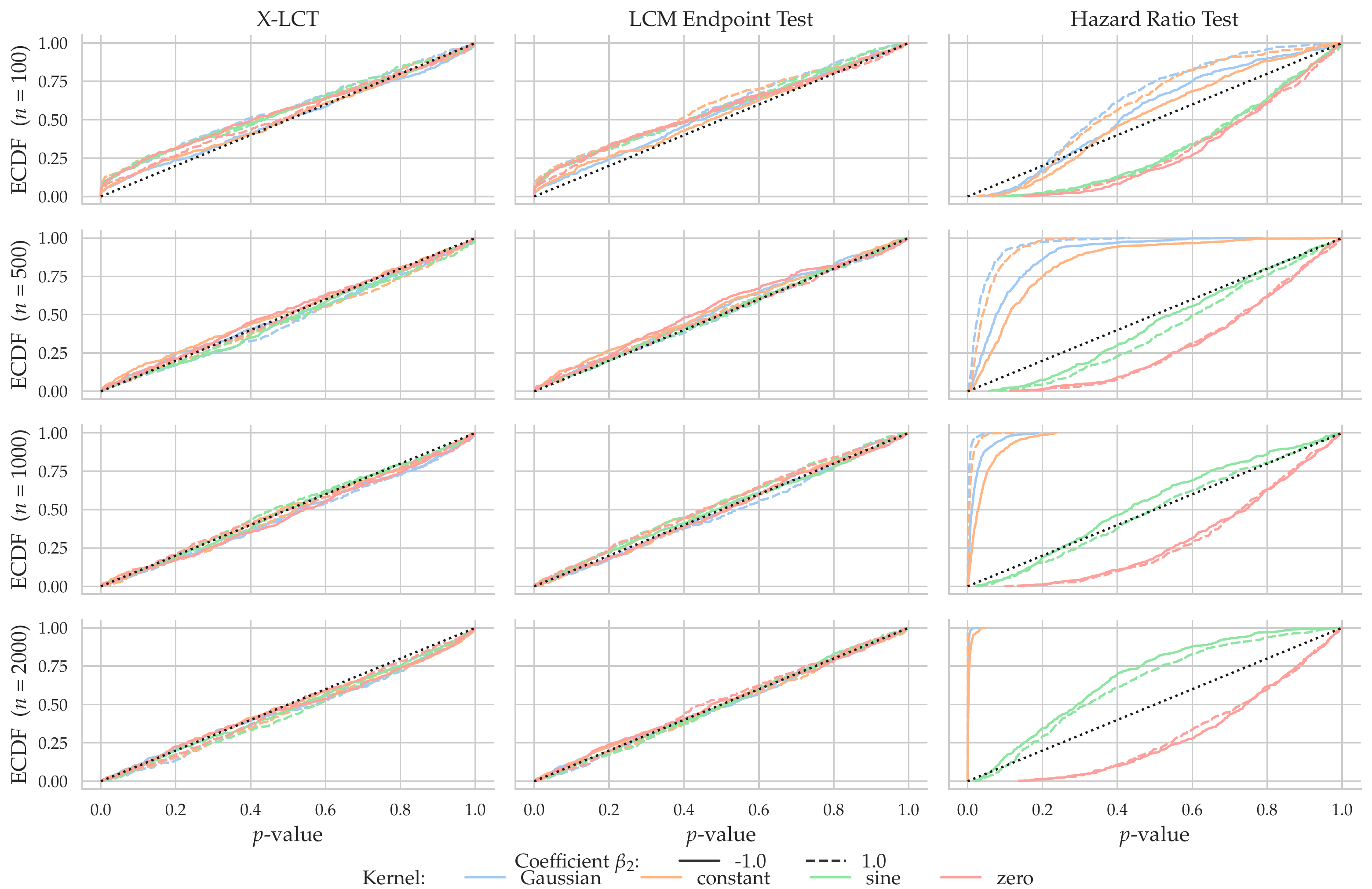}
    \caption{Empirical distribution functions of $p$-values for the three
    different conditional local independence tests considered, simulated under
    the sampling scheme described in \cref{sec:simulations}. The dotted line
    shows $y=x$ corresponding to a uniform distribution.}\label{fig:H0pvals_full}
\end{figure}

\begin{figure}[hb]
    \includegraphics[width = .7\linewidth]{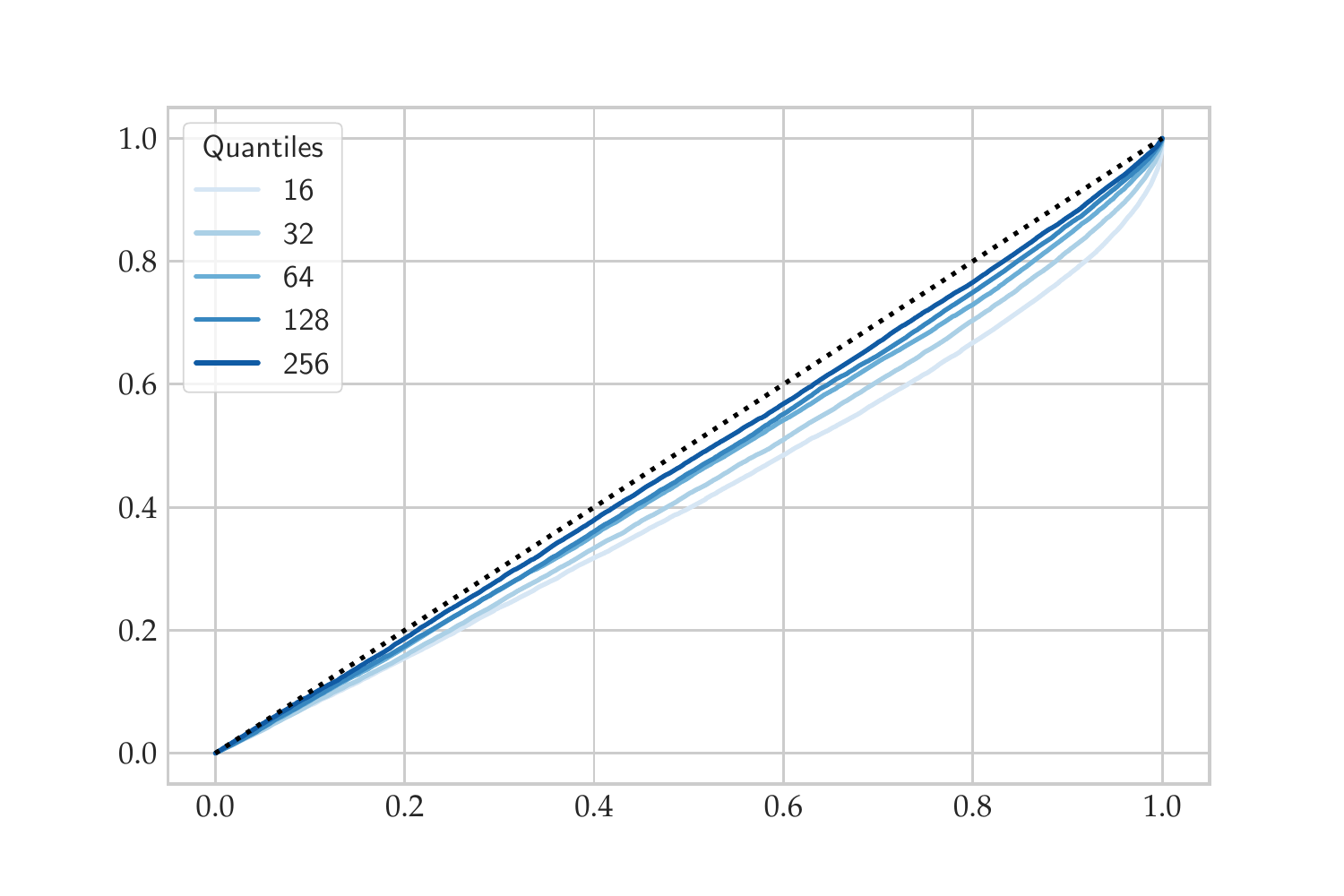}
    \caption{Empirical distribution functions of $p^{(q)} = 1- F_S(M^{(q)})$,
    where $M^{(q)} = (M_t^{(q)})_{t=1,\ldots,q}$ is a random walk with Gaussian
    increments such that $M_q^{(q)}$ has unit variance for each
    $q\in \{2^\ell \colon \ell=4,\ldots,8\}$. 
    Each empirical distribution function is based on $N=20\, 000$ samples.}
    \label{fig:supremumapprox}
\end{figure}

\end{document}